\documentclass[a4paper]{article}
\usepackage{amssymb}
\usepackage{amsmath}
\usepackage{bm}
\usepackage{verbatim}
\usepackage{algorithm,algpseudocode}
\usepackage{amsthm}
\usepackage{graphicx}
\usepackage{subfigure}
\usepackage{float}
\usepackage{tikz}
\usepackage[numbers]{natbib}

\usepackage{xcolor}
\usepackage{verbatim}
\usepackage{CJK}
\usepackage{mathrsfs}
\usepackage{hyperref}
\usepackage{booktabs}
\usepackage{multirow}
\usepackage{makecell}
\usepackage[nohead,margin=1.0in]{geometry}
\usepackage[title]{appendix}
\usepackage{authblk}

\newtheorem{theorem}{Theorem}[section]

\newtheorem{lemma}[theorem]{Lemma}
\newtheorem*{remark}{Remark}
\newtheorem{assumption}{Assumption}
\newcommand\keywords[1]{\textbf{Keywords}: #1}
\usepackage{url}

\begin{document}

\title{DiLO: Decoupling Generative Priors and Neural Operators via Diffusion Latent Optimization for Inverse Problems}

\author[1]{Haibo Liu}
\author[2]{Guang Lin}

\affil[1]{Department of Computer Science, The University of Texas at Austin, 2317 Speedway, Austin, TX 78712, USA}
\affil[2]{Department of Mathematics, School of Mechanical Engineering, Purdue University, 610 Purdue Mall, West Lafayette, IN 47907, USA}

\affil[ ]{\texttt{haiboliu@utexas.edu} \quad \texttt{guanglin@purdue.edu}}

\maketitle

\begin{abstract}

Diffusion models have emerged as powerful generative priors for solving PDE-constrained inverse problems. Compared with end-to-end approaches that learn inverse mappings from large paired observation--parameter datasets, decoupling the parameter prior from the forward physical model reduces reliance on paired data and improves flexibility. Fast neural operators can accelerate inference in such decoupled frameworks, but standard diffusion guidance evaluates them on partially denoised, non-physical intermediate states, introducing out-of-distribution (OOD) inputs. We therefore formulate the Manifold Consistency Requirement, under which the physical surrogate is evaluated only on fully denoised parameters. Based on this principle, we present Diffusion Latent Optimization (DiLO), which reformulates inference as deterministic optimization of the initial latent state and backpropagates measurement gradients through the full denoising trajectory. This avoids surrogate evaluations on partially denoised intermediate states while enabling accurate reconstruction. We establish asymptotic stationarity under bounded surrogate-gradient error. Experiments on Electrical Impedance Tomography, Inverse Scattering, and Inverse Navier-Stokes demonstrate DiLO's accuracy and robustness to noise. DiLO's code is available at \url{https://github.com/haibo-research/DiLO}.

\end{abstract}

\vspace{0.5em}
\noindent\keywords{Diffusion Models; Inverse Problems; Neural Operators; Latent Optimization; PDE-Constrained Inverse Problems; Manifold Consistency}

\vspace{1em}

\section{Introduction}

Inverse problems aim at recovering hidden physical properties from indirect observations. For example, Electrical Impedance Tomography (EIT) tries to reconstruct the internal conductivity distribution of a patient from electrical signals received from the skin \cite{borcea2002electrical, holder2004electrical}. Inverse Scattering Problem, crucial for seismic exploration and non-destructive testing, aims at reconstructing the refractive index from far-field or near-field data \cite{colton2019inverse, van2012extended}. Furthermore, inverse problems in fluid dynamics may seek to recover the initial state from later-time observations in a Navier--Stokes (N-S) system~\cite{raissi2019physics}. Despite their broad application, these problems remain challenging due to highly nonlinear physical models and the ill-posedness caused by complex underlying parameter distributions.

Traditionally, adjoint-based iterative methods were developed to solve problems involving PDE constraints, combined with various regularization techniques to address the severe ill-posedness. However, these methods suffer from high computational costs, and artificially imposed regularization terms limit their application scenarios while increasing the difficulty of parameter tuning \cite{vogel2002computational, bakushinsky2004iterative}. To bridge this gap, deep learning has emerged as a powerful alternative for data-driven reconstruction. In particular, generative models trained on underlying data distributions provide strong priors for solving inverse problems \cite{chen2024solving, denker2025deep, chen2025implicit}.

Recently, diffusion models \cite{ho2020denoising, song2020score} have proven to be highly effective generative priors, setting new benchmarks in inverse problems through unsupervised score-based sampling methods \cite{song2022solving}, Diffusion Posterior Sampling (DPS) \cite{chung2022diffusion}, and variational regularization techniques like RED-diff \cite{mardani2023variational}. Recent benchmarking efforts, such as InverseBench \cite{zheng2025inversebench}, further demonstrate the efficacy of these models. Advanced optimization and consistency techniques, such as exploiting the inherent denoising capabilities of diffusion priors via projection gradient descent (ProjDiff) \cite{zhang2024unleashing} or accelerating data consistency steps using deep learning (DDC) \cite{chen2024deep}, have been developed to improve reconstruction speed and robustness. However, most diffusion-based solvers are designed for linear or mildly nonlinear operators \cite{jiang2025ode}. For complex PDE-constrained problems, the computational burden of iteratively evaluating the forward model during the reverse diffusion process is prohibitive. 

To address these challenges, recent research has explored the integration of generative modeling and operator learning. Early hybrid approaches attempted to formulate diffusion models as probabilistic neural operators \cite{haitsiukevich2024diffusion}, physics-aware generative operators for discontinuous inputs (e.g., DGNO) \cite{zang2025dgno}, or explicitly embed discretized physical constraints into score-based models like CoCoGen \cite{jacobsen2024cocogen} and continuous flow matching models like PCFM \cite{utkarsh2024physics}. Alongside generative models, advancements in operator learning have directly targeted inverse problems through novel architectures like Invertible Fourier Neural Operators (iFNO) \cite{long2024invertible} and Physics-Informed Deep Inverse Operator Networks (PI-DIONs) \cite{cho2025physics}, which learn the solution operator without massive labeled datasets. For problems with partial observations, frameworks like DiffusionPDE \cite{huang2024diffusionpde} proposed modeling the joint distribution of PDE coefficients and solutions. While effective at handling arbitrary sparsity, such joint-embedding methods carry a heavy computational burden: they require training a task-specific diffusion model from scratch on massive datasets of paired observations for every new physical system.

Consequently, the focus has shifted toward explicitly decoupling the generative prior from the physical constraints, a paradigm central to plug-and-play priors. Methods like Function-space Diffusion Posterior Sampling (FunDPS) \cite{Yao2025GuidedDS} and the Decoupled Diffusion Inverse Solver (DDIS) \cite{lin2026decoupled} successfully mitigate the guidance attenuation inherent to joint-embedding models by applying plug-and-play guidance or decoupled noise annealing to functional priors. However, evaluating the physical surrogate within these iterative sampling loops requires feeding it partially denoised approximations—typically posterior means derived via Tweedie's estimate ($\hat{x}_0(x_t)$). Because physical surrogates are trained exclusively on clean, fully-resolved data manifolds with sharp high-frequency details, feeding them these blurry, non-physical intermediate states forces the operators into OOD regimes, leading to inaccurate guidance and gradient instability. Furthermore, directly modeling diffusion in the ambient function space inherently limits their capacity to capture highly complex underlying parameter distributions.

Building upon these insights, we introduce \textbf{Diffusion Latent Optimization (DiLO)}, a novel framework designed to address the bottlenecks of plug-and-play decoupled diffusion inverse solvers. First, it operates within a compressed latent space to strictly decouple the diffusion prior from the physical model, thereby bypassing dimensional limitations and reducing reliance on large-scale paired inverse-training data. Second, to avoid the distribution shift associated with partially denoised intermediate states, we evaluate the physical surrogate exclusively on fully denoised parameters—a principle we define as the Manifold Consistency Requirement. DiLO inherently satisfies this requirement by breaking away from stochastic sampling and reformulating the inverse problem as a deterministic optimization over the initial latent noise. This deterministic trajectory ensures the surrogate is evaluated strictly on fully denoised parameters, naturally bypassing the guidance instability that hinders existing methods. With theoretical guarantees of asymptotic stationarity under bounded surrogate-gradient error, DiLO achieves excellent reconstruction accuracy and exceptional noise robustness across three challenging physical domains: EIT, Inverse Scattering, and Inverse N-S. Table \ref{tab:comparison_methods} presents a systematic comparison between DiLO and representative PDE-constrained solvers, highlighting fundamental differences in inference paradigms, prior spaces, data requirements, and physics evaluation mechanisms.

\begin{table}[htbp]
    \centering
    \caption{Comparison of DiLO with existing PDE-constrained inverse solvers. 
    Paired data refer to parameter--observation pairs; Inv. Train. and Fwd. Train. denote their use for direct inverse-model training and forward-surrogate training, respectively.
    Abbreviations: NO = Neural Operator, Gen. = Generative, Opt. = Optimization, 
    Inv. = Inverse, Fwd. = Forward, OOD = Out-of-Distribution.}
    \label{tab:comparison_methods}

    \resizebox{\textwidth}{!}{%
    \begin{tabular}{lccccc}
        \toprule
        \textbf{Property} 
        & \textbf{PINN} 
        & \textbf{Direct NO} 
        & \textbf{DiffusionPDE} 
        & \textbf{DPS/FunDPS} 
        & \textbf{DiLO (Ours)} \\
        \midrule

        \textbf{Paradigm} 
        & Physics Opt. 
        & End-to-End 
        & Joint Gen. 
        & Decoupled 
        & \textbf{Decoupled} \\

        \textbf{Prior Space} 
        & None 
        & None 
        & Function Space 
        & Ambient/Func. 
        & \textbf{Latent Space} \\

        \textbf{Paired Data: Inv. Train.} 
        & No 
        & Yes 
        & Joint 
        & No 
        & \textbf{No} \\

        \textbf{Paired Data: Fwd. Train.} 
        & No (PDE) 
        & N/A 
        & Joint 
        & Method-dependent 
        & \textbf{Yes (surrogate)} \\

        \textbf{Retraining$^\dagger$} 
        & Yes 
        & Yes 
        & Yes 
        & Method-dependent 
        & \textbf{Surrogate only} \\

        \textbf{Physics Eval.} 
        & Clean 
        & N/A 
        & N/A (Joint) 
        & Noisy (OOD Risk) 
        & \textbf{Fully denoised} \\

        \textbf{Trajectory} 
        & Deterministic 
        & Feed-forward 
        & Stochastic 
        & Stochastic 
        & \textbf{Det. Opt.} \\

        \textbf{Physics Cost} 
        & High (PDE) 
        & Low (Surrogate) 
        & N/A 
        & Varies 
        & \textbf{Low (Surrogate)} \\

        \bottomrule
        \multicolumn{6}{l}{
        \footnotesize $^\dagger$ The DiLO prior is reusable; its forward surrogate may require retraining or fine-tuning when the observation operator changes.
        }
    \end{tabular}%
    }
\end{table}
This paper is structured as follows. Section \ref{Background} introduces the mathematical framework for the target inverse problems and reviews diffusion models. Section \ref{method} details the proposed DiLO framework and provides a rigorous theoretical convergence analysis. Section \ref{Experiment} evaluates DiLO's accuracy and extreme noise robustness against baseline methods. Finally, Section \ref{Discussion} concludes the paper and discusses future research directions.

\section{Background.}\label{Background}

\subsection{Mathematical Framework.}

Let $\Omega \subset \mathbb{R}^d$ be a bounded domain with a smooth boundary $\partial\Omega$. We define the computational domain as $X = \Omega$ (stationary) or $X = \Omega \times (0, T)$ (time-dependent), with boundary $\partial X$. Let $a \in \mathcal{A}$ be an unknown physical parameter (e.g., a coefficient field) in an admissible function space. 

The forward PDE is formulated as:
\begin{equation} \label{eq:abstract_pde}
    \mathcal{L}_a u = f \quad \text{in } X, \qquad \mathcal{B} u = g \quad \text{on } \partial X,
\end{equation}
where $\mathcal{L}_a$ and $\mathcal{B}$ are the parameter-dependent differential operator and the boundary operator, respectively. The state variable $u \in \mathcal{U}$, source term $f \in \mathcal{F}$, and boundary condition $g \in \mathcal{G}$ belong to appropriate Banach or Hilbert spaces.

Solving the forward problem yields $u$ given $a$, $f$, and $g$. In practice, the full internal field $u$ is rarely accessible. Instead, we typically rely on boundary measurements, defined by an observation operator:
\begin{equation} \label{eq:observation_op}
    \Lambda_a : \mathcal{G} \to \mathcal{H}, \quad g \mapsto \Lambda_a(g) = h(u)|_{\partial X},
\end{equation}
which maps the boundary input $g$ to a measurement $h(u)$ in space $\mathcal{H}$. For a fixed $f$, this induces the parameter-to-observation map (or forward operator) $\mathcal{F}$:
\begin{equation} \label{eq:forward_map}
    \mathcal{F} : \mathcal{A} \to \mathcal{L}(\mathcal{G}, \mathcal{H}), \quad a \mapsto \mathcal{F}(a) = \Lambda_a,
\end{equation}
where $\mathcal{L}(\mathcal{G}, \mathcal{H})$ is the space of bounded linear operators from $\mathcal{G}$ to $\mathcal{H}$.

The inverse problem aims to reconstruct $a \in \mathcal{A}$ from measurements $\Lambda_a$. The inverse map is defined as:
\begin{equation} \label{eq:inverse_map}
    \mathcal{F}^{-1} : \text{Ran}(\mathcal{F}) \subset \mathcal{L}(\mathcal{G}, \mathcal{H}) \to \mathcal{A}, \quad \Lambda_a \mapsto a = \mathcal{F}^{-1}(\Lambda_a).
\end{equation}
While recovering $a$ from a single measurement is generally ill posed, utilizing the full observation operator $\Lambda_a$ often restores uniqueness. Theoretical guarantees for inverse problems focus on the uniqueness and stability of $\mathcal{F}^{-1}$. We recommend \cite{colton1990inverse} for a systematic introduction.

Following this abstract mathematical framework, we provide three concrete examples of PDE-constrained inverse problems to which this framework is applicable.

\subsection{EIT.}

EIT is a non-invasive modality that reconstructs the internal conductivity $\sigma$ of a medium from boundary electrical measurements. This technique is highly valued for real time applications such as lung ventilation monitoring and stroke detection due to its portability and lack of ionizing radiation. Based on the quasi-static approximation of Maxwell's equations where displacement currents are neglected, the electric potential $u$ satisfies\cite{borcea2002electrical}:
\begin{equation}
\nabla \cdot (\sigma(\mathbf{x}) \nabla u(\mathbf{x})) = 0. \quad \mathbf{x} \in \Omega \subset \mathbb{R}^2
\end{equation}

The forward problem determines the potential $u$ given a conductivity $\sigma$ and applied current density $j$ via the Neumann boundary condition $\sigma \partial u / \partial \mathbf{n} = j$ on $\partial\Omega$. In practice, measurements are acquired through $M$ discrete electrode patterns, yielding voltage data $V^k = u|_{\partial\Omega}^{(k)}$ for $k = 1, \ldots, M$. We define the forward mapping as $V = \mathcal{G}[\sigma]$, where the numerical solution is typically obtained via the Finite Element Method (FEM).

The inverse problem seeks to invert this mapping to recover the unknown conductivity:
\begin{equation}\label{eit_pro}
\sigma^* = \mathcal{G}^{-1}(V).
\end{equation}
This is formulated as a PDE-constrained optimization task to minimize the discrepancy between observed and predicted voltages:
\begin{equation}
\sigma^* = \arg\min_{\sigma \in \mathcal{A}} \sum_{k=1}^M \|V^k - (\mathcal{G}[\sigma])^k\|^2,
\end{equation}
where $\mathcal{A}$ is the set of admissible conductivities satisfying physical bounds $\sigma_{\min} \leq \sigma \leq \sigma_{\max}$. 

The EIT inverse problem is notoriously ill posed in the sense of Hadamard. This manifests as an exponential decay of singular values in the linearized forward operator, meaning high frequency spatial components are severely attenuated at the boundary. Consequently, reconstructions are highly sensitive to noise.

\subsection{Inverse Scattering.}

We consider two dimensional time harmonic acoustic scattering governed by the Helmholtz equation. This model represents the temporal Fourier transform of the constant density wave equation\cite{colton2019inverse}:
\begin{equation}
\Delta u(\mathbf{x}) + \omega^2 n(\mathbf{x}) u(\mathbf{x}) = 0, \quad \mathbf{x} \in \Omega \subset \mathbb{R}^2
\end{equation}
where $u$ is the total wave field, $\omega$ the angular frequency, and $n(\mathbf{x})$ the refractive index. Defining the perturbation $\eta(\mathbf{x}) = n(\mathbf{x}) - 1$, we assume $ \operatorname{supp}(\eta) \subset \Omega$ within a homogeneous background.

The forward problem determines the scattered field $u^{\text{sc}} = u - u^{\text{in}}$ generated by an incident plane wave $u^{\text{in}} = e^{i\omega s \cdot \mathbf{x}}$ with direction $s \in \mathbb{S}^1$. The scattered field satisfies:
\begin{equation}\label{helmholtz}
\begin{cases}
\Delta u^{\text{sc}} + \omega^2 (1 + \eta) u^{\text{sc}} = -\omega^2 \eta u^{\text{in}} \\
\lim_{|\mathbf{x}| \to \infty} \sqrt{|\mathbf{x}|} \left( \frac{\partial u^{\text{sc}}}{\partial |\mathbf{x}|} - i\omega u^{\text{sc}} \right) = 0
\end{cases}
\end{equation}
where the second line is the Sommerfeld radiation condition ensuring uniqueness. Measurements are acquired on a circle $D$ of radius $R$ enclosing $\Omega$. For receivers at $r \in \mathbb{S}^1$, the scattering data is $\Lambda^{\omega}(r, s) = u^{\text{sc}}(Rr; s)$, defining the forward map $\Lambda^{\omega} = \mathcal{F}^{\omega}[\eta]$.

The inverse problem seeks to reconstruct the perturbation $\eta$ from scattering data across a set of frequencies $\Omega_{freq}$:
\begin{equation}\label{inverse_scattering}
\eta^* = \mathcal{F}^{-1}(\{\Lambda^{\omega}\}_{\omega \in \Omega_{freq}}).
\end{equation}
This is formulated as a PDE-constrained optimization task minimizing the $L^2$ data misfit:
\begin{equation}
\eta^* = \arg\min_{\eta} \sum_{\omega \in \Omega_{freq}} \int_{[0,2\pi]^2} |\Lambda^{\omega}(r, s) - (\mathcal{F}^{\omega}\eta)(r, s)|^2 \, dr \, ds.
\end{equation}

\subsection{Inverse N-S.}
We focus on recovering the initial state of a 2D incompressible fluid from later observations. Under the velocity-vorticity formulation, the Navier-Stokes equations are defined by the scalar vorticity $w(\mathbf{x}, t) = \nabla \times \mathbf{u}$, where $\mathbf{u}$ is the velocity field. Given an initial vorticity distribution $w(\mathbf{x}, 0) = w_0(\mathbf{x})$, the forward problem determines the state of the fluid $w(\mathbf{x}, T)$ at a target time $T$:
\begin{align}\label{ns_forward}
    \partial_t w(\mathbf{x},t) + \mathbf{u}(\mathbf{x},t) \cdot \nabla w(\mathbf{x},t) &= \nu \Delta w(\mathbf{x},t) + f(\mathbf{x}), & \mathbf{x} &\in (0, 2\pi)^2, t \in (0,T] \nonumber \\
    \nabla \cdot \mathbf{u}(\mathbf{x},t) &= 0, & \mathbf{x} &\in (0, 2\pi)^2, t \in [0,T] \\
    w(\mathbf{x},0) &= w_0(\mathbf{x}), & \mathbf{x} &\in (0, 2\pi)^2 \nonumber
\end{align}
where $\nu$ denotes the kinematic viscosity and $f$ represents external forcing.

We denote the forward evolution operator as $\mathcal{S}_T$, such that $w_T = \mathcal{S}_T[w_0]$. The mapping $\mathcal{S}_T$ is highly nonlinear and exhibits sensitive dependence on initial conditions. The inverse N-S problem seeks to reconstruct the initial vorticity field $w_0$ from observations of the vorticity $w_{\text{obs}}$ acquired at time $t=T$. Mathematically, this entails solving:
\begin{equation} \label{ns_inverse}
w_0^* = \mathcal{S}_T^{-1}(w_{\text{obs}}).
\end{equation}
This is formulated as a PDE-constrained optimization problem aimed at minimizing the discrepancy between the evolved state and the observed data:
\begin{equation}
w_0^* = \arg\min_{w_0} \frac{1}{2} \|\mathcal{S}_T[w_0] - w_{\text{obs}}\|^2.
\end{equation}

The inverse N-S problem is inherently ill posed as the diffusion term $\nu \Delta w$ dissipates high-frequency information, resulting in high sensitivity to noise. Furthermore, nonlinear convection induces a non-convex optimization landscape prone to converging to local minima.

\subsection{Denoising Diffusion Probabilistic Models.}
Given a dataset sampled from an unknown distribution $p_{\text{data}}(x)$, generative models aim to draw samples from this distribution. Diffusion models define a forward stochastic differential equation (SDE) that diffuses the data distribution $p_{\text{data}}(x)$ into a known prior distribution $\pi(x)$ (typically a Gaussian) over a time horizon $t \in [0,T]$:
\begin{equation}
\label{eq:general_sde}
\mathrm{d}x_t = f(t)x_t \mathrm{d}t + g(t)\mathrm{d}w_t,
\end{equation}
where $f(t)$ and $g(t)$ are drift and diffusion coefficients, and $w_t$ is a standard Wiener process. The marginal distribution $p_t(x)$ transitions from the data distribution $p_0(x) = p_{\text{data}}(x)$ to the noise distribution $p_T(x) \approx \mathcal{N}(0, \mathbf{I})$.

The generative process corresponds to the time reversal of ~\eqref{eq:general_sde}. The reverse-time SDE is given by:
\begin{equation}
\label{eq:reverse_sde}
\mathrm{d}x_t = \left[ f(t)x_t - g(t)^2 \nabla_{x_t} \log p_t(x_t) \right] \mathrm{d}t + g(t)\mathrm{d}\bar{w}_t,
\end{equation}
where $\bar{w}_t$ is a standard Wiener process in the reverse time. The critical term here is the score function $\nabla_{x_t} \log p_t(x_t)$, which points towards high density regions of the data distribution at each noise level.

Since the true score function is unknown, it is estimated by a time dependent neural network $s_\theta(x_t, t)$. This network is trained via Denoising Score Matching to minimize:
\begin{equation}
    \mathcal{L}(\theta) = \mathbb{E}_{t, x_0, x_t} \left[ \left\| s_\theta(x_t, t) - \nabla_{x_t} \log p_{t|0}(x_t|x_0) \right\|^2 \right],
\end{equation}
where $t$ is uniformly sampled from $[0, T]$ and the expectation is taken over $t$, $x_t \sim p(x_t|x_0)$, and $x_0 \sim p_{\text{data}}(x)$.

\paragraph{Denoising Diffusion Implicit Models}
For sampling, we use the noise-prediction parameterization of DDIM. Let $\epsilon_\theta(x_t,t)$ predict the Gaussian noise in $x_t=\sqrt{\bar\alpha_t}x_0+\sqrt{1-\bar\alpha_t}\epsilon$. Its relation to the score parameterization is
$s_\theta(x_t,t)=-\epsilon_\theta(x_t,t)/\sqrt{1-\bar\alpha_t}$. DDIM constructs a non-Markovian diffusion process that permits accelerated sampling, with update
\begin{equation}
    x_{t-1} = \sqrt{\bar{\alpha}_{t-1}}\hat{x}_0(x_t) + \sqrt{1 - \bar{\alpha}_{t-1} - \eta^2\delta_t^2}\epsilon_\theta(x_t, t) + \eta\delta_t\xi, \quad t = T, \dots, 1,
\end{equation}
where $\alpha_t = 1 - \beta_t$, $\bar{\alpha}_t = \prod_{i=1}^t \alpha_i$, $\xi \sim \mathcal{N}(0, \mathbf{I})$, $\eta$ is the temperature parameter, and $\delta_t$ controls the stochasticity of the update. Under the epsilon-prediction convention, the clean estimate is
\begin{equation}\label{Tweedie_estimat}
    \hat{x}_0(x_t) = \frac{1}{\sqrt{\bar{\alpha}_t}}\left(x_t - \sqrt{1 - \bar{\alpha}_t}\epsilon_\theta(x_t, t)\right).
\end{equation}
This convention is used consistently in the remainder of the paper and in Algorithm~\ref{alg:DiLO}.

\paragraph{Solving Inverse Problems with Diffusion Models}
This work focuses on the inference of physical parameters $x$ (e.g., conductivity, permeability) from sparse or noisy observations $y$. The inverse problem could be modeled as sampling from the posterior distribution $p(x | y)$. Analogous to the unconditional sampling process by solving the reverse time SDE \eqref{eq:reverse_sde}, we can sample from the posterior distribution $p(x|y)$ by solving the following conditional reverse time SDE~\cite{song2020score}:
\begin{equation}
\label{eq:conditional_sde}
\mathrm{d}x_t = \big[ f(t)x_t - g(t)^2 \nabla_{x_t} \log p_t(x_t | y) \big] \mathrm{d}t + g(t)\mathrm{d}\bar{w}_t, \quad t \in [0,1].
\end{equation}

The core challenge in diffusion-based inverse solvers lies in integrating physical measurements into this generative process. While conditional score models \cite{song2020score} often require extensive paired datasets for supervised learning, unsupervised test-time guidance methods, such as Diffusion Posterior Sampling (DPS) \cite{chung2022diffusion} and RED-diff \cite{mardani2023variational}, approximate conditional inference using unconditionally-trained priors. For complex PDE-constrained problems, the computational burden of iteratively evaluating the true forward physical model during the reverse diffusion process is prohibitive. To mitigate this, recent methods like FunDPS \cite{Yao2025GuidedDS} and DDIS \cite{lin2026decoupled} have shifted toward explicitly decoupling the generative prior from physical constraints, often employing neural operators as fast surrogate solvers. 

Despite these structural advancements, existing decoupled methods still face fundamental limitations that hinder their application to highly nonlinear physical systems. A critical bottleneck in these iterative sampling loops is the reliance on Tweedie's estimate, $\hat{x}_0(x_t)$, to approximate the clean state for physics evaluation. Crucially, under an ideal score model, $\hat{x}_0(x_t)$ corresponds to the posterior mean $\mathbb{E}[x_0|x_t]$ at early and intermediate timesteps---a blurry, smoothed-out approximation lacking the high-frequency structural details of a true physical state. While mathematically viable for simple linear masking (as in joint-embedding models), using these partially denoised, low-fidelity predictions as inputs for highly nonlinear PDE surrogates can introduce a substantial distribution shift. This exposes the surrogate models to inputs that differ from their training distribution, potentially leading to inaccurate physics evaluations and unstable gradients. To address these algorithmic mismatches, we introduce the DiLO framework.

\section{Diffusion Latent Optimization (DiLO)}\label{method}
As illustrated in Figure \ref{fig:pipeline_overview}, DiLO decouples generative priors from physical constraints through a two phase framework. \textbf{A. Offline Training} establishes the core models: (i) a Latent Diffusion Model (LDM) trained in a compressed space $\mathcal{Z}$ via an Encoder ($\mathcal{E}$) and Decoder ($\mathcal{D}$) to capture the prior distribution of physical parameters $a \in \mathcal{A}$, and (ii) a fast Differentiable Surrogate $\tilde{\mathcal{F}}_\phi$ pre-trained on paired data to replace traditional numerical PDE solvers. Building upon these components, \textbf{B. Online Inference} reformulates the inverse problem as a strictly deterministic optimization over the initial latent noise $z_T$. By projecting $z_T$ through a variance free denoising trajectory, we obtain a fully denoised latent state $z_0$. Crucially, the surrogate $\tilde{\mathcal{F}}_\phi$ evaluates the physical loss exclusively on the fully denoised reconstructed parameter field $\hat{a}=\mathcal{D}(z_0)$, avoiding surrogate evaluations on partially denoised intermediate states. Precise measurement gradients are then backpropagated through $\mathcal{D}$ and the deterministic sampling pathway via Solver-in-the-Loop Backpropagation. This stable optimization directly updates $z_T$, guiding the optimization along the learned physical manifold to recover a physically admissible reconstructed field $a^*$.

\begin{figure}
    \centering
    \includegraphics[width=\textwidth]{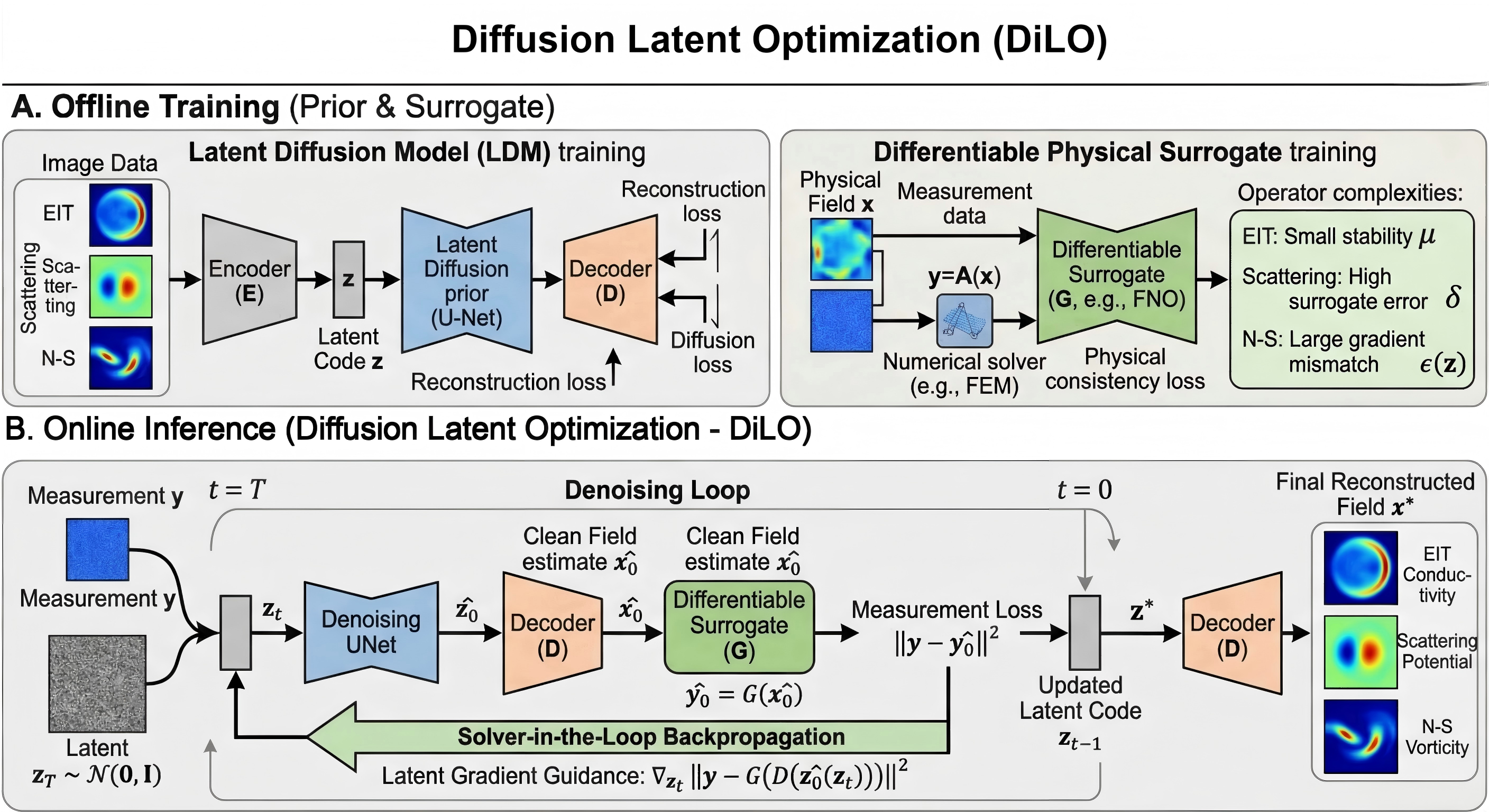} 
    \caption{The DiLO framework. (A) \textbf{Offline Training}: Pre-training the Latent Diffusion Model (LDM) to serve as a plug-and-play prior, alongside a Differentiable Surrogate (e.g., FNO) acting as a forward physics operator. (B) \textbf{Online Inference}: Solving the inverse problem via deterministic latent optimization, where precise gradients from the measurement loss are backpropagated to the initial noise $z_T$.}
    \label{fig:pipeline_overview}
\end{figure}

\subsection{Offline Training}
\paragraph{Efficient Inference via Compressed Diffusion}
Direct optimization in the original high dimensional parameter space $\mathcal{A}$ is often ill-posed and prone to converging to unphysical local minima. To enable efficient and stable inference for high dimensional physical parameters, we perform the diffusion process in a compressed representation space. Operating in a compressed space inherently constrains the search process to a valid, low dimensional physical manifold.

We define the inference trajectory in a lower dimensional space $\mathcal{Z}$. To this end, we first pre-train an autoencoder consisting of an encoder $\mathcal{E}$ and a decoder $\mathcal{D}$ by minimizing the reconstruction error:
\begin{equation}
    \mathcal{L}_{\text{ae}} = \mathbb{E}_{a \sim p_{\text{data}}(a)} \left[ \| a - \mathcal{D}(\mathcal{E}(a)) \|_{\mathcal{A}}^2 \right],
\end{equation}
where $\mathcal{E}: \mathcal{A} \to \mathcal{Z}$ maps the high-dimensional physical parameters to a latent representation $z$, and $\mathcal{D}: \mathcal{Z} \to \mathcal{A}$ performs the inverse reconstruction. For a pixel-space diffusion prior, the same formulation applies by taking $\mathcal{Z}=\mathcal{A}$ and $\mathcal{E}=\mathcal{D}=\mathrm{Id}$. In this case, the optimized variable is the initial diffusion state itself.

\paragraph{Latent Diffusion Model}
Once the compressed latent space is established, we apply the diffusion framework introduced previously to learn the prior distribution of the latent codes $z = \mathcal{E}(a)$. We use a noise-prediction network $\epsilon_\theta(z_t,t)$ throughout; it is related to the score by $s_\theta(z_t,t)=-\epsilon_\theta(z_t,t)/\sqrt{1-\bar\alpha_t}$. The network is trained using the latent epsilon-matching objective:

\begin{equation}
    \mathcal{L}_{\text{diff}} = \mathbb{E}_{z \sim \mathcal{E}(a), \epsilon \sim \mathcal{N}(0, \mathbf{I}), t} \left[ \| \epsilon - \epsilon_\theta(z_t, t) \|^2 \right],
\end{equation}
where $z_t = \sqrt{\bar{\alpha}_t} z + \sqrt{1 - \bar{\alpha}_t} \epsilon$ represents the noisy latent state at timestep $t$. This pre-trained network $\epsilon_\theta(z_t, t)$ now serves as a powerful, low dimensional generative prior for our subsequent optimization.

\paragraph{Neural Operator as Differentiable Forward Physics}
To bridge the gap between the generative prior and the observed measurements $y \in \mathcal{H}$, we utilize a Neural Operator $\tilde{\mathcal{F}}_\phi$ (e.g., FNO or DeepONet), which provides a fast, differentiable mapping from the parameter space $\mathcal{A}$ to the observation space $\mathcal{H}$. The surrogate operator is pre-trained on paired forward samples $(a,y)$ generated by the true PDE map $\mathcal{F}$:
\begin{equation}
    \mathcal{L}_{\text{op}} = \mathbb{E}_{(a, y)} \left[ \| y - \tilde{\mathcal{F}}_\phi(a) \|_{\mathcal{H}}^2 \right].
\end{equation}

\paragraph{The Bottleneck of Standard Diffusion Guidance}
With both models pre-trained, a natural approach to solve the inverse problem is to enforce measurement consistency during the generative process. This is typically achieved by utilizing the differentiable surrogate $\tilde{\mathcal{F}}_\phi$ to guide the reverse diffusion trajectory step-by-step. Specifically, at each noise level $t$, the likelihood gradient is approximated via Tweedie's estimate $\hat{a}_0(a_t)$:
\begin{equation} \label{eq:flawed_guidance}
    \nabla_{a_t} \log p(y | a_t) \approx -\gamma \nabla_{a_t} \| y - \tilde{\mathcal{F}}_\phi(\hat{a}_0(a_t)) \|_{\mathcal{H}}^2,
\end{equation}
where $y$ is the observation and $\gamma$ scales the guidance strength. However, this direct integration exposes a critical theoretical bottleneck. During the early and middle stages of diffusion, under an ideal score model, the intermediate estimates of the physical state, $\hat{a}_0(a_t)$, correspond to the posterior means $\mathbb{E}[a_0|a_t]$. These averaged states inherently exhibit smoothing artifacts and lack the sharp, high-fidelity structures characteristic of true physical fields.

Conversely, physical surrogates like FNO are trained exclusively on $\mathcal{A}_{\text{valid}}$, the manifold of fully-resolved parameter fields governed by precise PDE dynamics. Feeding these blurry, non-physical conditional mean states into a highly nonlinear surrogate creates a fundamental physics mismatch. Due to the intrinsic nonlinearity of the PDE forward map $\mathcal F$, evaluating the forward model at a conditional mean generally differs from averaging the corresponding forward responses; in general,
\[
\mathcal F(\mathbb E[a_0\mid a_t]) \neq \mathbb E[\mathcal F(a_0)\mid a_t].
\] 
This structural mismatch forces the neural operator into OOD regimes, yielding uninformative loss evaluations and gradient instabilities.

To systematically resolve this bottleneck, we formalize the \textbf{Manifold Consistency Requirement}: \textit{the physical surrogate should be evaluated on fully denoised diffusion outputs rather than partially denoised intermediate estimates.} Recognizing that intermediate denoising estimates can induce a distribution shift relative to the surrogate training data, we introduce the DiLO framework to fundamentally reformulate the inference trajectory.

\subsection{Diffusion Latent Optimization}
To solve the inverse problem $y_{\text{obs}} \approx \tilde{\mathcal{F}}_\phi(a)$ while strictly satisfying the Manifold Consistency Requirement, we must treat the generative process as an end-to-end differentiable computational graph and directly optimize the initial latent state $z_T$.

\paragraph{Deterministic Sampling Trajectory}
To make this end-to-end optimization mathematically tractable and numerically stable, DiLO completely eliminates the inherent randomness of standard diffusion sampling. We adopt the deterministic Denoising Diffusion Implicit Model (DDIM) by setting the temperature parameter $\eta = 0$. The reverse trajectory in the latent space $\mathcal{Z}$ thus becomes a deterministic discrete DDIM trajectory:
\begin{equation}\label{deterministic_ode}
    z_{t-1} = \sqrt{\bar{\alpha}_{t-1}}\hat{z}_0(z_t) + \sqrt{1 - \bar{\alpha}_{t-1}} \epsilon_\theta(z_t, t),
\end{equation}
where $\hat{z}_0(z_t)$ is the predicted clean latent state. 

The deterministic trajectory defines a fixed differentiable mapping from the initial noise $z_T$ to the final latent code $z_0$. The pre-trained diffusion model provides a structural prior for the optimization, while the deterministic DDIM trajectory enables backpropagation of the measurement loss to $z_T$.

\begin{remark}
From a numerical optimization perspective, retaining the stochasticity (i.e., $\eta>0$) makes direct gradient-based optimization substantially noisier and less stable. A stochastic sampling path implies that evaluating the measurement loss $\mathcal{L}(z_T)$ for the exact same initial noise $z_T$ would yield widely varying results across different iterations. The deterministic DDIM formulation is thus not merely a design choice, but provides a fixed and reproducible loss landscape for direct optimization.
\end{remark}

\paragraph{Latent Space Optimization}
With the deterministic mapping established, the inverse problem is solved by minimizing the discrepancy between the predicted observation and the measured data $y_{\text{obs}} \in \mathcal{H}$. Instead of sequentially modifying intermediate scores, we define the physical measurement loss $\mathcal{L}$ and optimize the initial latent noise $z_T$ directly:
\begin{equation}
    \mathcal{L}(z_T) = \| \tilde{\mathcal{F}}_\phi(\mathcal{D}(z_0(z_T))) - y_{\text{obs}} \|_{\mathcal{H}}^2.
\end{equation}
Utilizing the chain rule, precise gradients flow from the fully denoised reconstructed parameter field back through the deterministic sampling trajectory to $z_T$:
\begin{equation}
    \frac{\partial \mathcal{L}}{\partial z_T} = \frac{\partial \mathcal{L}}{\partial y} \cdot \frac{\partial \tilde{\mathcal{F}}_\phi}{\partial a} \cdot \frac{\partial \mathcal{D}}{\partial z_0} \cdot \frac{\partial z_0}{\partial z_T},
\end{equation}
where we denote the intermediate physical parameter as $a = \mathcal{D}(z_0)$ and the predicted observation as $y = \tilde{\mathcal{F}}_\phi(a)$.

Equipped with this gradient, we can iteratively update the initial latent state $z_T$ using standard gradient-based optimizers (e.g., Adam or Gradient Descent). We formally summarize this iterative procedure as the Diffusion Latent Optimization (DiLO) framework in Algorithm \ref{alg:DiLO}.

\begin{algorithm}[htbp]
\caption{Diffusion Latent Optimization  (DiLO)}
\label{alg:DiLO}
\begin{algorithmic}[1]
    \State \textbf{Require:} Pre-trained noise predictor $\epsilon_\theta$, Decoder $\mathcal{D}$, Differentiable Surrogate $\tilde{\mathcal{F}}_\phi$, Observations $y_{\text{obs}}$, Iterations $M$, learning rate $\eta_{\mathrm{lr}}$
    \State \textbf{Initialize:} Initial latent noise $z_T \sim \mathcal{N}(0, \mathbf{I})$ 
    \For{$i = 1$ to $M$} \Comment{Outer optimization loop}
        \State \textbf{1. Deterministic Trajectory ($\eta = 0$):}
        \For{$t = T$ \textbf{down to} $1$} 
            \State \quad $\hat{z}_0 = \frac{1}{\sqrt{\bar{\alpha}_t}} \left( z_t - \sqrt{1 - \bar{\alpha}_t} \epsilon_\theta(z_t, t) \right)$ \Comment{Predict clean latent}
            \State \quad $z_{t-1} = \sqrt{\bar{\alpha}_{t-1}}\hat{z}_0 + \sqrt{1 - \bar{\alpha}_{t-1}} \epsilon_\theta(z_t, t)$ \Comment{Deterministic DDIM update}
        \EndFor
        \State \quad \Comment{Yields final deterministic latent code $z_0$}
        
        \State \textbf{2. Physical Reconstruction:}
        \State \quad $\hat{a} = \mathcal{D}(z_0)$ \Comment{Map latent code to physical parameter}
        
        \State \textbf{3. Forward Physics \& Loss:}
        \State \quad $\hat{y} = \tilde{\mathcal{F}}_\phi(\hat{a})$ \Comment{Fast evaluation via Neural Operator}
        \State \quad $\mathcal{L}_{\text{physics}} = \|\hat{y} - y_{\text{obs}}\|_{\mathcal{H}}^2$
        
        \State \textbf{4. Latent Gradient Guidance:}
        \State \quad $\mathbf{g} = \nabla_{z_T} \mathcal{L}_{\text{physics}}$ \Comment{Solver-in-the-loop backpropagation}
        \State \quad $z_T \gets z_T - \eta_{\mathrm{lr}} \mathbf{g}$ \Comment{Update initial noise}
    \EndFor
    \State \textbf{Output:} Optimized physical parameter $a^* = \mathcal{D}(z_0(z_T^*))$
\end{algorithmic}
\end{algorithm}

\subsection{Convergence Analysis}
\label{sec:convergence}

We ground our analysis in Electrical Impedance Tomography (EIT); the convergence analysis for other physical systems follows a similar methodology. For the Neumann formulation, we assume the standard compatibility condition $\int_{\partial\Omega} g\,ds=0$ and represent boundary potentials using the zero-mean gauge $\int_{\partial\Omega}u\,ds=0$. The observed voltage $V_{\mathrm{obs}}$ is normalized using the same convention. The forward physical model $\mathcal{F}: \sigma \mapsto u|_{\partial \Omega}$ is governed by:
\begin{equation} \label{eq:eit_forward}
\begin{cases}
    \nabla \cdot (\sigma(\mathbf{x}) \nabla u(\mathbf{x})) = 0, \quad &\mathbf{x} \in \Omega \\
    \sigma(\mathbf{x}) \frac{\partial u}{\partial \mathbf{n}} = g. \quad &\mathbf{x} \in \partial \Omega
\end{cases}
\end{equation}
Assume the conductivity is restricted to an admissible set of uniformly elliptic $L^\infty$ functions:
\begin{equation}
    \Sigma_{\text{ad}} = \left\{ \sigma \in L^\infty(\Omega) \;\middle|\; 0 < \lambda \le \sigma(\mathbf{x}) \le \Lambda < \infty \text{ a.e. in } \Omega \right\}.
\end{equation}
DiLO optimizes the initial latent state $z_T \in \mathcal{Z}$ through a deterministic generation path $\sigma = \mathcal{M}(z_T)$. By substituting the Tweedie estimate \eqref{Tweedie_estimat}, the Deterministic Sampling Trajectory update \eqref{deterministic_ode} expands to:
\begin{equation} \label{eq:ddim_expanded}
    z_{t-1} = \underbrace{\frac{\sqrt{\bar{\alpha}_{t-1}}}{\sqrt{\bar{\alpha}_t}}}_{c_t} z_t + \underbrace{\left( \sqrt{1 - \bar{\alpha}_{t-1}} - \frac{\sqrt{\bar{\alpha}_{t-1}}}{\sqrt{\bar{\alpha}_t}} \sqrt{1-\bar{\alpha}_t} \right)}_{d_t} \epsilon_\theta(z_t, t).
\end{equation}
We explicitly assume that the deterministic generative map satisfies $\mathcal{M}(z_T) \in \Sigma_{\text{ad}}$ for all iterates considered. This condition can be enforced in practice by a bounded decoder output map with range $[\lambda,\Lambda]$ and guarantees the uniform ellipticity of the EIT forward equation.

Let the exact physical data misfit and its corresponding composite objective in the latent space be defined respectively as:
\begin{equation} \label{eq:exact_objectives}
    \mathcal{L}_{\text{phys}}(\sigma) = \frac{1}{2} \|\mathcal{F}(\sigma) - V_{\text{obs}}\|_{L^2(\partial \Omega)}^2, \qquad \mathcal{L}_{\text{exact}}(z_T) = \mathcal{L}_{\text{phys}}(\mathcal{M}(z_T)).
\end{equation}
DiLO optimizes a surrogate objective parameterized by a fast neural operator $\tilde{\mathcal{F}}_\phi \approx \mathcal{F}$:
\begin{equation} \label{eq:surrogate_objective}
    \mathcal{L}_{\text{surr}}(z_T) = \frac{1}{2} \|\tilde{\mathcal{F}}_\phi(\mathcal{M}(z_T)) - V_{\text{obs}}\|_{L^2(\partial \Omega)}^2.
\end{equation}

By the continuous adjoint state method\cite{feldman2019calderon}, the first-order Fréchet derivative of the physical misfit is $\nabla_\sigma \mathcal{L}_{\text{phys}}(\sigma) = -\nabla u \cdot \nabla w$, where $w(\mathbf{x})$ is the adjoint potential satisfying:
\begin{equation} \label{eq:adjoint}
\begin{cases}
    \nabla \cdot (\sigma(\mathbf{x}) \nabla w(\mathbf{x})) = 0, \quad &\mathbf{x} \in \Omega \\
    \sigma(\mathbf{x}) \frac{\partial w}{\partial \mathbf{n}} = u - V_{\text{obs}}. \quad &\mathbf{x} \in \partial \Omega
\end{cases}
\end{equation}
All Neumann states below, including the adjoint and its linearized states, are understood in the same zero-mean gauge. Since both $u|_{\partial\Omega}$ and $V_{\mathrm{obs}}$ have zero boundary mean, the adjoint datum $u-V_{\mathrm{obs}}$ satisfies the Neumann compatibility condition. Differentiating Eq.~\ref{eq:ddim_expanded}, the single-step generative Jacobian explicitly evaluates to:
\begin{equation} \label{eq:step_jacobian}
    J_t = \frac{\partial z_{t-1}}{\partial z_t} = c_t \mathbf{I} + d_t \nabla_{z_t} \epsilon_\theta(z_t, t).
\end{equation}

By the chain rule in function spaces, let $D\mathcal{M}(z_T):\mathcal{Z}\to L^\infty(\Omega)$ denote the Fr\'echet derivative of the generative map and let $D\mathcal{M}(z_T)^*:L^1(\Omega)\to\mathcal{Z}$ denote its adjoint under the $L^1$--$L^\infty$ duality pairing. Then
\begin{equation} \label{eq:exact_gradient}
\begin{aligned}
    \nabla_{z_T} \mathcal{L}_{\text{exact}}(z_T)
    &= D\mathcal{M}(z_T)^* \big( -\nabla u \cdot \nabla w \big),\\
    \big[\nabla_{z_T} \mathcal{L}_{\text{exact}}(z_T)\big]_j
    &= \int_\Omega \big(-\nabla u\cdot\nabla w\big)
       \frac{\partial \mathcal{M}(z_T)(\mathbf{x})}{\partial (z_T)_j}\,d\mathbf{x}.
\end{aligned}
\end{equation}
The derivative of the full generative map is the operator composition
\begin{equation} \label{eq:full_generative_jacobian}
    D\mathcal{M}(z_T) = D\mathcal{D}(z_0) \prod_{t=1}^T \big( c_t \mathbf{I} + d_t \nabla_{z_t} \epsilon_\theta(z_t, t) \big).
\end{equation}
After spatial discretization, $D\mathcal{M}(z_T)^*$ is represented by the usual conjugate transpose of the corresponding Jacobian matrix.

To establish the $L$-smoothness of $\nabla_{z_T} \mathcal{L}_{\text{exact}}$, we evaluate its second-order derivative. It is explicitly characterized through the second-order adjoint state method.

\begin{lemma}[Hessian via Second-Order Adjoints and Unrolled Dynamics] \label{lemma:hessian_decomp}
    Let $u$ and $w$ be the solutions to the forward equation \eqref{eq:eit_forward} and the first-order adjoint equation \eqref{eq:adjoint}, respectively. For any conductivity perturbation $\hat{\sigma}$ pushed forward from the latent space, we define the tangent state $\hat{u}$ and the second-order adjoint state $\hat{w}$ as the unique zero-mean solutions to the following boundary value problems:
    \begin{equation} \label{eq:second_order_pdes}
        \begin{cases}
            \nabla \cdot (\sigma \nabla \hat{u}) = -\nabla \cdot (\hat{\sigma} \nabla u) \\
            \sigma \frac{\partial \hat{u}}{\partial \mathbf{n}} = -\hat{\sigma} \frac{\partial u}{\partial \mathbf{n}}
        \end{cases}
        \qquad \text{and} \qquad
        \begin{cases}
            \nabla \cdot (\sigma \nabla \hat{w}) = -\nabla \cdot (\hat{\sigma} \nabla w) \\
            \sigma \frac{\partial \hat{w}}{\partial \mathbf{n}} = \hat{u} - \hat{\sigma} \frac{\partial w}{\partial \mathbf{n}}
        \end{cases}
    \end{equation}
    Using these solutions, the Hessian matrix $\mathbf{H}_{z_T} = \nabla_{z_T}^2 \mathcal{L}_{\text{exact}}(z_T) \in \mathbb{R}^{\dim(\mathcal{Z}) \times \dim(\mathcal{Z})}$ admits:
    \begin{equation} \label{eq:exact_hessian}
        \mathbf{H}_{z_T} = D\mathcal{M}(z_T)^* \Big[ \mathbf{H}_{\text{phys}} \Big] D\mathcal{M}(z_T) + \int_{\Omega} \big( -\nabla u \cdot \nabla w \big) \nabla_{z_T}^2 \mathcal{M}(\mathbf{x}; z_T) \, d\mathbf{x},
    \end{equation}
    where $D\mathcal{M}(z_T)^*$ is the function-space adjoint defined by the duality pairing in Eq.~\ref{eq:exact_gradient}. The action of the continuous physical Hessian operator $\mathbf{H}_{\text{phys}}$ on the perturbation $\hat{\sigma}$ is given by:
    \begin{equation} \label{eq:physical_hessian_action}
        \mathbf{H}_{\text{phys}} \hat{\sigma} = -\nabla \hat{u} \cdot \nabla w - \nabla u \cdot \nabla \hat{w},
    \end{equation}
    and the generative curvature tensor $\nabla_{z_T}^2 \mathcal{M}_i(z_T)$ is explicitly unrolled as a sum of step-wise diffusion Hessians pulled back via the accumulated Jacobians $J_s = c_s \mathbf{I} + d_s \nabla_{z_s} \epsilon_\theta(z_s, s)$:
    \begin{equation} \label{eq:generative_curvature_unrolled}
    \begin{aligned}
        \nabla_{z_T}^2 \mathcal{M}_i(z_T) &= \left( \prod_{s=1}^T J_s \right)^T \Big[ \nabla_{z_0}^2 \mathcal{D}_i(z_0) \Big] \left( \prod_{s=1}^T J_s \right) \\
        &\quad + \sum_{t=1}^T \left( \prod_{s=t+1}^T J_s \right)^T \Bigg[ \sum_k \left( J_{\mathcal{D}}(z_0) \prod_{m=1}^{t-1} J_m \right)_{i,k} d_t \nabla_{z_t}^2 \epsilon_\theta^{(k)}(z_t, t) \Bigg] \left( \prod_{s=t+1}^T J_s \right).
    \end{aligned}
    \end{equation}
\end{lemma}

\begin{proof}
    The decomposition follows from applying the generalized tensor chain rule to $\mathcal{L}_{\text{phys}} \circ \mathcal{M}$. The explicit bilinear formulation of the physical Hessian $\mathbf{H}_{\text{phys}}$ is derived by taking the directional Fréchet derivative of the first-order gradient, yielding the tangent and second-order adjoint PDEs. Concurrently, the unrolled generative curvature is derived by applying the multivariate chain rule for second derivatives to the discrete DDIM trajectory. The full step-by-step derivation is detailed in \hyperref[app:hessian_derivation]{Appendix~A}.
\end{proof}

For the convergence analysis, we use the following boundedness and approximation assumptions.

\begin{assumption}[Neural Network and Surrogate Bounds] \label{assump:bounds} \mbox{}
\begin{itemize}
    \item The continuous time-conditioned neural network $\epsilon_\theta(z_t, t)$ possesses uniformly bounded first and second spatial derivatives: $\|\nabla_{z_t} \epsilon_\theta\| \le L_{\epsilon, 1}$ and $\|\nabla_{z_t}^2 \epsilon_\theta\| \le L_{\epsilon, 2}$.
    \item The spatial decoder $\mathcal{D}(z_0)$ is twice continuously differentiable with bounded Jacobians and Hessians: $\|J_{\mathcal{D}}(z_0)\| \le L_{\mathcal{D}, 1}$ and $\|\nabla_{z_0}^2 \mathcal{D}(z_0)\| \le L_{\mathcal{D}, 2}$.
    \item The well-trained neural surrogate closely tracks the exact adjoint-derived gradient on the restricted manifold $\mathcal{M}(\mathcal{Z})$, such that $\|\nabla_{z_T} \mathcal{L}_{\text{surr}} - \nabla_{z_T} \mathcal{L}_{\text{exact}}\|_{\mathcal{Z}} \le \delta$.
    \item The surrogate objective is bounded below and has an $L_{\mathrm{surr}}$-Lipschitz gradient on the latent region visited by the iterates.
\end{itemize}
\end{assumption}

\begin{lemma}[Hessian-Bounded $L$-Smoothness] \label{lemma:smoothness}
    Under the domain restriction $\mathcal{M}(z_T) \in \Sigma_{\text{ad}}$, the exact objective $\mathcal{L}_{\text{exact}}(z_T)$ is uniformly $L_{\mathrm{exact}}$-smooth on the considered latent region (i.e., its Hessian is uniformly bounded). The smoothness of the surrogate objective is stated separately in Assumption~\ref{assump:bounds}; uniform gradient approximation alone does not imply smoothness.
\end{lemma}

\begin{proof}
    To establish this $L$-smoothness bound, we bound the spectral norm of the composite Hessian $\mathbf{H}_{z_T}$ derived in Eq.~\ref{eq:exact_hessian}. Applying the triangle inequality to the unrolled decomposition, we independently bound the generative and physical components.

    \vspace{0.5em}
    \noindent \textit{Step 1: Generative Boundedness.} \\
    By Assumption~\ref{assump:bounds}, the norm of the step-wise diffusion Jacobian $J_s = c_s \mathbf{I} + d_s \nabla_{z_s} \epsilon_\theta$ is uniformly bounded by $\|J_s\| \le |c_s| + |d_s| L_{\epsilon, 1} \equiv \rho_s$. Consequently, any accumulated backward or forward Jacobian is strictly bounded by the product of these constants, e.g., $\|\prod_{s=t+1}^T J_s\| \le \prod_{s=t+1}^T \rho_s$. 
    
    Applying the triangle and submultiplicative inequalities to the explicitly unrolled generative curvature (Eq.~\ref{eq:generative_curvature_unrolled}), we obtain:
    \begin{equation} \label{eq:generative_Hessian_bound}
    \begin{aligned}
        \|\nabla_{z_T}^2 \mathcal{M}(z_T)\| &\le \left( \prod_{s=1}^T \rho_s \right)^2 \|\nabla_{z_0}^2 \mathcal{D}\| + \sum_{t=1}^T \left( \prod_{s=t+1}^T \rho_s \right)^2 \|J_{\mathcal{D}}\| \left( \prod_{m=1}^{t-1} \rho_m \right) |d_t| \|\nabla_{z_t}^2 \epsilon_\theta\| \\
        &\le L_{\mathcal{D}, 2} \left( \prod_{s=1}^T \rho_s \right)^2 + L_{\mathcal{D}, 1} L_{\epsilon, 2} \sum_{t=1}^T |d_t| \left( \prod_{s=t+1}^T \rho_s \right)^2 \left( \prod_{m=1}^{t-1} \rho_m \right) \equiv C_H < \infty.
    \end{aligned}
    \end{equation}
    Similarly, the full generative Jacobian is bounded by $\|J_{\mathcal{M}}(z_T)\| \le L_{\mathcal{D}, 1} \prod_{t=1}^T \rho_t \equiv C_J < \infty$. Since $T$ is finite, $C_H$ and $C_J$ are well-defined finite constants.

    \vspace{0.5em}
    \noindent \textit{Step 2: Physical Regularity and Boundedness.} \\
    By the assumed restriction $M(z_T)\in\Sigma_{\mathrm{ad}}$, the conductivity satisfies the uniform ellipticity condition $\sigma\geq\lambda>0$. By standard elliptic regularity and the Lax-Milgram theorem, the weak solutions to the forward equation \eqref{eq:eit_forward} and the adjoint equation \eqref{eq:adjoint} are bounded in $H^1(\Omega)$ norms: $\|u\|_{H^1} \le C_u$ and $\|w\|_{H^1} \le C_w$. 
    
    To bound the physical Hessian action (Eq.~\ref{eq:physical_hessian_action}), we evaluate the dependence of the tangent state $\hat{u}$ and the second-order adjoint state $\hat{w}$ on the latent perturbation $\hat{\sigma}$. The tangent linear equation is governed by $\nabla \cdot (\sigma \nabla \hat{u}) = -\nabla \cdot (\hat{\sigma} \nabla u)$. Testing this weak form with $v = \hat{u}$ and applying uniform ellipticity and Hölder's inequality yields the energetic bound:
    \begin{equation}
        \lambda \|\nabla \hat{u}\|_{L^2(\Omega)}^2 \le \int_\Omega \sigma |\nabla \hat{u}|^2 = -\int_\Omega \hat{\sigma} \nabla u \cdot \nabla \hat{u} \le \|\hat{\sigma}\|_{L^\infty(\Omega)} \|\nabla u\|_{L^2(\Omega)} \|\nabla \hat{u}\|_{L^2(\Omega)}.
    \end{equation}
    Dividing by $\|\nabla \hat{u}\|_{L^2}$ confirms that $\|\nabla \hat{u}\|_{L^2} \le \frac{1}{\lambda} \|\nabla u\|_{L^2} \|\hat{\sigma}\|_{L^\infty} \le C_{\hat{u}} \|\hat{\sigma}\|_{L^\infty}$. 
    
    A parallel elliptic regularity argument, augmented by trace theorems for the state-dependent boundary condition $\sigma \partial_{\mathbf{n}} \hat{w} = \hat{u} - \hat{\sigma} \partial_{\mathbf{n}} w$, establishes that the second-order adjoint gradient is similarly bounded: $\|\nabla \hat{w}\|_{L^2} \le C_{\hat{w}} \|\hat{\sigma}\|_{L^\infty}$.

    Consequently, applying the Cauchy-Schwarz inequality to the explicit physical Hessian bilinear form (Eq.~\ref{eq:physical_hessian_action}) yields:
    \begin{equation} \label{eq:Physical_boundedness}
    \begin{aligned}
        \|\mathbf{H}_{\text{phys}} \hat{\sigma}\|_{L^1(\Omega)} &= \|-\nabla \hat{u} \cdot \nabla w - \nabla u \cdot \nabla \hat{w}\|_{L^1(\Omega)} \\
        &\le \|\nabla \hat{u}\|_{L^2(\Omega)} \|\nabla w\|_{L^2(\Omega)} + \|\nabla u\|_{L^2(\Omega)} \|\nabla \hat{w}\|_{L^2(\Omega)} \\
        &\le \big( C_{\hat{u}} C_w + C_u C_{\hat{w}} \big) \|\hat{\sigma}\|_{L^\infty(\Omega)} \equiv C_{\text{phys}} \|\hat{\sigma}\|_{L^\infty(\Omega)}.
    \end{aligned}
    \end{equation}
    This proves that the physical Hessian operator $\mathbf{H}_{\text{phys}}$ is uniformly bounded by a constant $C_{\text{phys}}$ on the admissible manifold. Furthermore, the first-order gradient satisfies $\|\nabla_\sigma \mathcal{L}_{\text{phys}}(\sigma)\|_{L^1} \le C_u C_w \equiv C_G < \infty$.

    \vspace{0.5em}
    \noindent \textit{Step 3: Synthesis.} \\
    Substituting the explicit bounds derived in Step 1 and Step 2 back into the composite Hessian decomposition (Eq.~\ref{eq:exact_hessian}) yields the uniform upper bound:
    \begin{equation}
        \|\mathbf{H}_{z_T}\| \le C_J^2 C_{\text{phys}} + C_G C_H \equiv L_{\mathrm{exact}} < \infty.
    \end{equation}
    Because the spectral norm of the explicit composite Hessian is uniformly bounded by $L_{\mathrm{exact}}$ on the considered latent region, the exact objective $\mathcal{L}_{\text{exact}}(z_T)$ is $L_{\mathrm{exact}}$-smooth.
\end{proof}

Based on the $L$-smoothness established above, we now establish asymptotic stationarity of the surrogate-gradient sequence and the corresponding stationarity bound for the exact objective under fixed-step gradient descent; the Adam and AdamW optimizers used in the numerical experiments are not covered by this convergence statement.

\begin{theorem}[Asymptotic Stationarity under Surrogate Gradient Error] \label{thm:convergence}
Let $\{z_{T}^{(k)}\}_{k\ge 0}$ be generated by gradient descent on the surrogate loss,
$z_{T}^{(k+1)} = z_{T}^{(k)} - \eta_{\mathrm{lr}} \nabla_{z_T} \mathcal{L}_{\text{surr}}(z_{T}^{(k)})$,
with a constant learning rate $0 < \eta_{\mathrm{lr}} \le 1/L_{\mathrm{surr}}$.
Under Assumption~\ref{assump:bounds},
\begin{equation}
    \lim_{K \to \infty} \|\nabla_{z_T} \mathcal{L}_{\text{surr}}(z_{T}^{(K)})\|_{\mathcal{Z}} = 0.
\end{equation}
Moreover, the same sequence is asymptotically $\delta$-stationary for the exact physical objective:
\begin{equation}
    \limsup_{K \to \infty} \|\nabla_{z_T} \mathcal{L}_{\text{exact}}(z_{T}^{(K)})\|_{\mathcal{Z}} \le \delta.
\end{equation}
\end{theorem}

\begin{proof}
By the $L_{\mathrm{surr}}$-smoothness of the surrogate objective, the descent lemma gives
\begin{equation}
\begin{aligned}
    \mathcal{L}_{\text{surr}}(z_{T}^{(k+1)})
    &\le \mathcal{L}_{\text{surr}}(z_{T}^{(k)})
    - \eta_{\mathrm{lr}}\left(1-\frac{L_{\mathrm{surr}}\eta_{\mathrm{lr}}}{2}\right)
    \|\nabla_{z_T}\mathcal{L}_{\text{surr}}(z_{T}^{(k)})\|_{\mathcal{Z}}^2\\
    &\le \mathcal{L}_{\text{surr}}(z_{T}^{(k)})
    - \frac{\eta_{\mathrm{lr}}}{2}
    \|\nabla_{z_T}\mathcal{L}_{\text{surr}}(z_{T}^{(k)})\|_{\mathcal{Z}}^2.
\end{aligned}
\end{equation}
Summing from $k=0$ to $K-1$ and using the lower bound
$\mathcal{L}_{\text{surr}}^{\inf}>-\infty$ yields
\begin{equation}
    \frac{\eta_{\mathrm{lr}}}{2}\sum_{k=0}^{K-1}
    \|\nabla_{z_T}\mathcal{L}_{\text{surr}}(z_{T}^{(k)})\|_{\mathcal{Z}}^2
    \le \mathcal{L}_{\text{surr}}(z_T^{(0)})-\mathcal{L}_{\text{surr}}^{\inf}.
\end{equation}
Hence the nonnegative series of squared gradient norms is summable, so its terms converge to zero. Finally, the surrogate-gradient approximation bound holds at every iterate:
\begin{equation}
    \|\nabla_{z_T}\mathcal{L}_{\text{exact}}(z_T^{(K)})\|_{\mathcal{Z}}
    \le \delta + \|\nabla_{z_T}\mathcal{L}_{\text{surr}}(z_T^{(K)})\|_{\mathcal{Z}}.
\end{equation}
Taking $\limsup_{K\to\infty}$ proves the second assertion without assuming that the iterates $z_T^{(K)}$ themselves converge.
\end{proof}


\section{Experimental Results.}\label{Experiment}

\subsection{Dataset Preparation and Pretraining Setup}

For each problem, we use a diffusion prior for the unknown physical parameter and an FNO~\cite{li2020fourier} as the differentiable forward surrogate. The diffusion prior and its latent representation depend on the task.

The generative prior is configured separately for the three problems. For EIT, we use the VQ-f4 autoencoder of Rombach et al.~\cite{rombach2022high}, which maps a $256\times256$ field to a $3\times64\times64$ latent representation, together with an FFHQ-pretrained latent diffusion model fine-tuned on LIDC data for $60{,}000$ steps. For inverse scattering, we directly use the pretrained pixel-space diffusion prior released with InverseBench~\cite{zheng2025inversebench}; no autoencoder is used, and the optimized variable is the initial diffusion state. For inverse N-S, we use a convolutional autoencoder with channels $1\rightarrow64\rightarrow128\rightarrow256$, $3\times3$ stride-2 convolutions and ReLU activations, followed by a fully connected projection to a 128-dimensional latent code and a symmetric decoder.

\paragraph{Electrical Impedance Tomography (EIT)}
We use the LIDC dataset~\cite{armato2011lung} to fine-tune the diffusion prior. The FNO is trained on $3{,}000$ samples and evaluated on $200$ samples at $256\times256$ resolution, with conductivity values in $[0.01,1.0]$. Boundary voltage measurements are lifted to a 2D interior potential field by solving $\Delta u=0$ with a finite-difference scheme. The FNO has four layers and eight Fourier modes.

\paragraph{Inverse Scattering}
We use the pretrained pixel-space diffusion prior from InverseBench~\cite{zheng2025inversebench}, trained on HL60 nucleus permittivity images. Training data for the FNO are generated by solving the 2D Helmholtz equation with fourth-order finite elements in NGSolve. We use five frequencies from $1.0$ to $5.0$ Hz, $64$ incident plane-wave directions, and $360$ receivers on a circle of radius $1.6$ m. The scattered fields are interpolated onto a $128\times128$ grid.

\paragraph{Inverse Navier--Stokes (N-S)}
We use $10{,}000$ initial states from InverseBench~\cite{zheng2025inversebench} and simulate the 2D incompressible Navier--Stokes equations at $\mathrm{Re}=200$ with forcing $f=-4\cos(4y)$. The equations are integrated on a $512\times512$ grid with $\Delta t=10^{-3}$ using a pseudo-spectral solver. After a burn-in period of $T=100$, the vorticity fields are downsampled for FNO training.

Unless otherwise stated, experiments were run in PyTorch on a server with eight NVIDIA GeForce RTX~4090 GPUs (24\,GB each). The EIT diffusion prior was fine-tuned for $60{,}000$ steps; the inverse-scattering prior was used without additional diffusion training; and for inverse N-S, the autoencoder and latent diffusion model were trained for $100$ and approximately $72$ epochs, respectively. The runtime and memory benchmarks reported below were measured separately on a single NVIDIA GH200 GPU.

\subsection{Inference Setting and Intermediate Results}

\paragraph{EIT} For the EIT problem, the latent optimization is run for $3{,}000$ iterations with a step size of $5\times10^{-3}$ using an AdamW optimizer. At each iteration, a fixed $50$-step deterministic DDIM trajectory generates a forward prediction, which is compared against the measured voltage data on a $256\times256$ boundary-element mesh ($360$ electrodes equally spaced on a unit circle). Figure~\ref{fig:eit_inter_evolution} illustrates the iterative reconstruction process for the EIT problem. The top row displays the conductivity fields reconstructed at specific intervals, with the final column showing the Ground Truth (GT) distribution for reference. The bottom row depicts the corresponding absolute error maps between the current reconstruction and the GT. The GT error map is identically zero and is therefore omitted from the last position. Starting from a random Gaussian latent code, the surrogate-guided prior rapidly resolves the boundary contours of inclusion bodies within the first few hundred iterations, converging stably to an accurate physical field.

\begin{figure}[htbp]
    \centering
    {\includegraphics[width=0.16\textwidth]{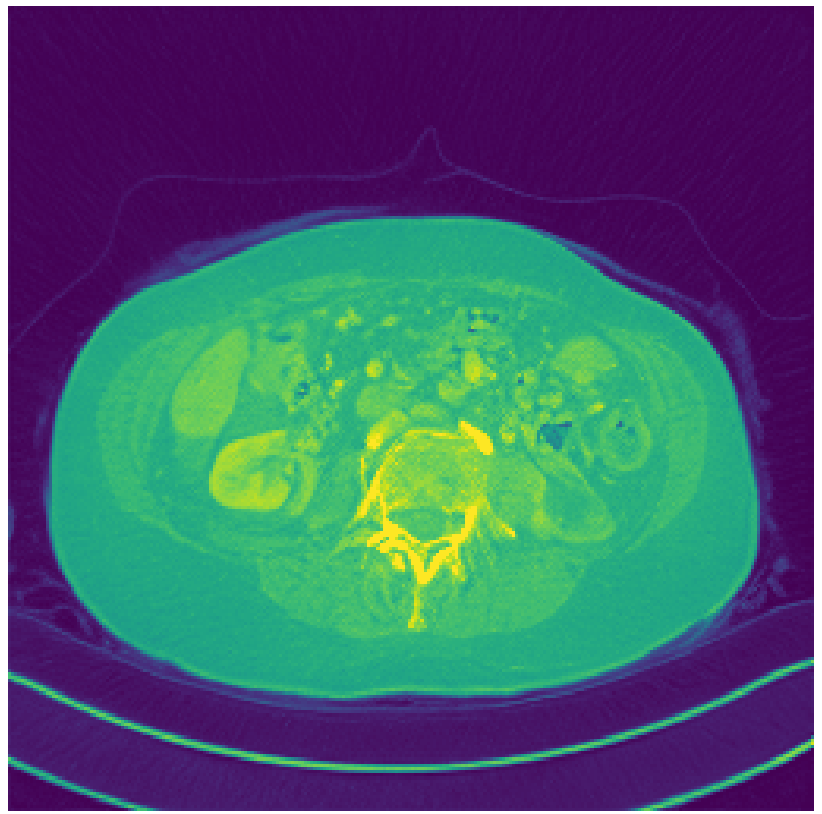}}
    {\includegraphics[width=0.16\textwidth]{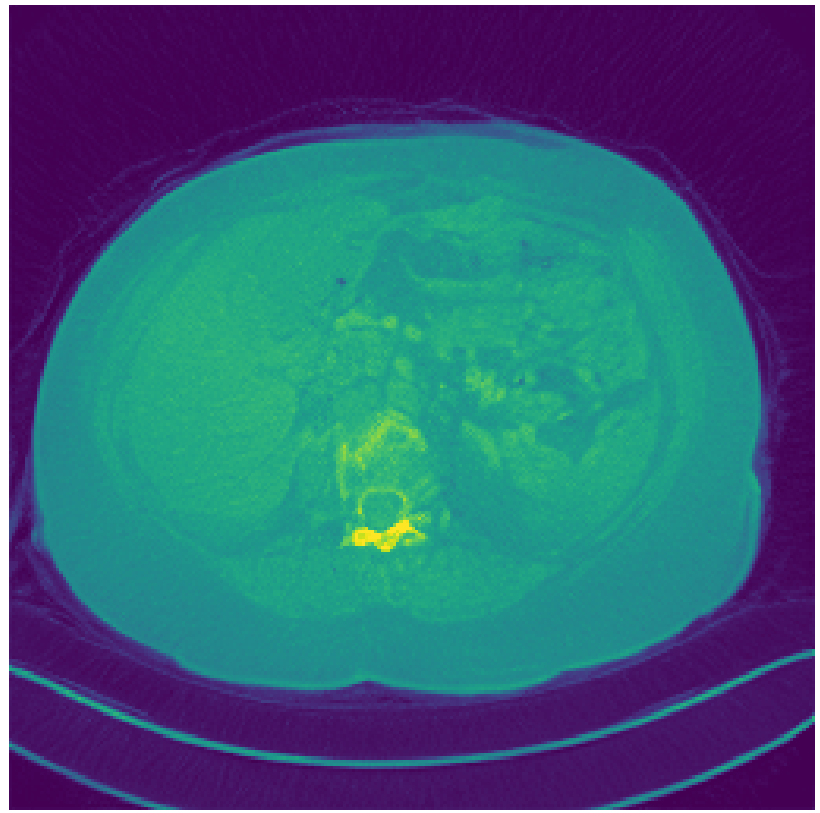}}
    {\includegraphics[width=0.16\textwidth]{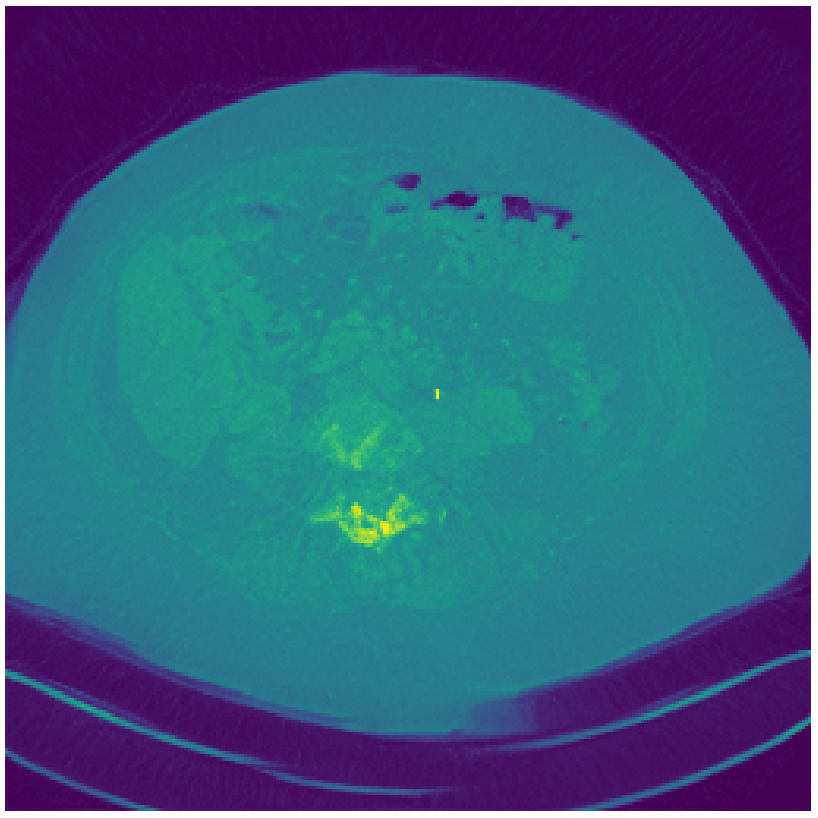}}
    {\includegraphics[width=0.16\textwidth]{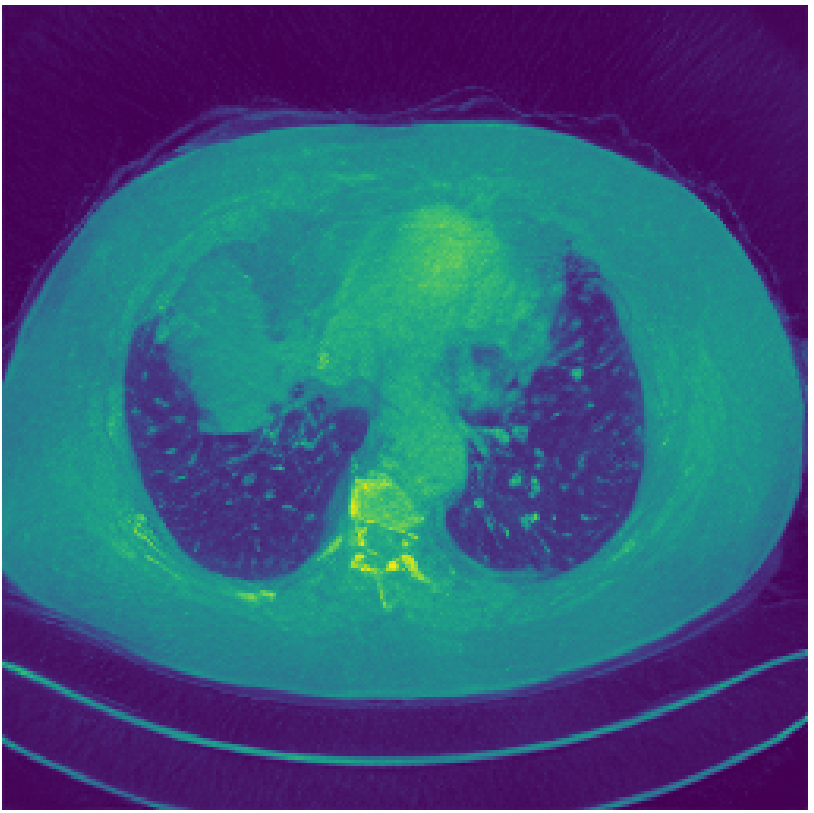}}
    {\includegraphics[width=0.16\textwidth]{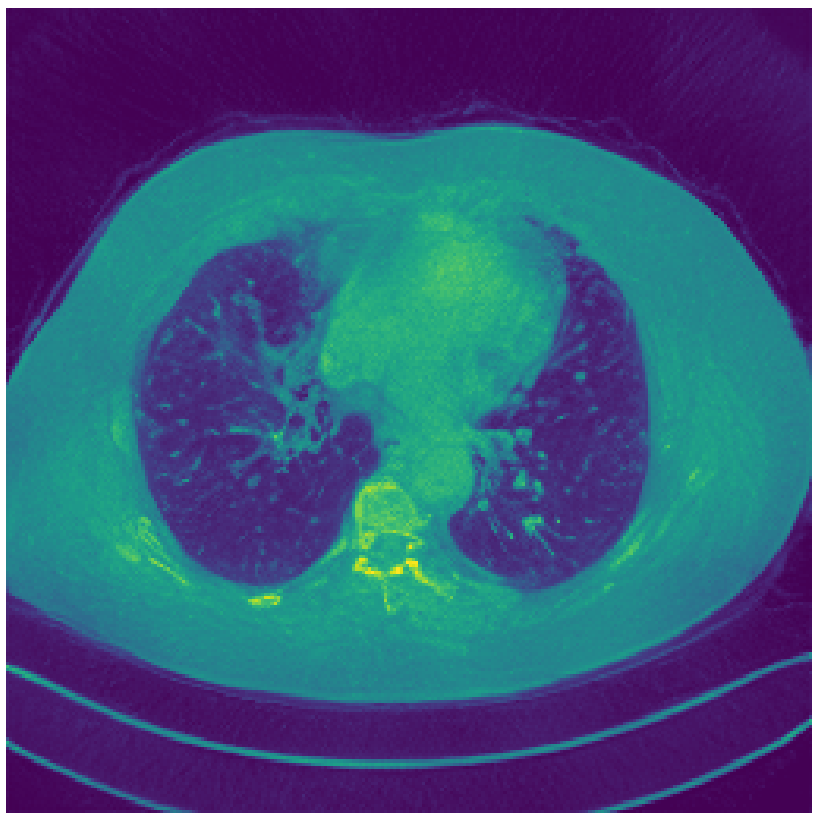}}
    {\includegraphics[width=0.16\textwidth]{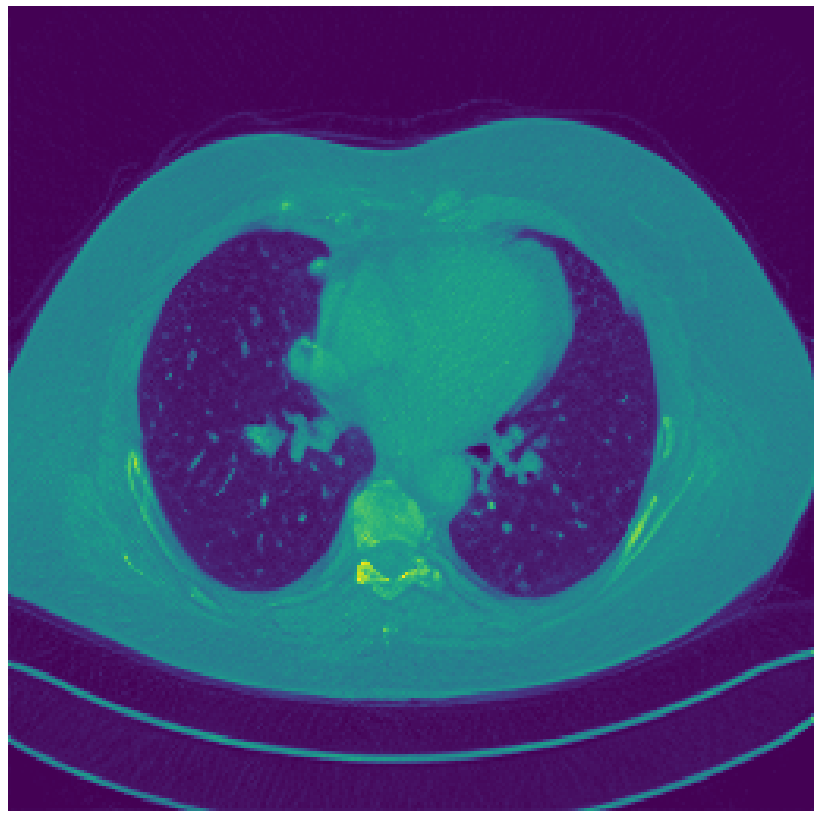}}
    
    \vspace{-0.1cm} 
    
    \hspace{-2.75cm}\subfigure[Iteration 0000]{\includegraphics[width=0.16\textwidth]{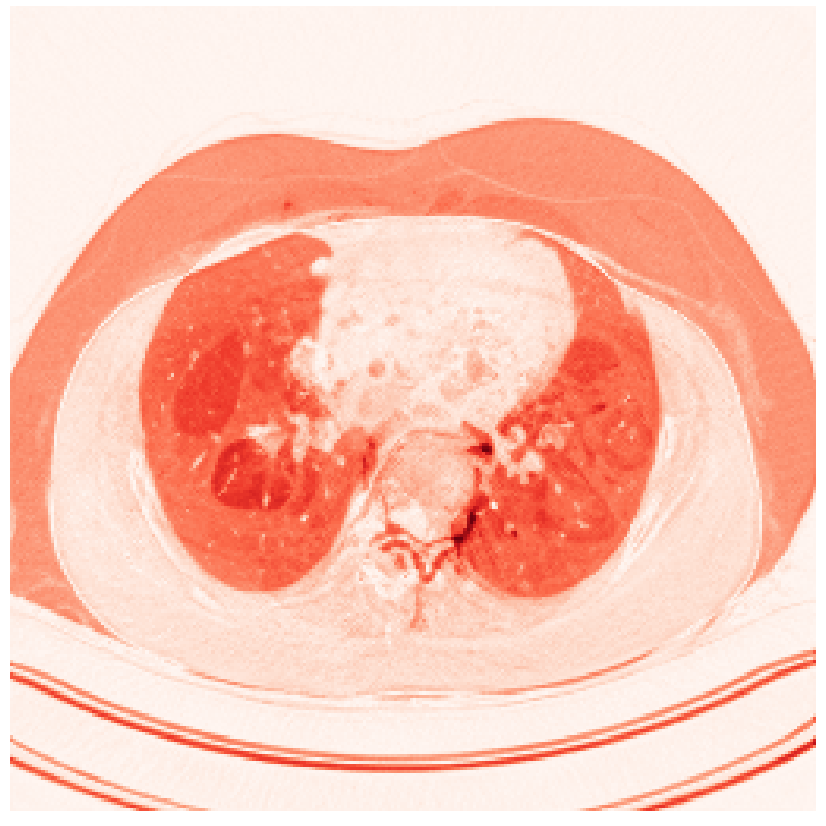}}
    \subfigure[Iteration 0050]{\includegraphics[width=0.16\textwidth]{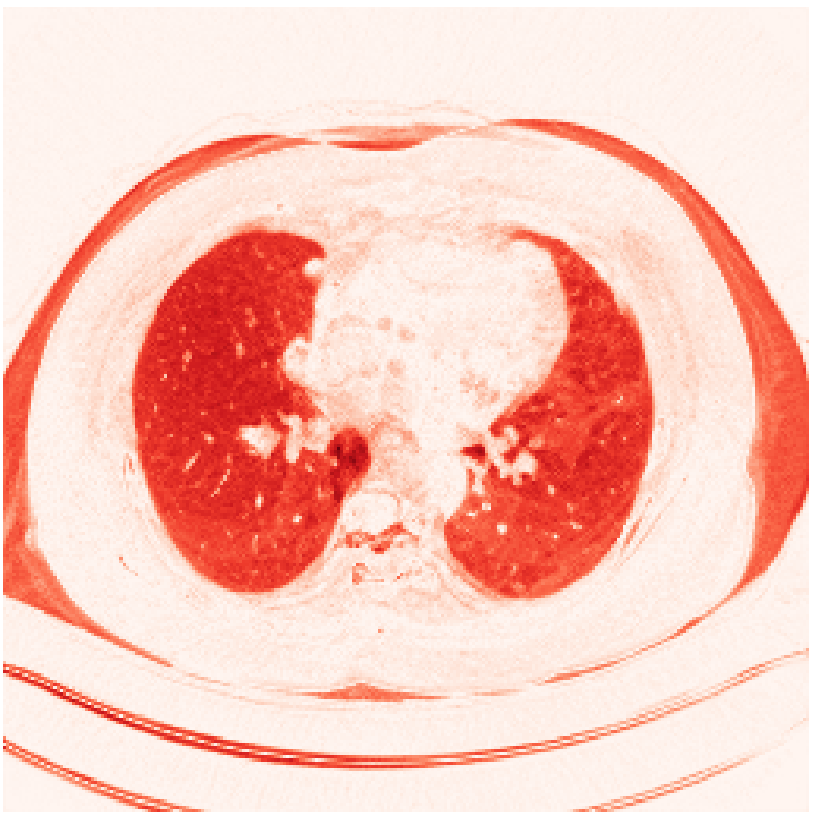}}
    \subfigure[Iteration 0100]{\includegraphics[width=0.16\textwidth]{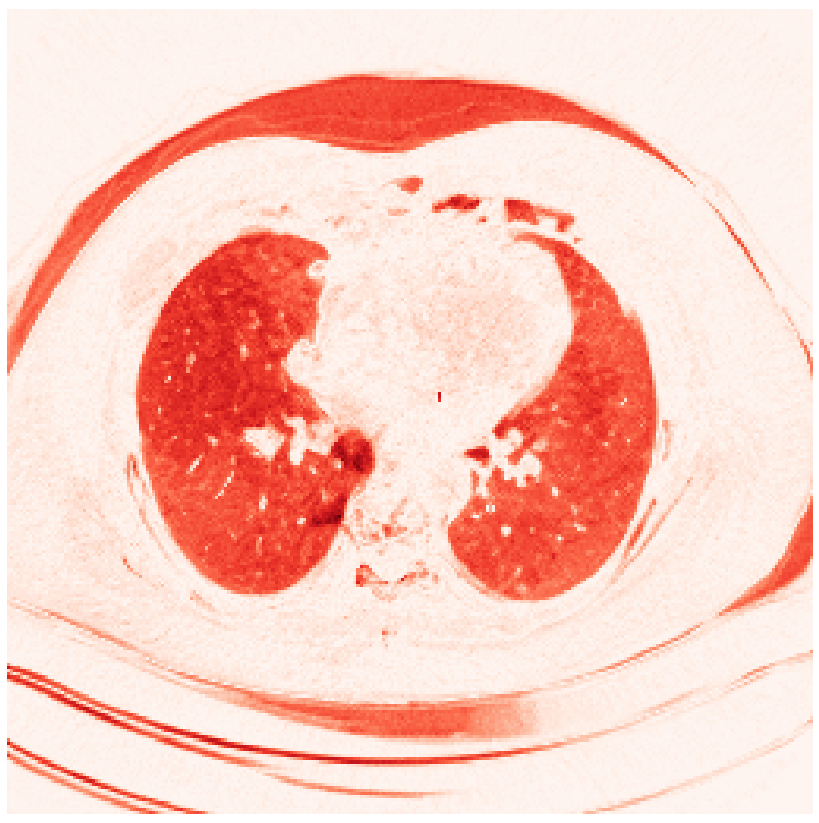}}
    \subfigure[Iteration 0500]{\includegraphics[width=0.16\textwidth]{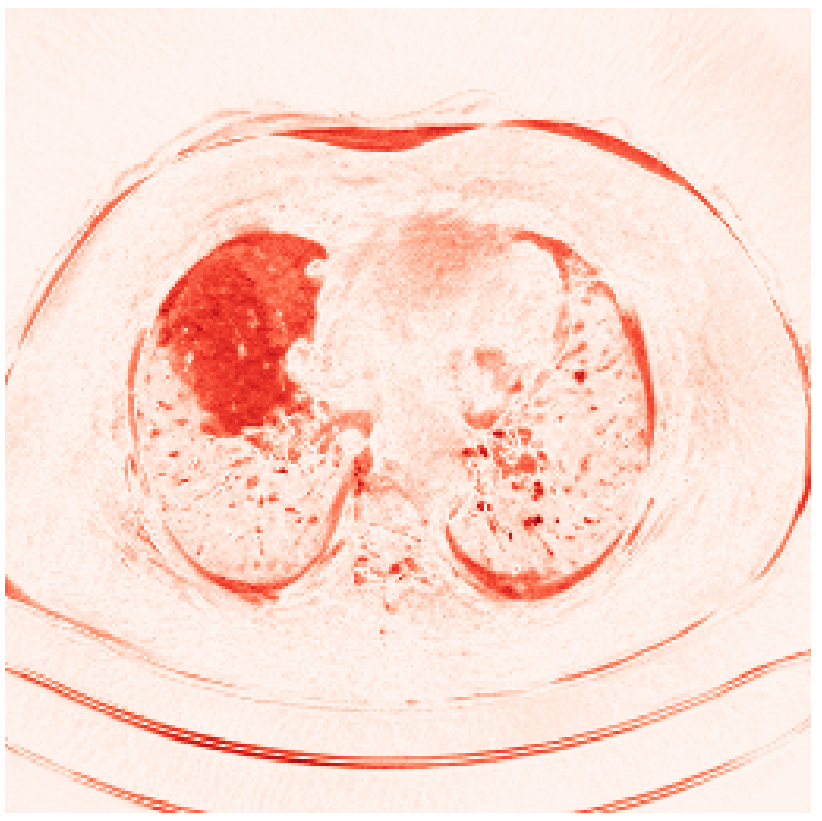}}
    \subfigure[Iteration 1000]{\includegraphics[width=0.16\textwidth]{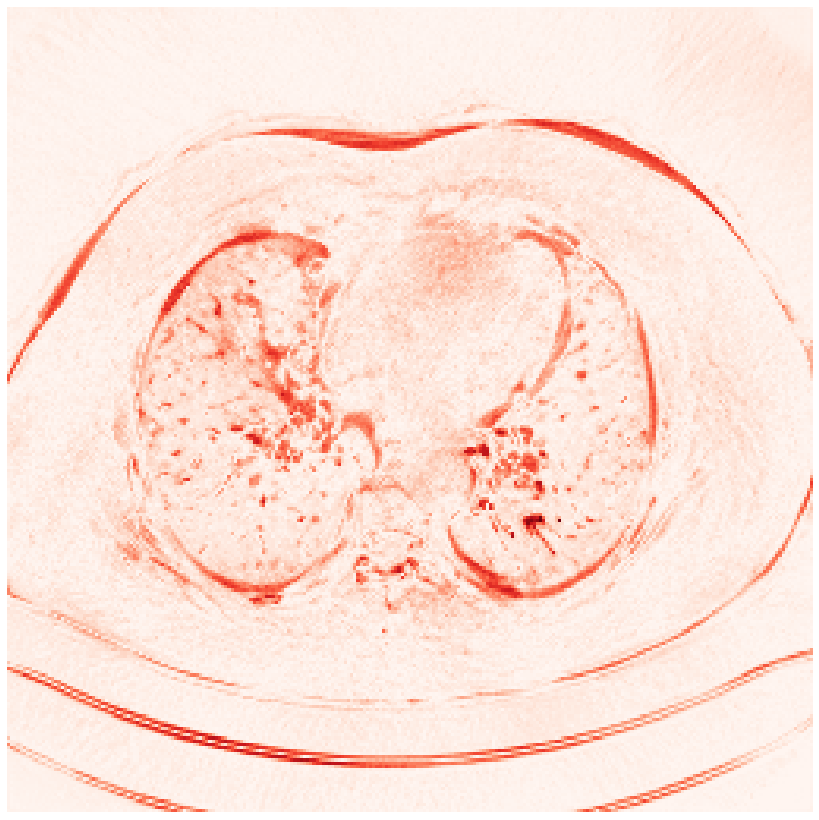}}

    \caption{Detailed iterative reconstruction for EIT. Top row: reconstructed conductivity at successive iterations and the GT conductivity. Bottom row: absolute reconstruction errors; the identically zero GT error map is omitted.}
    \label{fig:eit_inter_evolution}
\end{figure}

\paragraph{Inverse Scattering} For the Inverse Scattering problem (Figure~\ref{fig:is_inter}), the latent optimization proceeds over $3{,}000$ steps using an AdamW optimizer. The FNO surrogate is trained with $64$ propagation directions, $16\times16$ Fourier modes, width $64$, and $6$ layers. At each step, $50$-step deterministic DDIM sampling projects the latent contrast to the physical domain, and the scattered field residual is evaluated at $360$ equispaced sensors ($L_x=L_y=1.0$, radius $=1.6$). The physics loss combines an absolute term (weight $=5.0$) and a relative term (weight $=2.0$). The top row of Figure~\ref{fig:is_inter} shows the reconstructed scattering contrast, whereas the bottom row shows the forward scattered field associated with each reconstruction. The lower-right panel is the nonzero ground-truth scattered field, not an error map. Starting from a randomly initialized scattering contrast, the optimization recovers complex internal structural patterns within the first several hundred iterations, while maintaining physical consistency between the predicted and measured scattered fields throughout the entire trajectory.

\begin{figure}[htbp]
    \centering
    {\includegraphics[width=0.16\textwidth]{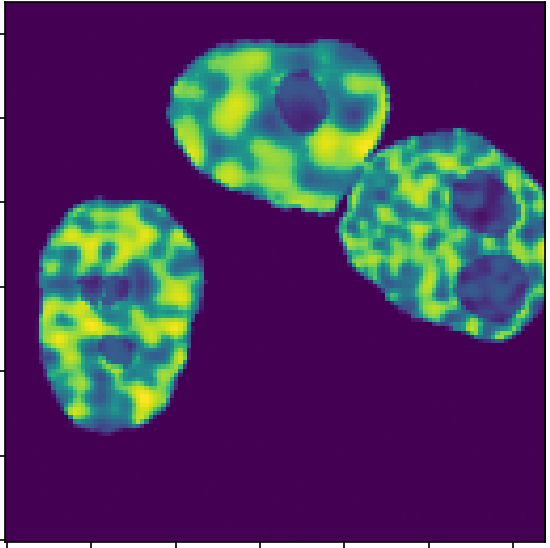}}
    {\includegraphics[width=0.16\textwidth]{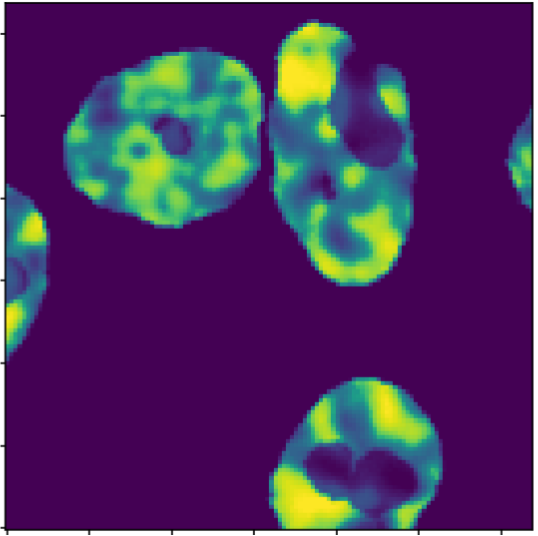}}
    {\includegraphics[width=0.16\textwidth]{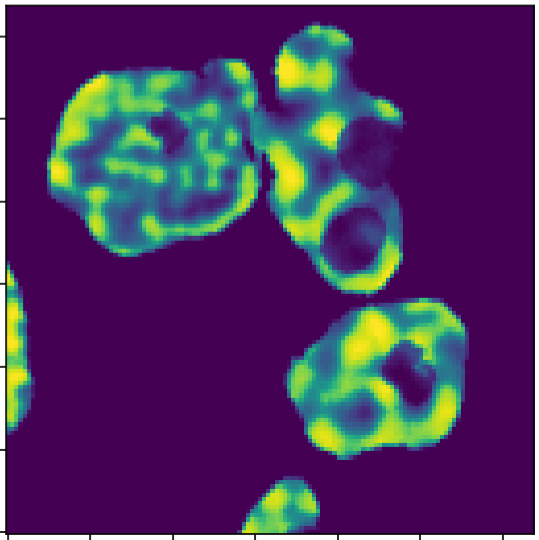}}
    {\includegraphics[width=0.16\textwidth]{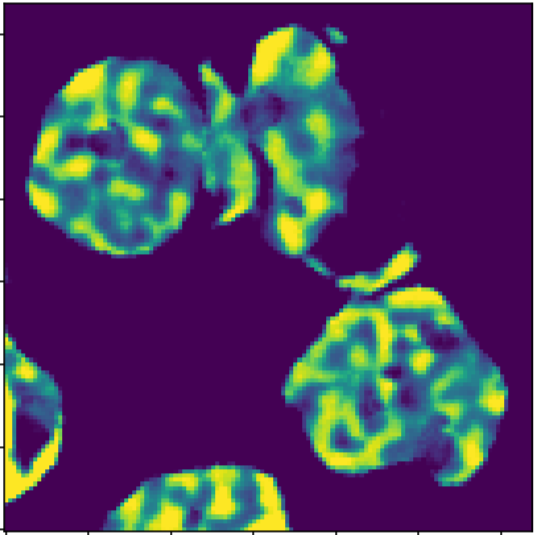}}
    {\includegraphics[width=0.16\textwidth]{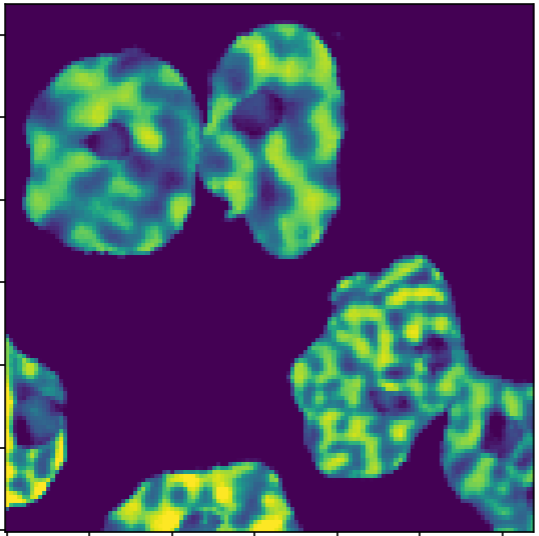}}
    {\includegraphics[width=0.16\textwidth]{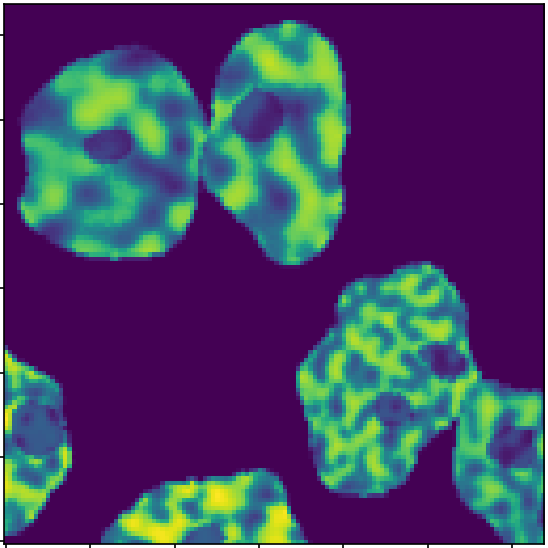}}
    
    \vspace{-0.1cm} 
    
    \subfigure[Iteration 0000]{\includegraphics[width=0.16\textwidth]{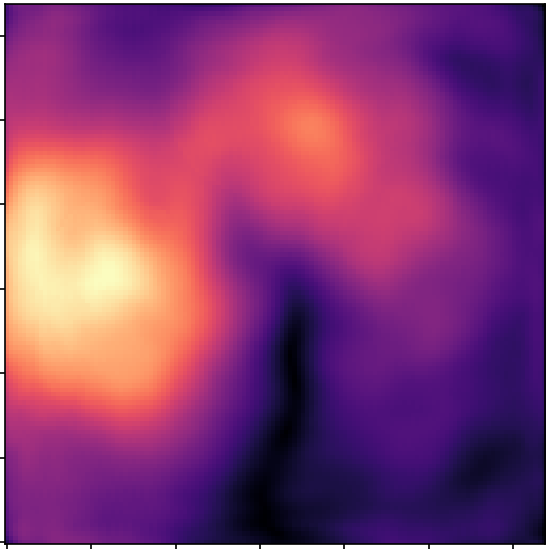}}
    \subfigure[Iteration 0050]{\includegraphics[width=0.16\textwidth]{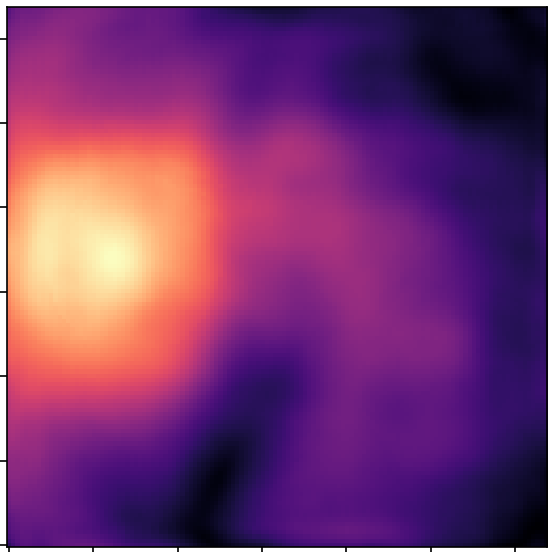}}
    \subfigure[Iteration 0200]{\includegraphics[width=0.16\textwidth]{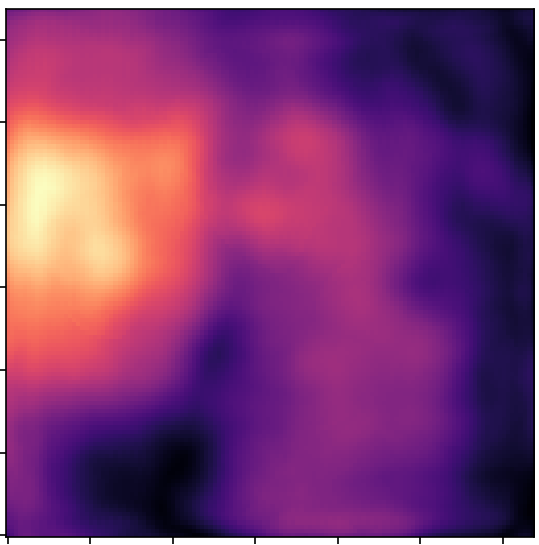}}
    \subfigure[Iteration 1000]{\includegraphics[width=0.16\textwidth]{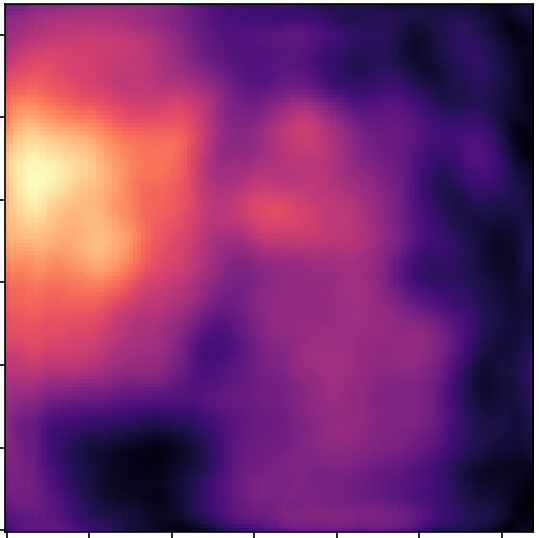}}
    \subfigure[Iteration 2000]{\includegraphics[width=0.16\textwidth]{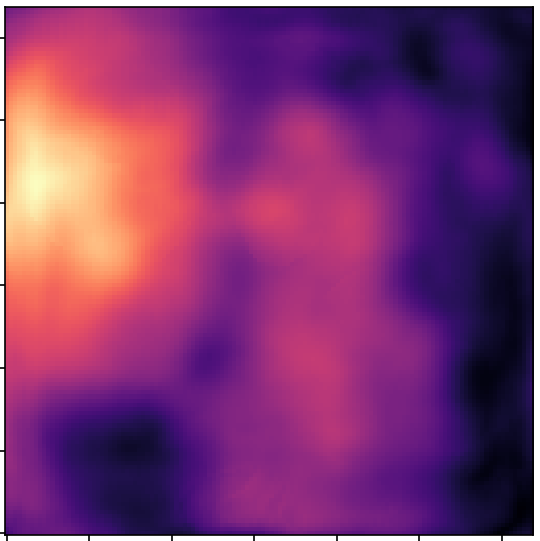}}
    \subfigure[Ground Truth]{\includegraphics[width=0.16\textwidth]{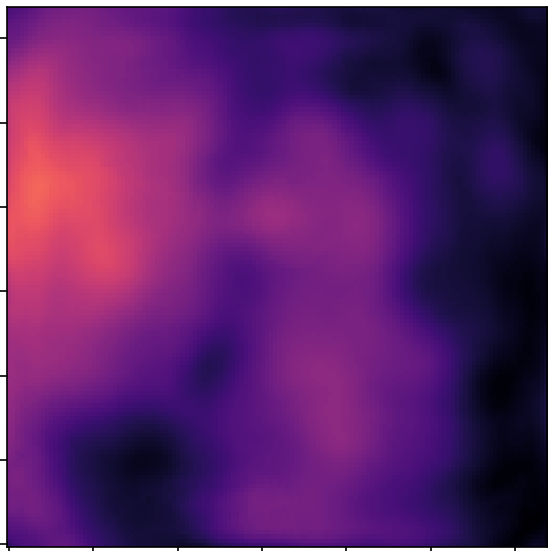}}
    
    \caption{Detailed iterative reconstruction for inverse scattering. Top row: reconstructed scattering contrast. Bottom row: corresponding predicted scattered fields, with the ground-truth scattered field in the lower-right panel.}
    \label{fig:is_inter}
\end{figure}

\paragraph{Inverse N-S}
For the Inverse Navier-Stokes problem, the latent optimization runs for $2{,}000$ iterations with a step size of $1\times10^{-2}$ using an Adam optimizer ($\beta_1=0.9$, $\beta_2=0.999$, $\epsilon=10^{-8}$). The autoencoder compresses the $32\times32$ vorticity field ($1{,}024$ pixels) into a $128$-dimensional latent code, and the diffusion model operates in the corresponding $128$-dimensional latent space with 500 denoising steps. The FNO forward model is trained on 40 timesteps with width 20 and 8 Fourier modes, and the measurement loss is the $L^2$ distance between the predicted terminal vorticity field and the target observations. In Figure~\ref{fig:ns_fast_convergence}, the top row shows reconstructed initial vorticity fields, while the bottom row shows the terminal fields predicted by the forward model; the lower-right panel is the nonzero ground-truth terminal field. Figure~\ref{fig:ns_fast_convergence} shows that the large-scale flow structure is already visible by Iteration~20 and becomes progressively closer to the ground truth over subsequent iterations.

\begin{figure}[htbp]
    \centering
    {\includegraphics[width=0.16\textwidth]{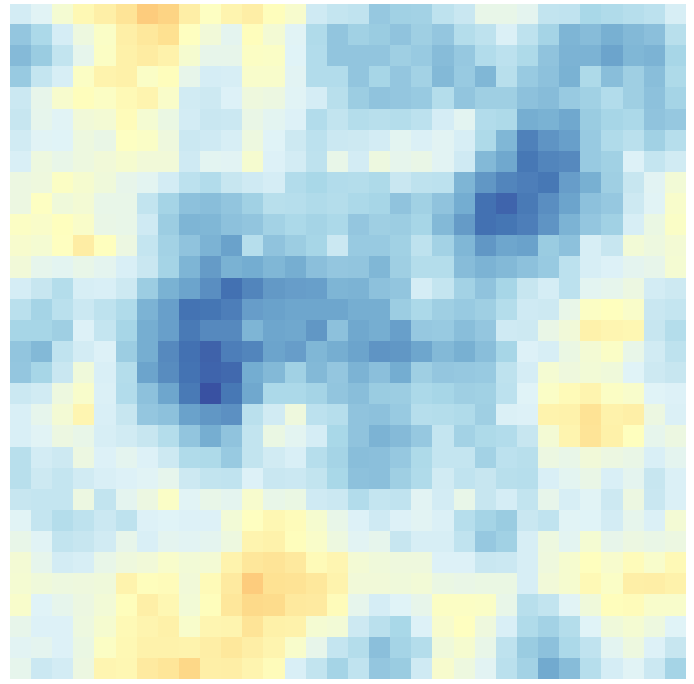}}
    {\includegraphics[width=0.16\textwidth]{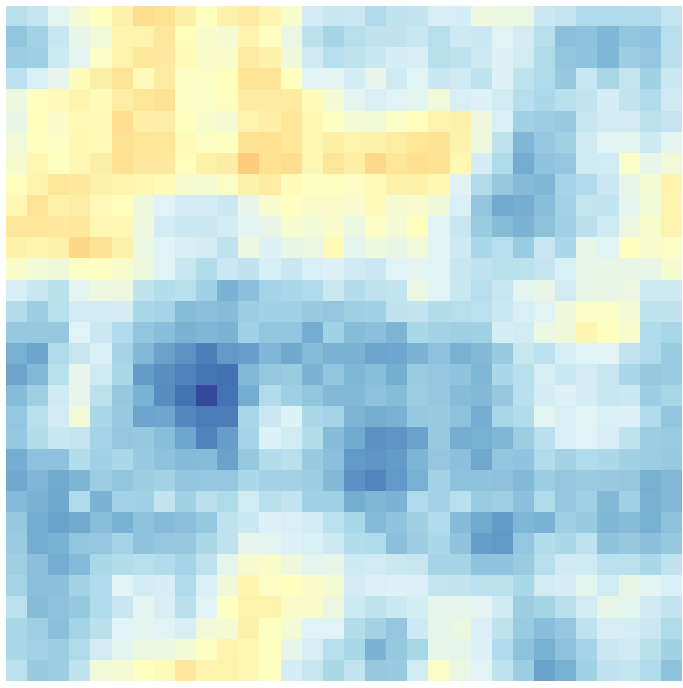}}
    {\includegraphics[width=0.16\textwidth]{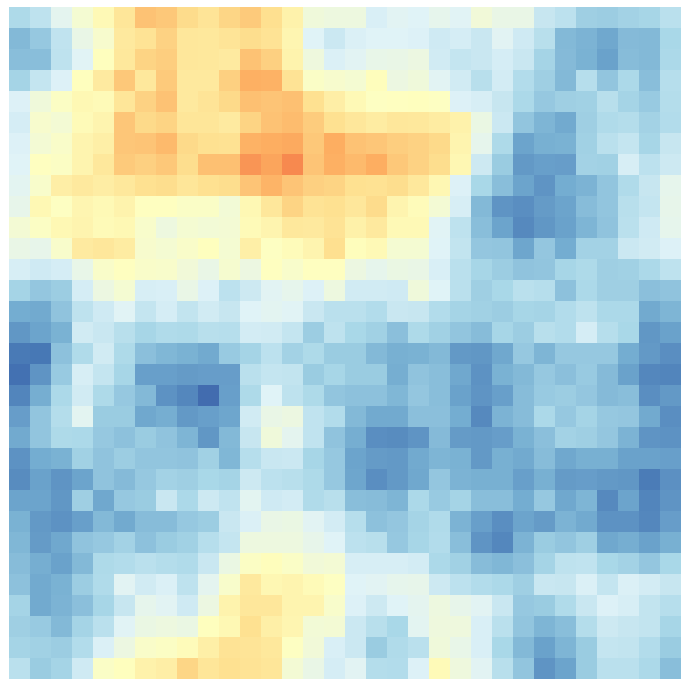}}
    {\includegraphics[width=0.16\textwidth]{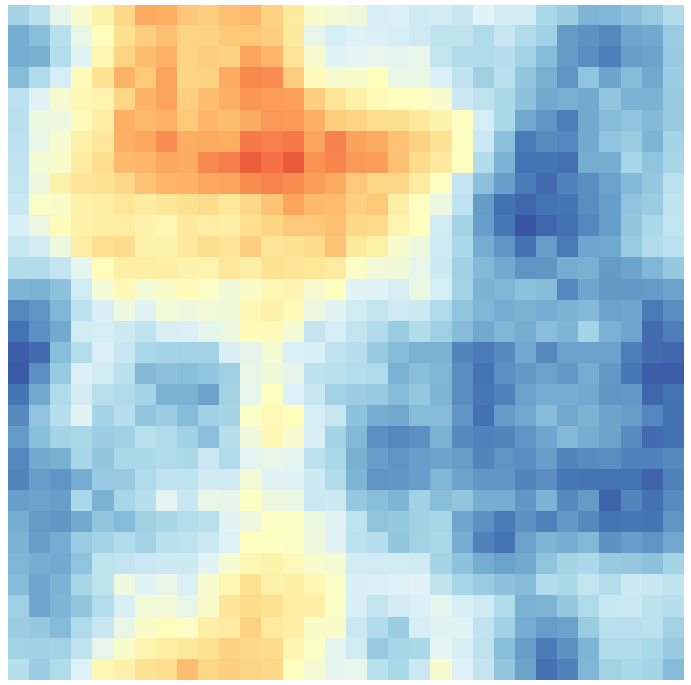}}
    {\includegraphics[width=0.16\textwidth]{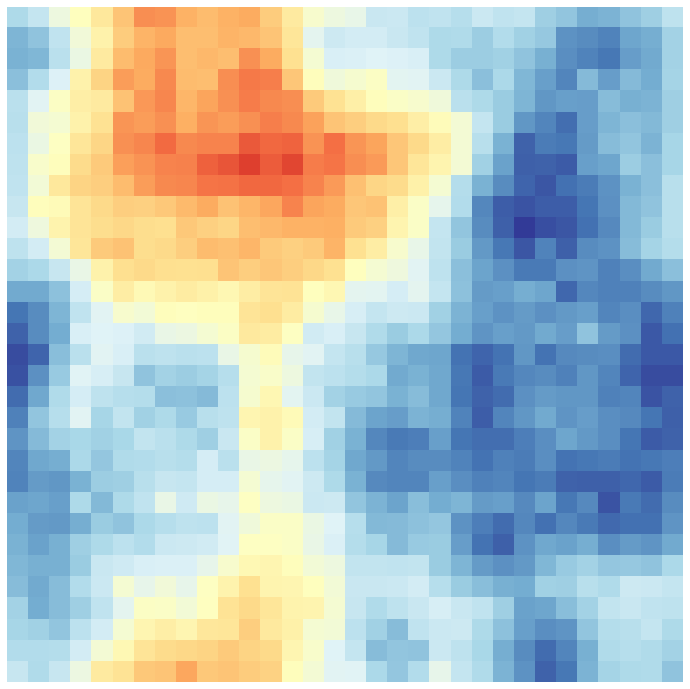}}
    {\includegraphics[width=0.16\textwidth]{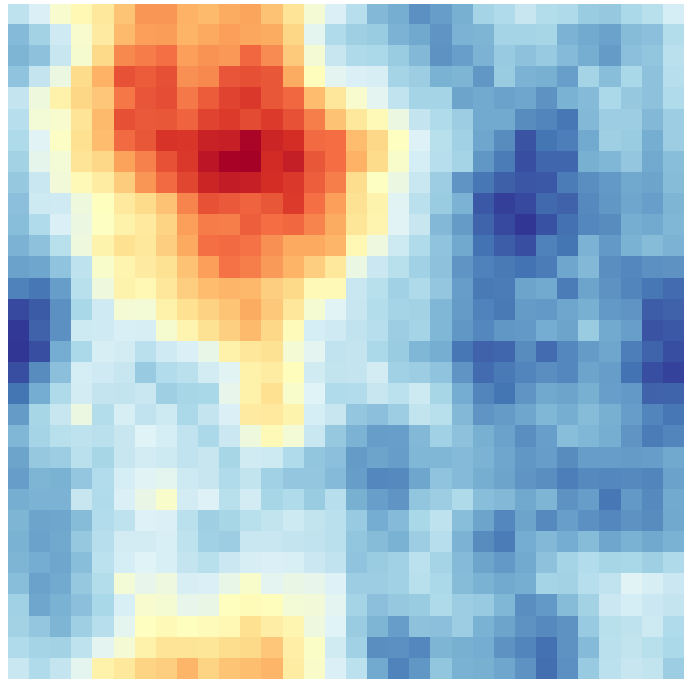}}
    
    \vspace{-0.1cm} 
    
    \subfigure[Iteration 0000]{\includegraphics[width=0.16\textwidth]{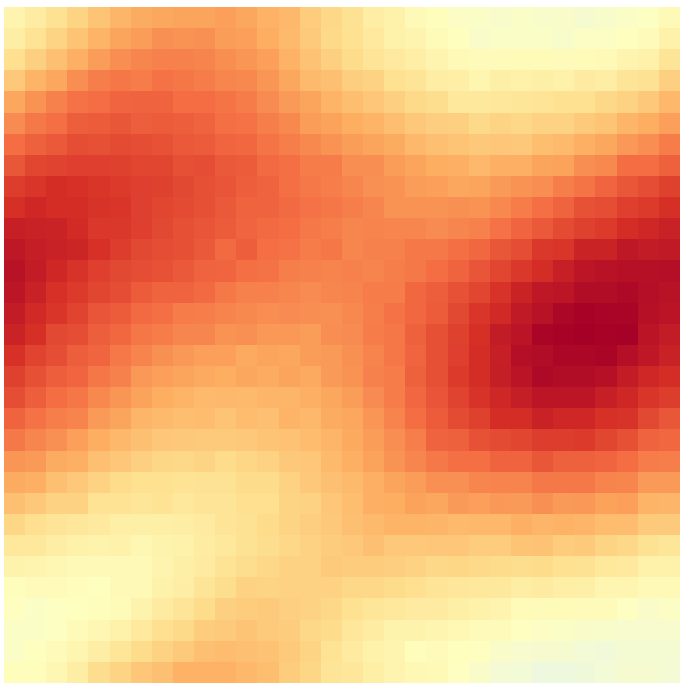}}
    \subfigure[Iteration 0020]{\includegraphics[width=0.16\textwidth]{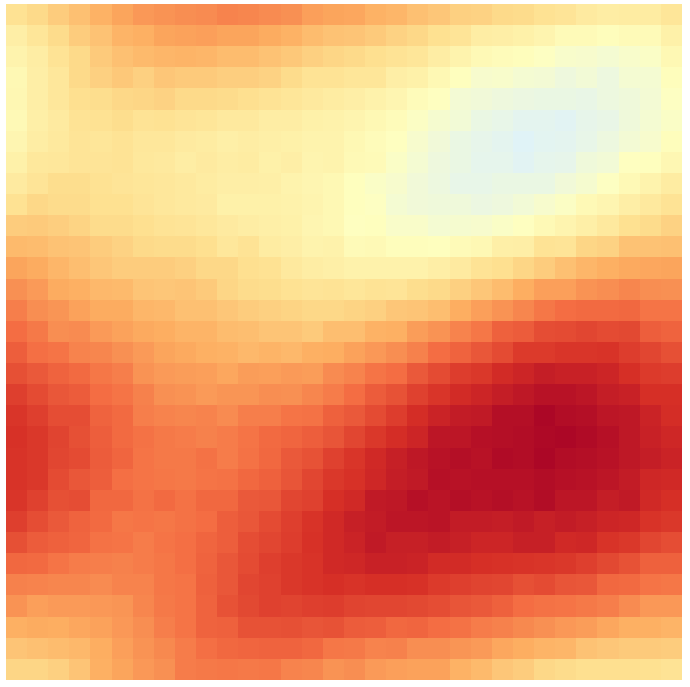}}
    \subfigure[Iteration 0040]{\includegraphics[width=0.16\textwidth]{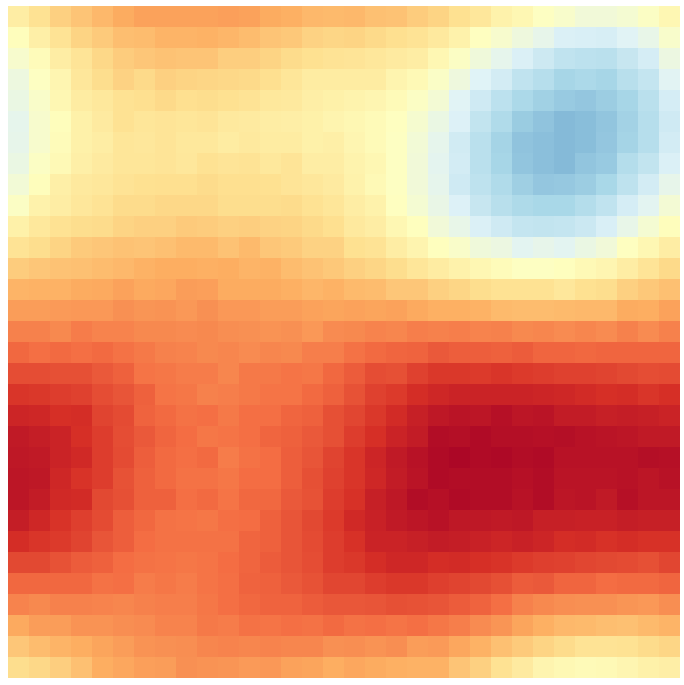}}
    \subfigure[Iteration 0060]{\includegraphics[width=0.16\textwidth]{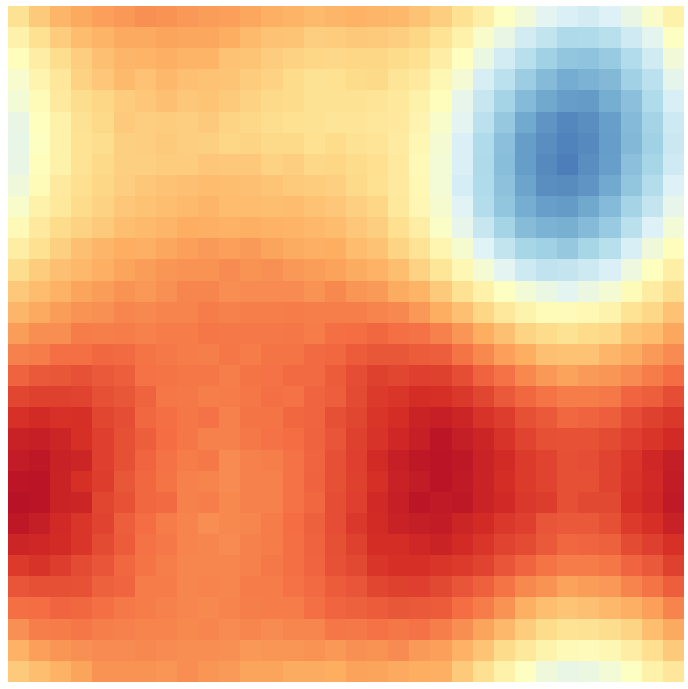}}
    \subfigure[Iteration 0080]{\includegraphics[width=0.16\textwidth]{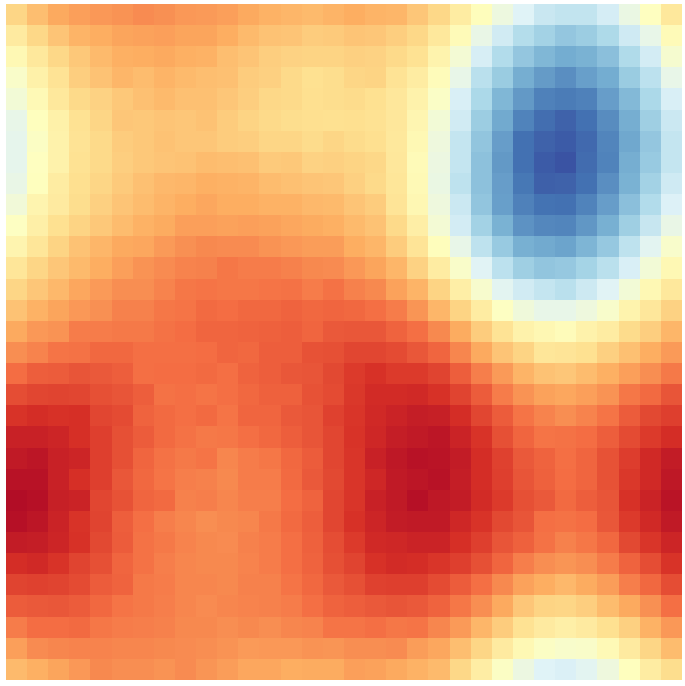}} 
    \subfigure[Ground Truth]{\includegraphics[width=0.16\textwidth]{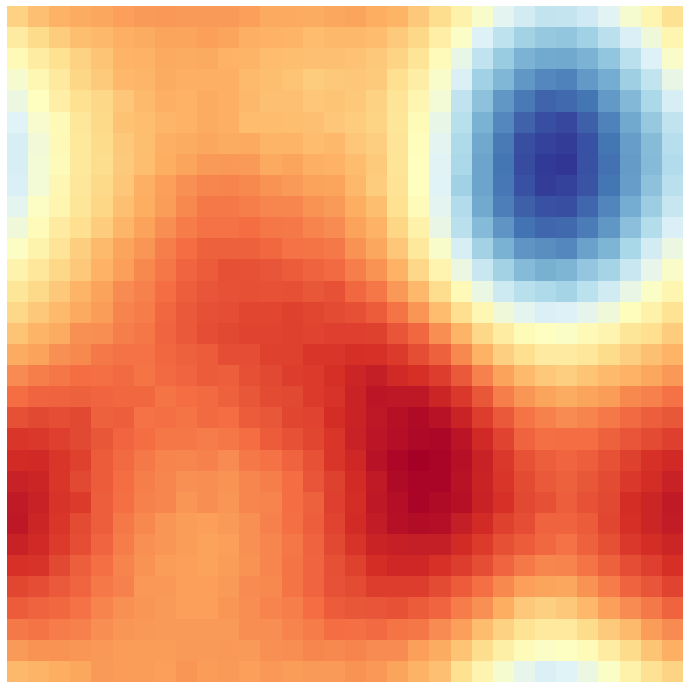}}
    
    \caption{Detailed iterative reconstruction for the inverse N-S problem. Top row: reconstructed initial vorticity. Bottom row: corresponding predicted terminal fields, with the ground-truth terminal field in the lower-right panel.}
    \label{fig:ns_fast_convergence}
\end{figure}

Dedicated measurements on a single NVIDIA GH200 GPU give per-iteration
runtime/peak-memory costs of $1.19$\,s/$4.56$\,GiB for EIT,
$3.40$\,s/$22.5$\,GiB for inverse scattering
(which evaluates $5\times64=320$ scattered fields per iteration), and
$2.93$\,s/$0.84$\,GiB for inverse N-S, corresponding to approximately
$1.0$, $2.8$, and $1.6$ GPU-hours per reconstruction, respectively.
For inverse N-S, runtime scales approximately linearly from $0.52$ to
$6.16$\,s/iteration when increasing the denoising trajectory from
$100$ to $1000$ steps, while peak allocated memory remains below $1$\,GiB.

\subsection{Main Evaluation Results}
We compare our proposed method against several established approaches to evaluate its performance. We evaluate recent neural operator baselines including FNOs \cite{li2020fourier}, PINO \cite{li2024physics}, and DeepONet \cite{lu2019deeponet}. Note that for boundary-value inverse problems such as EIT and inverse scattering, the raw observation $f$ is defined strictly on the domain boundary $\partial \Omega$ rather than as a full spatial image. To directly employ these volumetric neural operators for inversion, we dimensionally lift the boundary data by computing its harmonic extension. Specifically, we solve the source-free Laplace equation $\Delta u = 0$ in $\Omega$ subject to the Dirichlet boundary condition $u|_{\partial \Omega} = f$. The resulting interior field $u$ is then utilized as the spatial input for the neural operators to predict the unknown physical parameter $a$. Furthermore, we include DiffusionPDE \cite{huang2024diffusionpde} as a key baseline.

\begin{figure}[htbp] 
    \centering
    \small
    \setlength{\tabcolsep}{1.2pt}
    
    \begin{tabular}{ccccccc}
        GT & \textbf{Ours} & DiffusionPDE & FNO & DeepONet & PINO & PINN \\ \addlinespace[2pt]
        
        \includegraphics[width=0.135\textwidth]{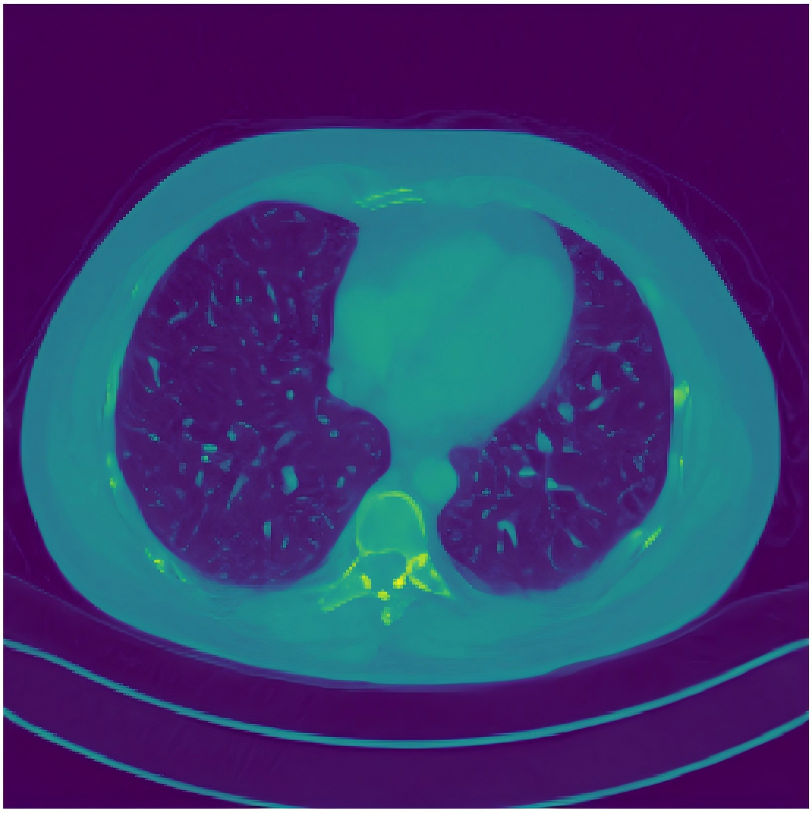} &
        \includegraphics[width=0.135\textwidth]{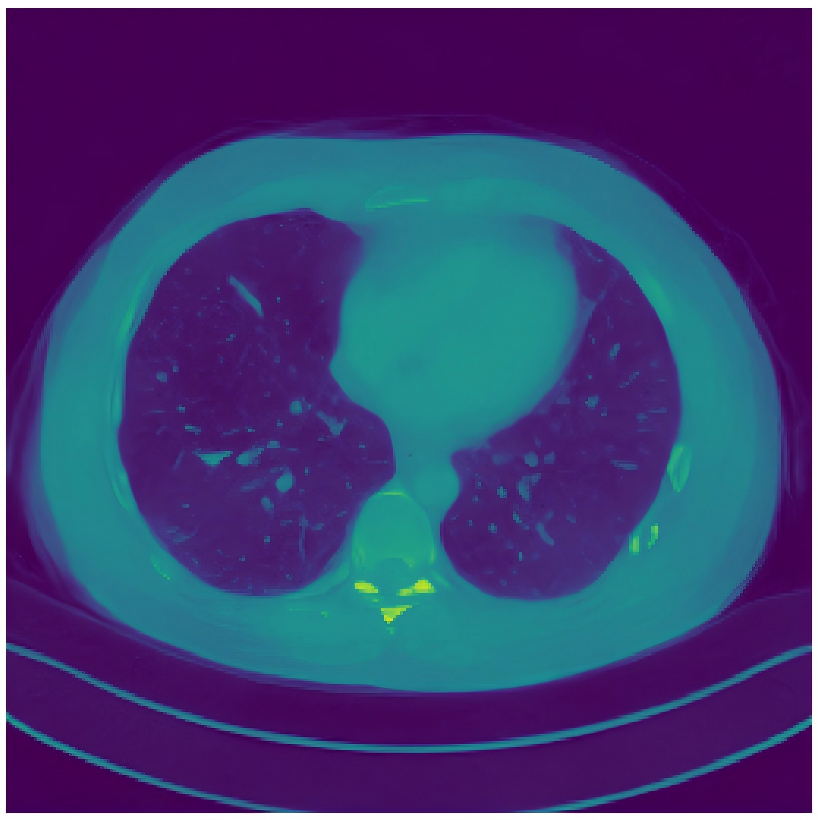} &
        \includegraphics[width=0.135\textwidth]{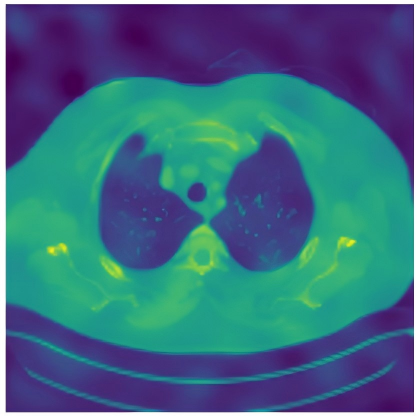} &
        \includegraphics[width=0.135\textwidth]{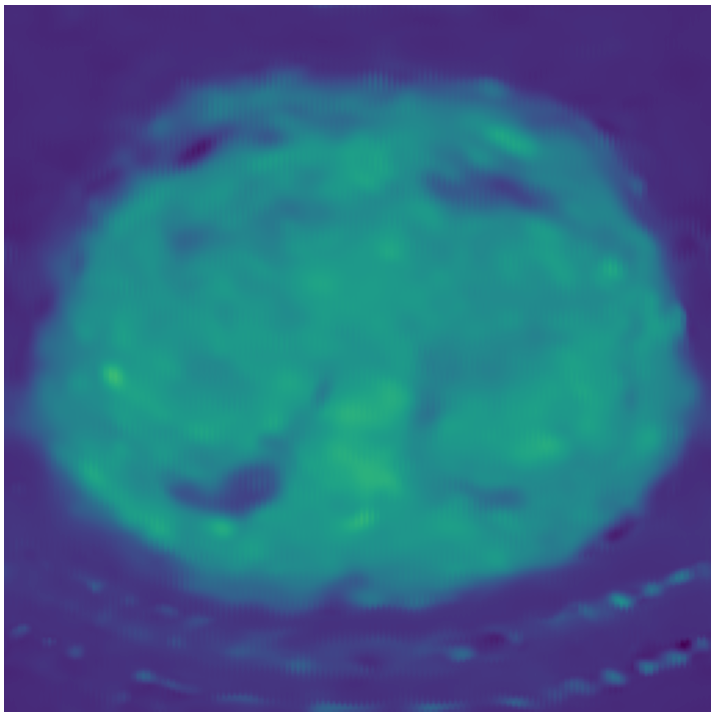} &
        \includegraphics[width=0.135\textwidth]{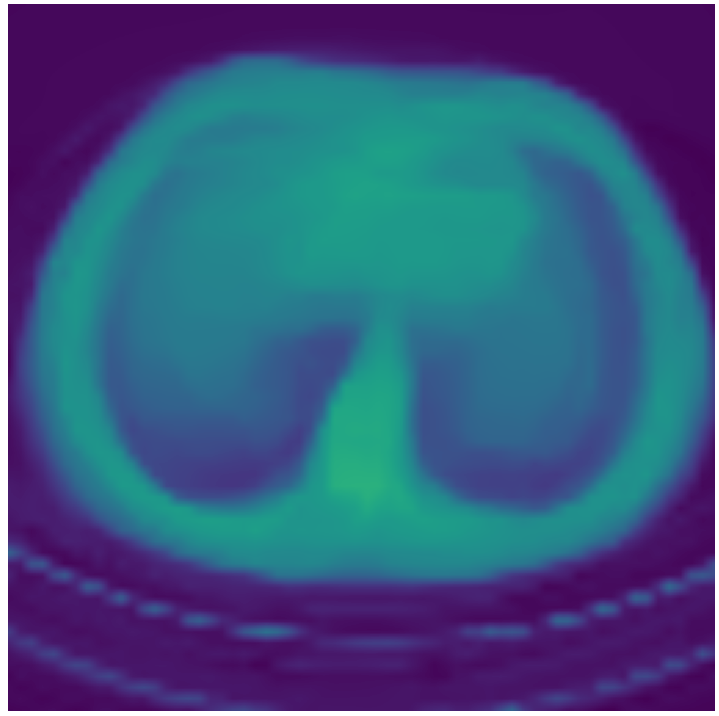} &
        \includegraphics[width=0.135\textwidth]{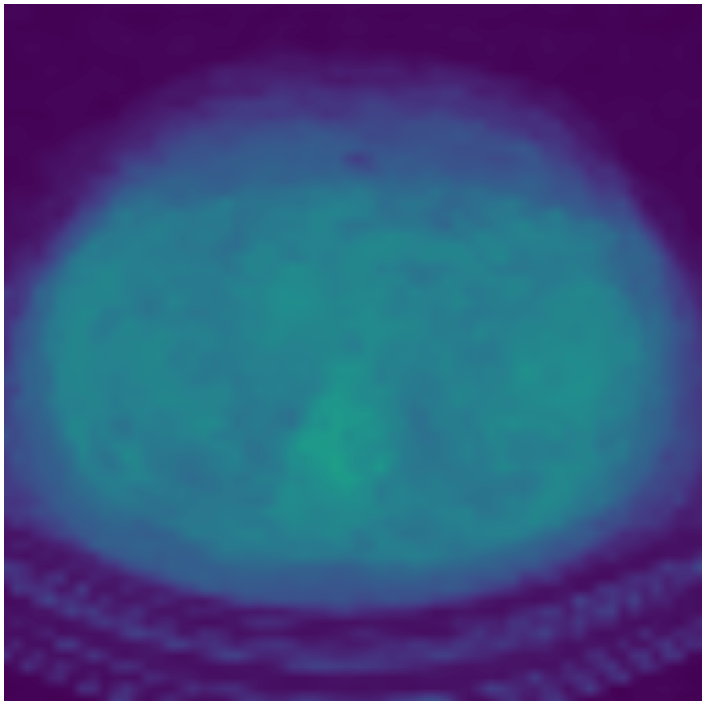} &
        \includegraphics[width=0.135\textwidth]{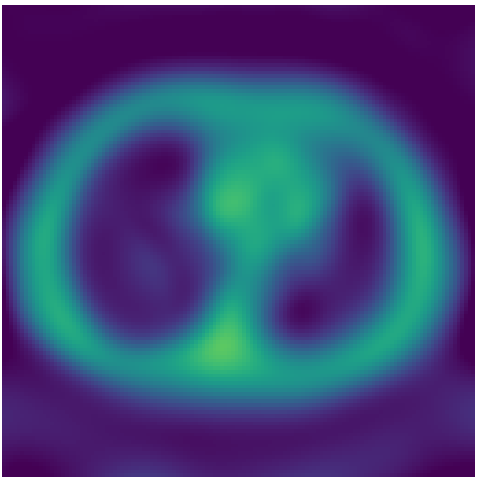} \\
        \multicolumn{7}{c}{\textit{(a) Electrical Impedance Tomography}} \\[8pt] 
        
        \includegraphics[width=0.135\textwidth]{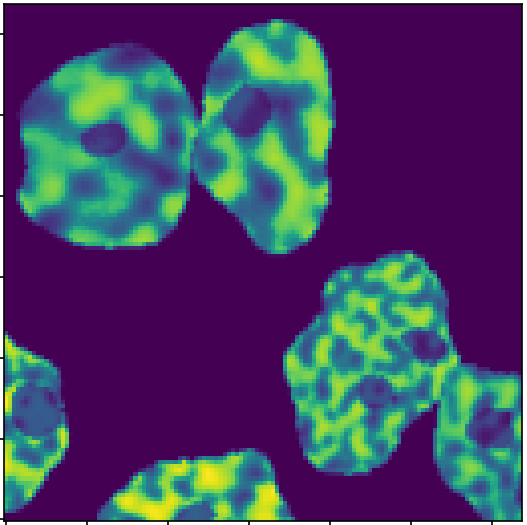} &
        \includegraphics[width=0.135\textwidth]{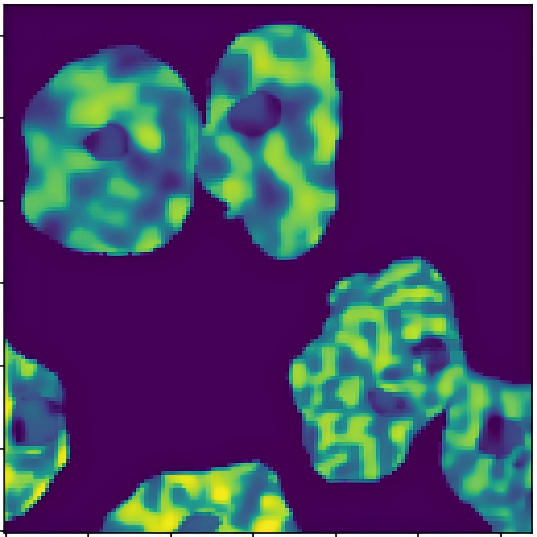} &
        \includegraphics[width=0.135\textwidth]{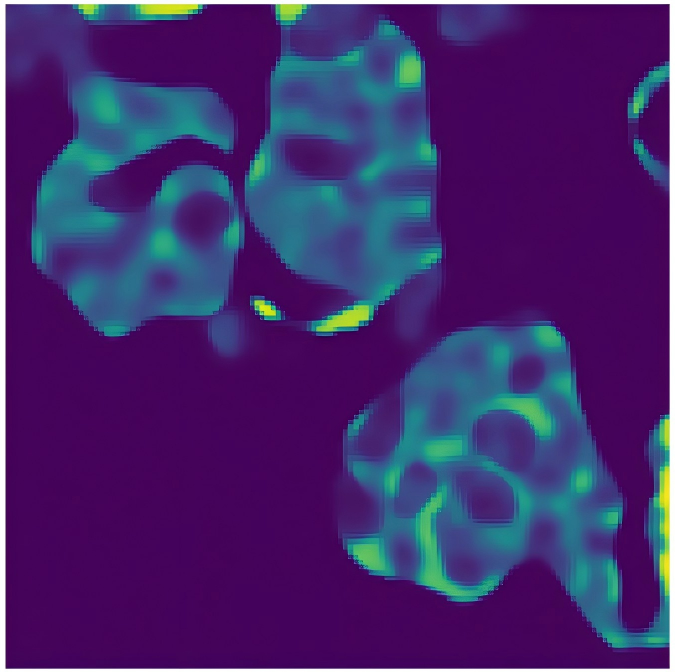} &
        \includegraphics[width=0.135\textwidth]{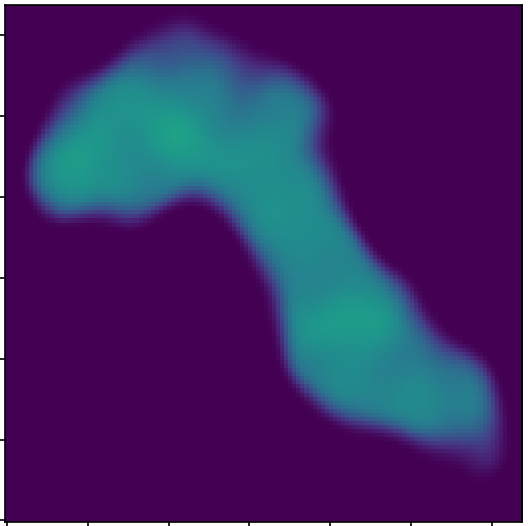} &
        \includegraphics[width=0.135\textwidth]{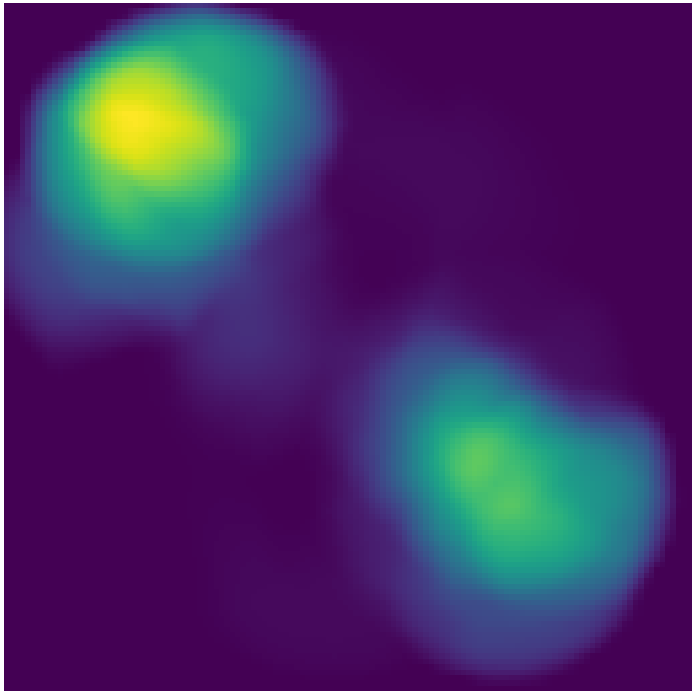} &
        \includegraphics[width=0.135\textwidth]{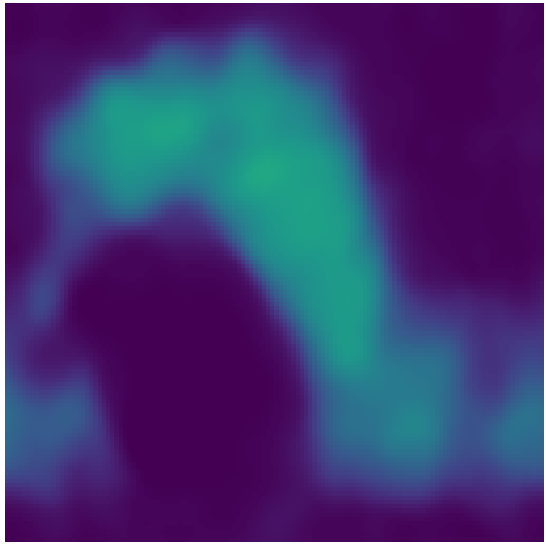} &
        \includegraphics[width=0.135\textwidth]{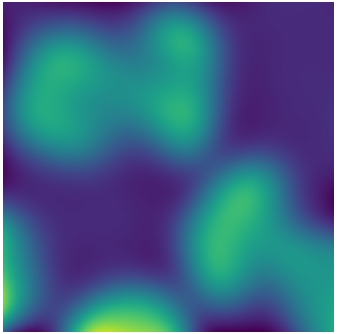} \\
        \multicolumn{7}{c}{\textit{(b) Inverse Scattering}} \\[8pt] 
        
        \includegraphics[width=0.135\textwidth]{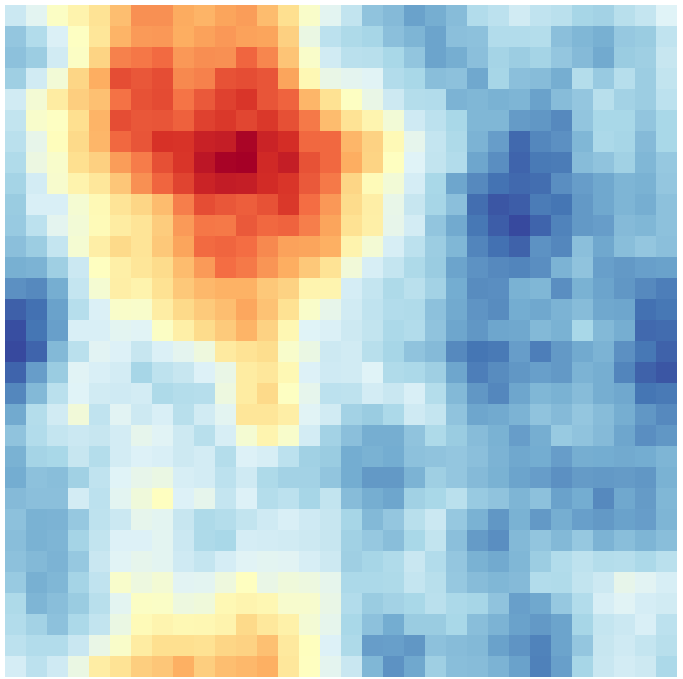} &
        \includegraphics[width=0.135\textwidth]{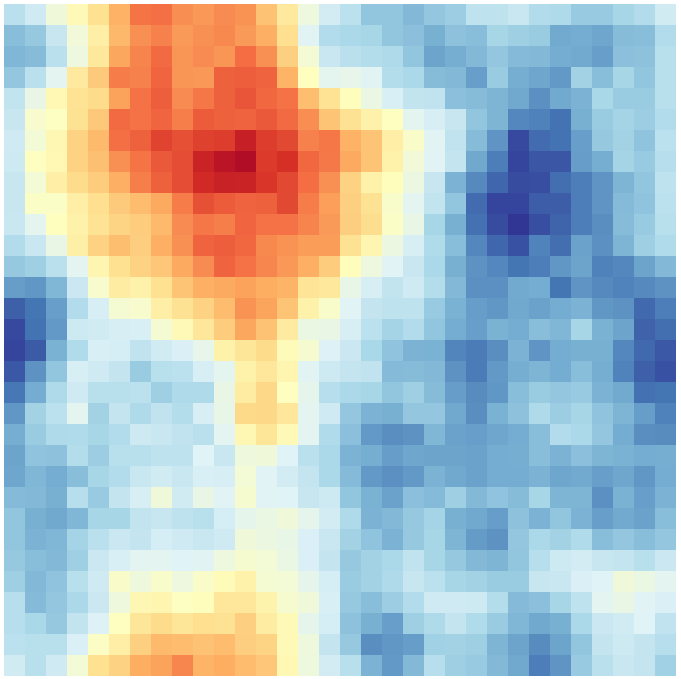} &
        \includegraphics[width=0.135\textwidth]{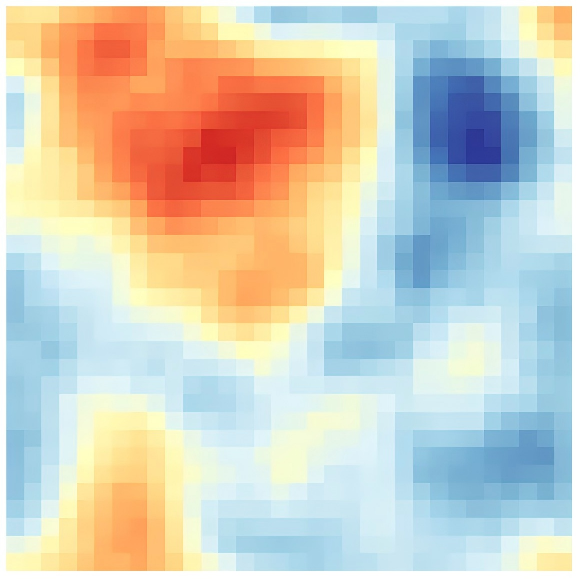} &
        \includegraphics[width=0.135\textwidth]{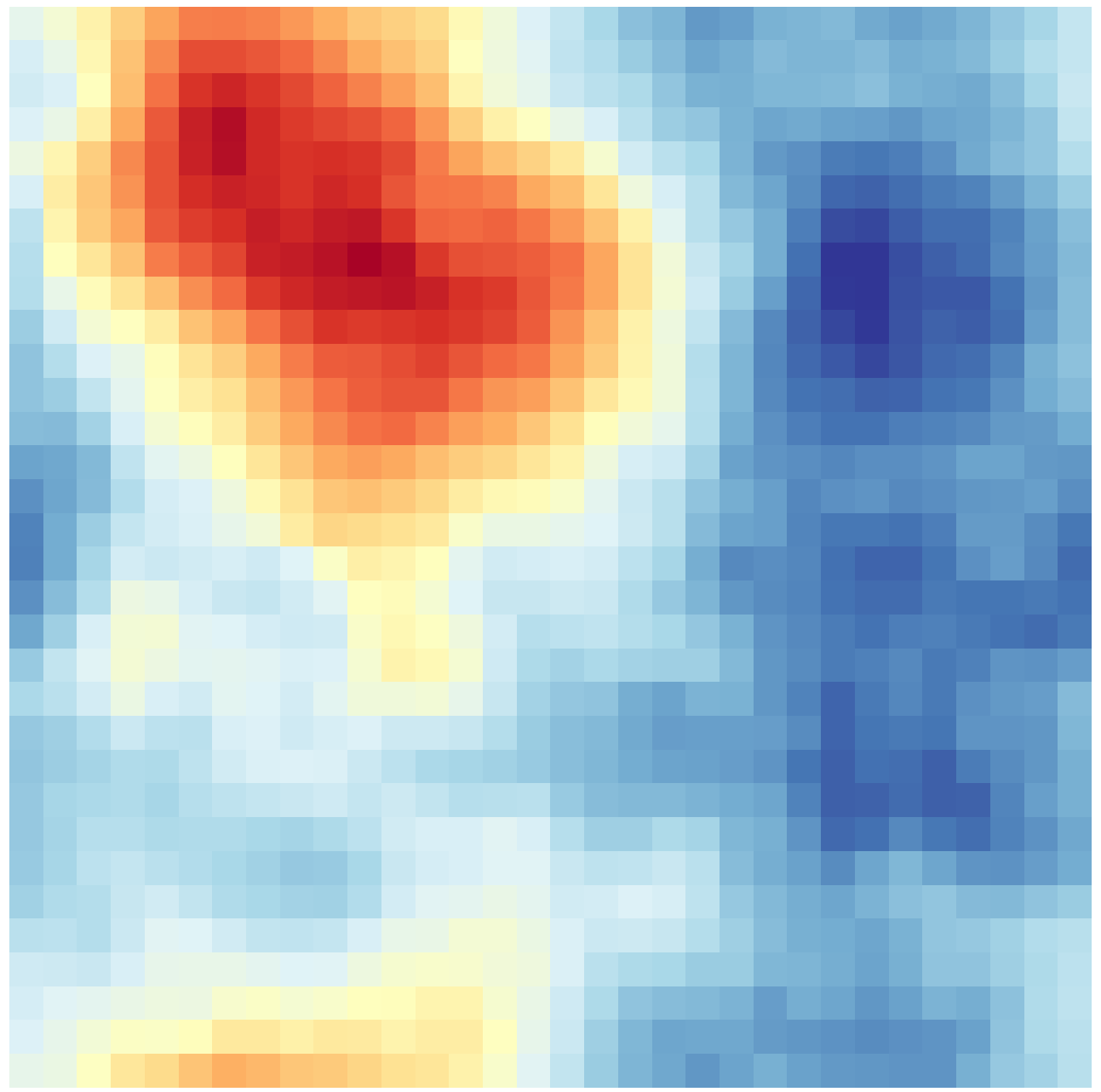} &
        \includegraphics[width=0.135\textwidth]{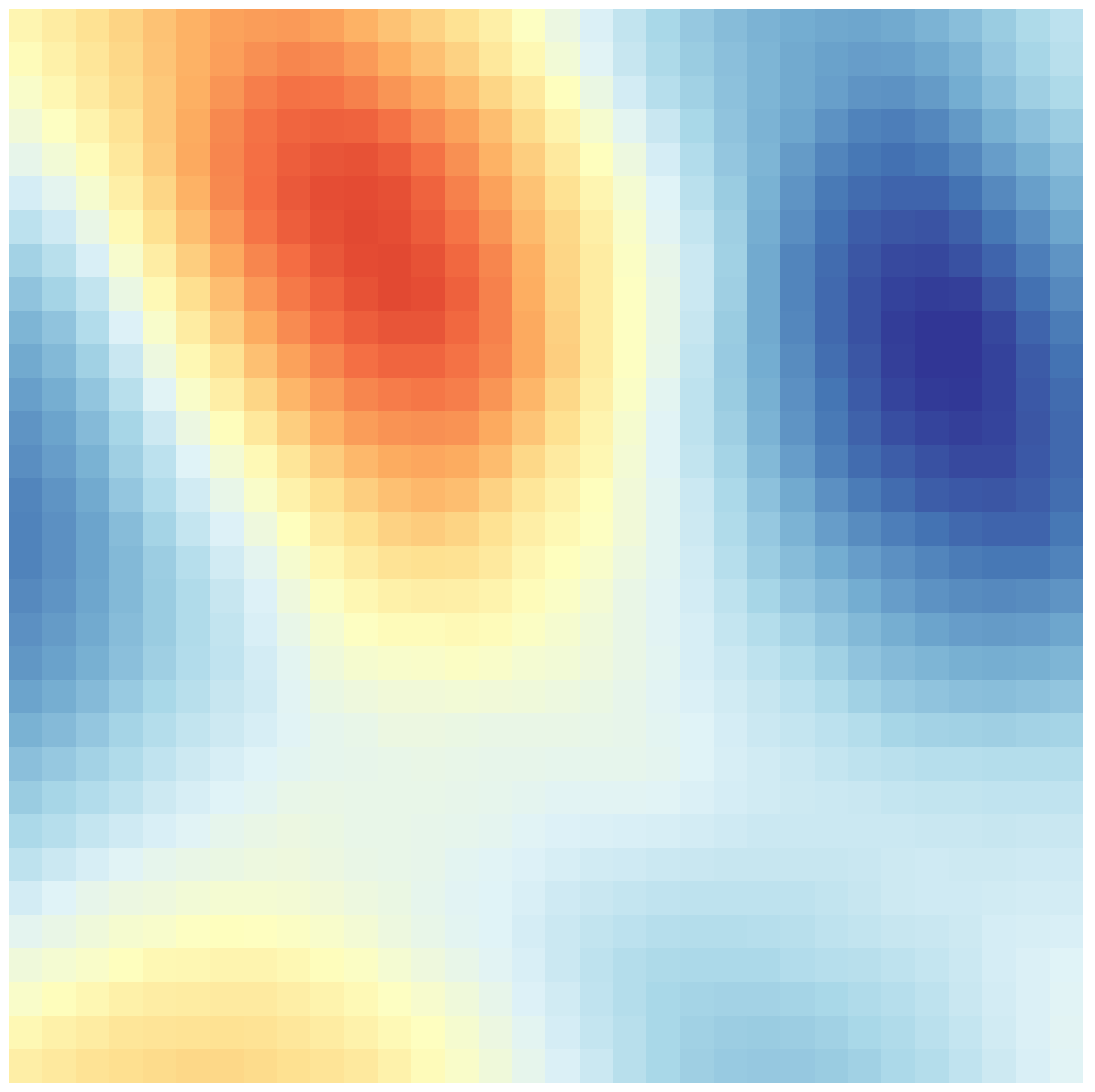} &
        \includegraphics[width=0.135\textwidth]{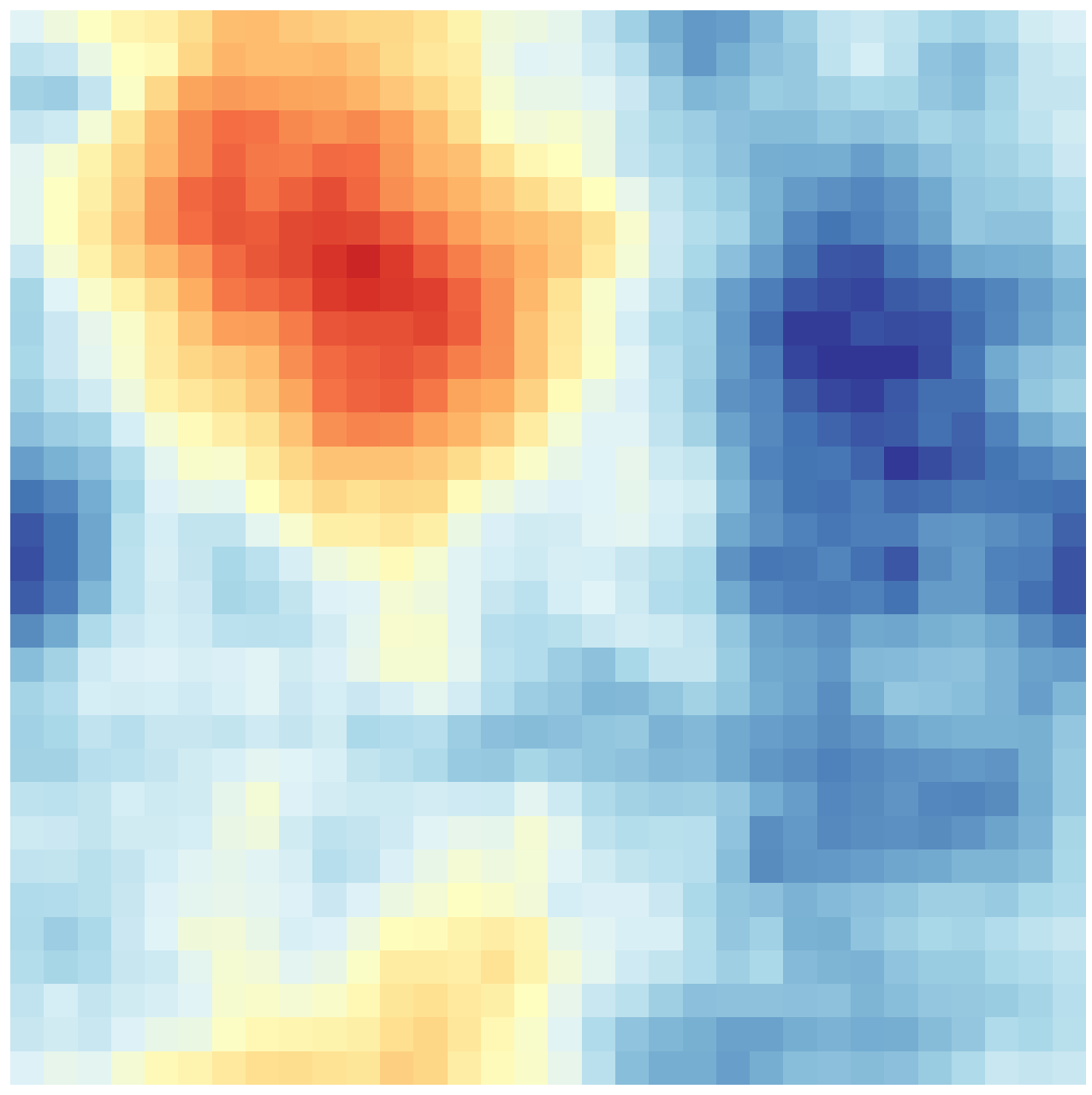} &
        \includegraphics[width=0.135\textwidth]{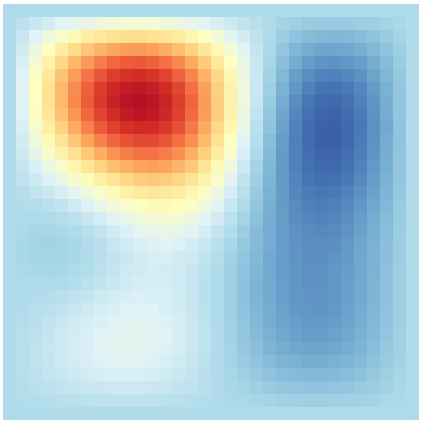} \\
        \multicolumn{7}{c}{\textit{(c) Inverse Navier-Stokes}} \\ 
    \end{tabular}
    
    \caption{
        Qualitative comparison of reconstruction accuracy under clean observation conditions.
        }
    \label{fig:main_comparison}
\end{figure}

The qualitative comparison of our proposed framework against various baselines is illustrated in Figure~\ref{fig:main_comparison}. We visualize the reconstruction results for the EIT, Inverse Scattering, and Inverse N-S problems to highlight the structural fidelity and noise-suppression capabilities of our method.

\begin{figure}[htbp]
    \centering
    \small
    \setlength{\tabcolsep}{1.2pt}
    
    \begin{tabular}{ccccccc}
        GT & \textbf{Ours} & DiffusionPDE & FNO & DeepONet & PINO & PINN \\ \addlinespace[2pt]
        
        \includegraphics[width=0.138\textwidth]{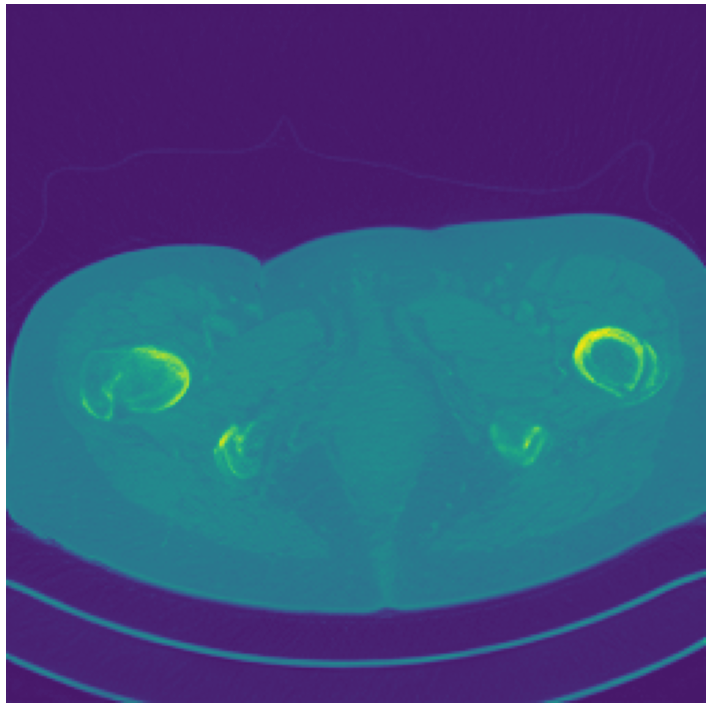} &
        \includegraphics[width=0.138\textwidth]{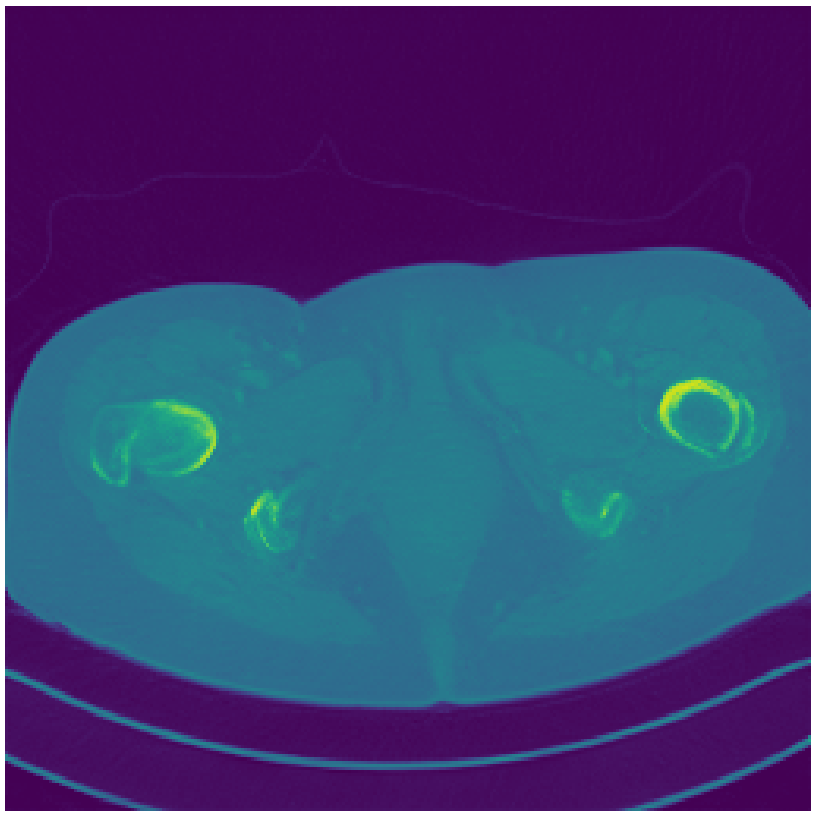} &
        \includegraphics[width=0.138\textwidth]{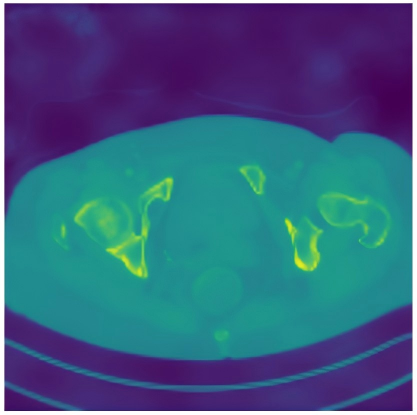} &
        \includegraphics[width=0.138\textwidth]{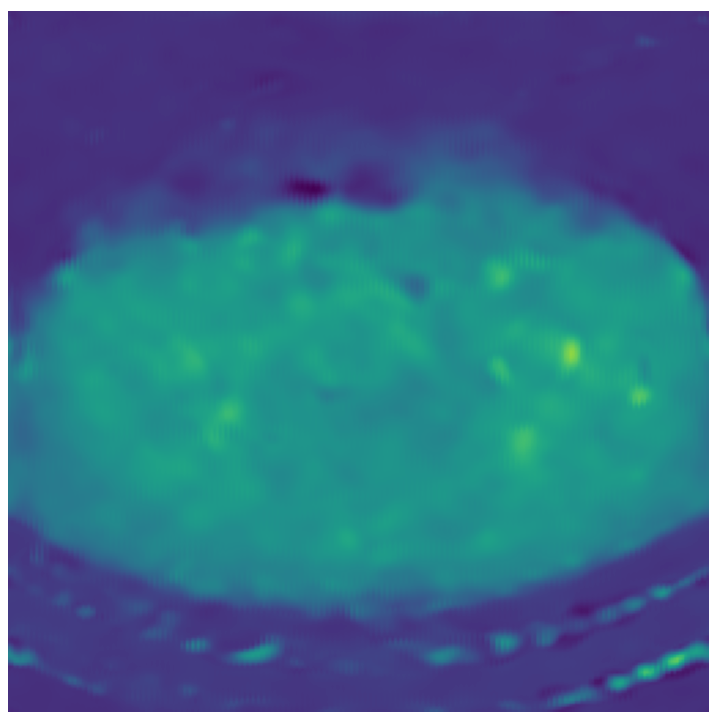} &
        \includegraphics[width=0.138\textwidth]{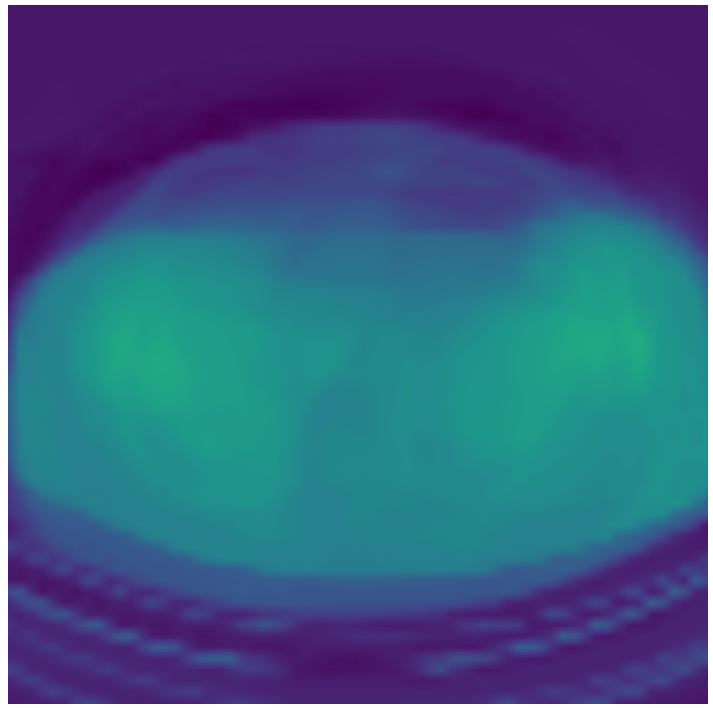} &
        \includegraphics[width=0.138\textwidth]{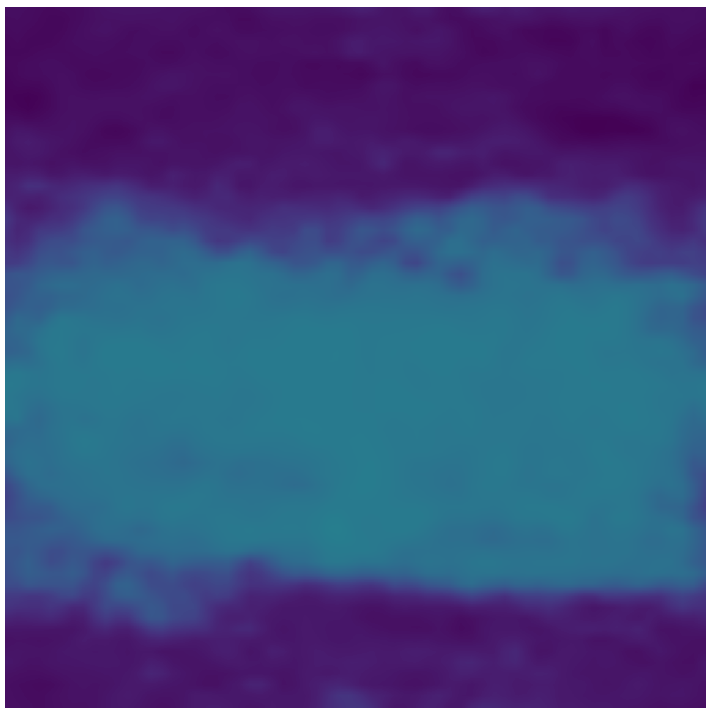} &
        \includegraphics[width=0.138\textwidth]{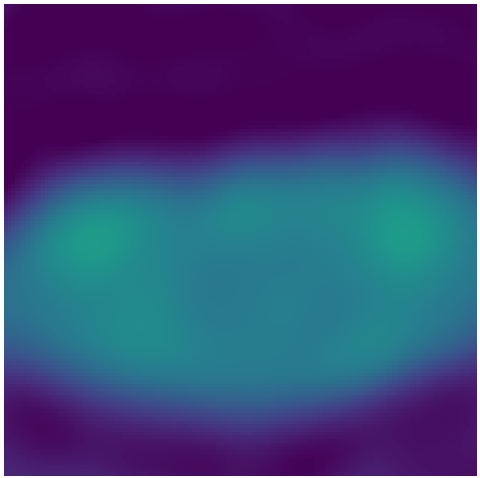} \\
        \multicolumn{7}{c}{\textit{(a) Electrical Impedance Tomography}} \\[8pt] 
        
        \includegraphics[width=0.138\textwidth]{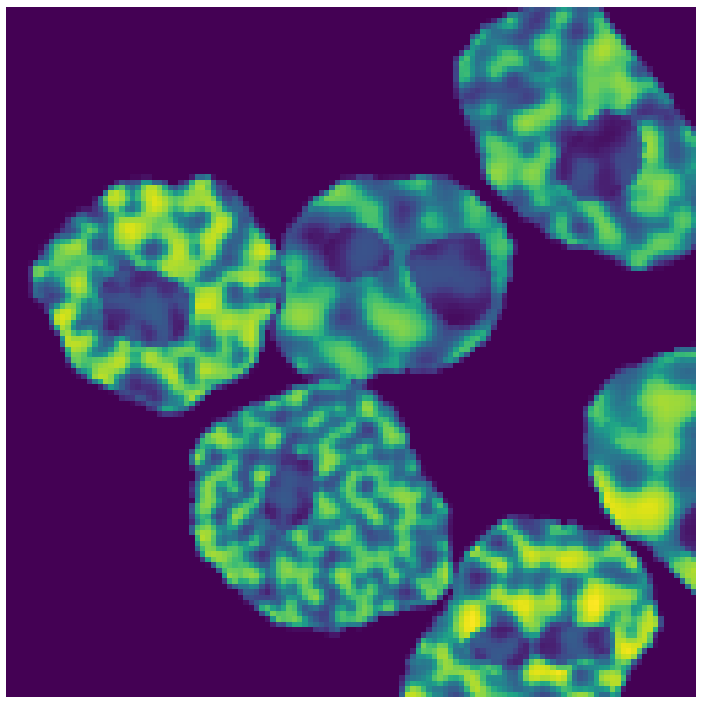} &
        \includegraphics[width=0.138\textwidth]{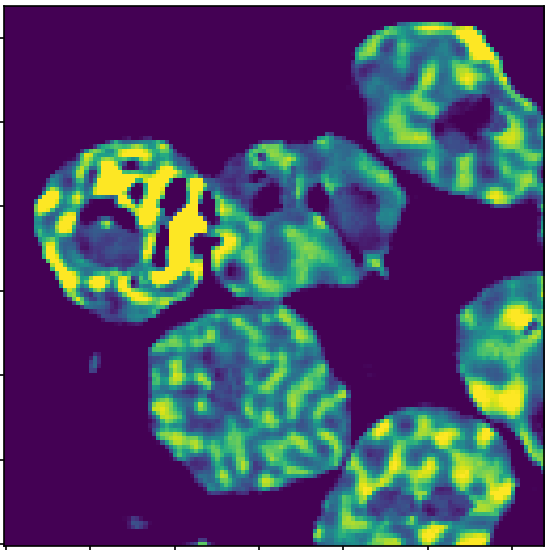} &
        \includegraphics[width=0.138\textwidth]{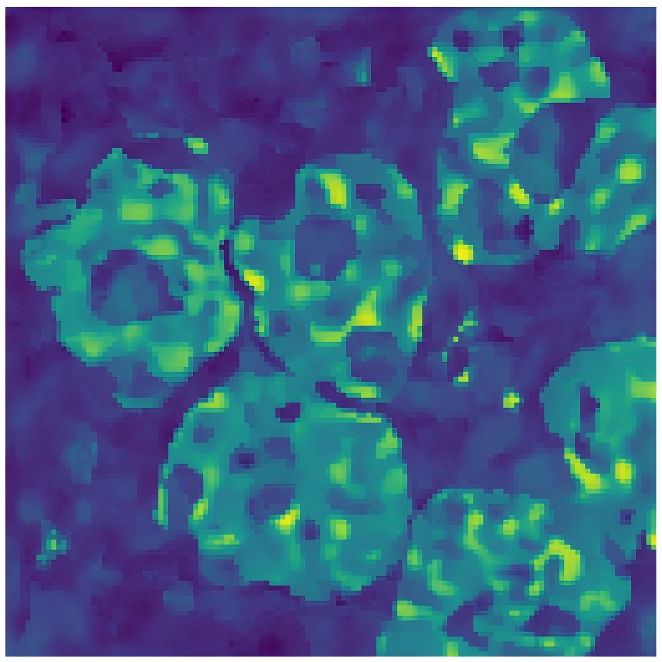} &
        \includegraphics[width=0.138\textwidth]{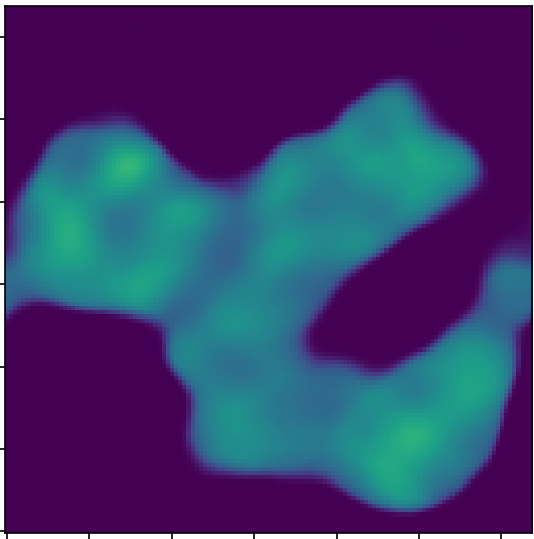} &
        \includegraphics[width=0.138\textwidth]{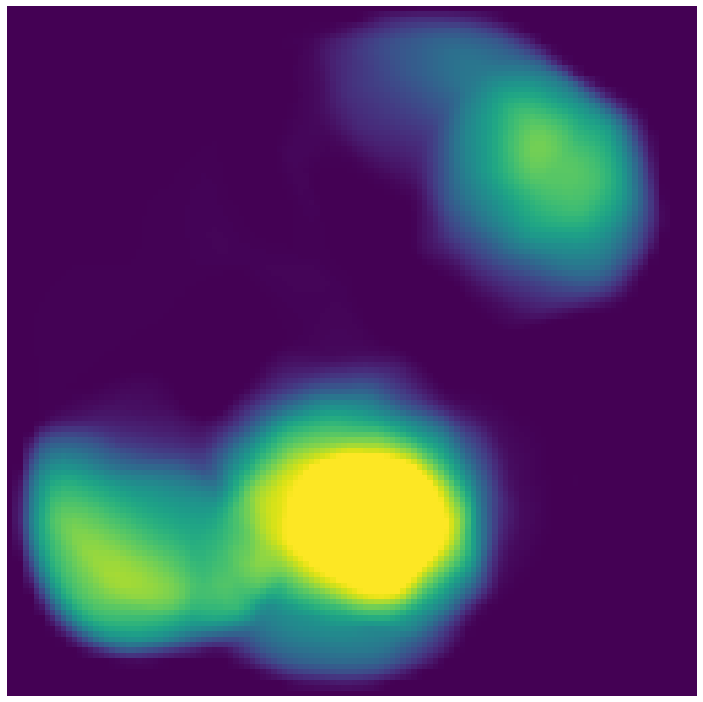} &
        \includegraphics[width=0.138\textwidth]{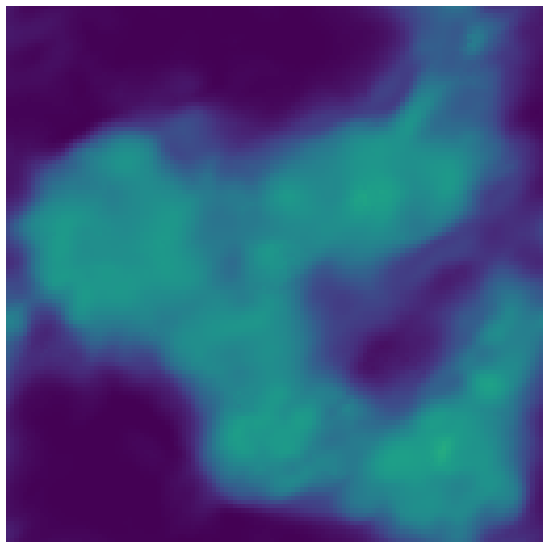} &
        \includegraphics[width=0.138\textwidth]{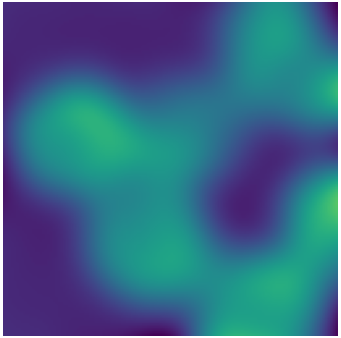} \\
        \multicolumn{7}{c}{\textit{(b) Inverse Scattering}} \\[8pt] 
        
        \includegraphics[width=0.138\textwidth]{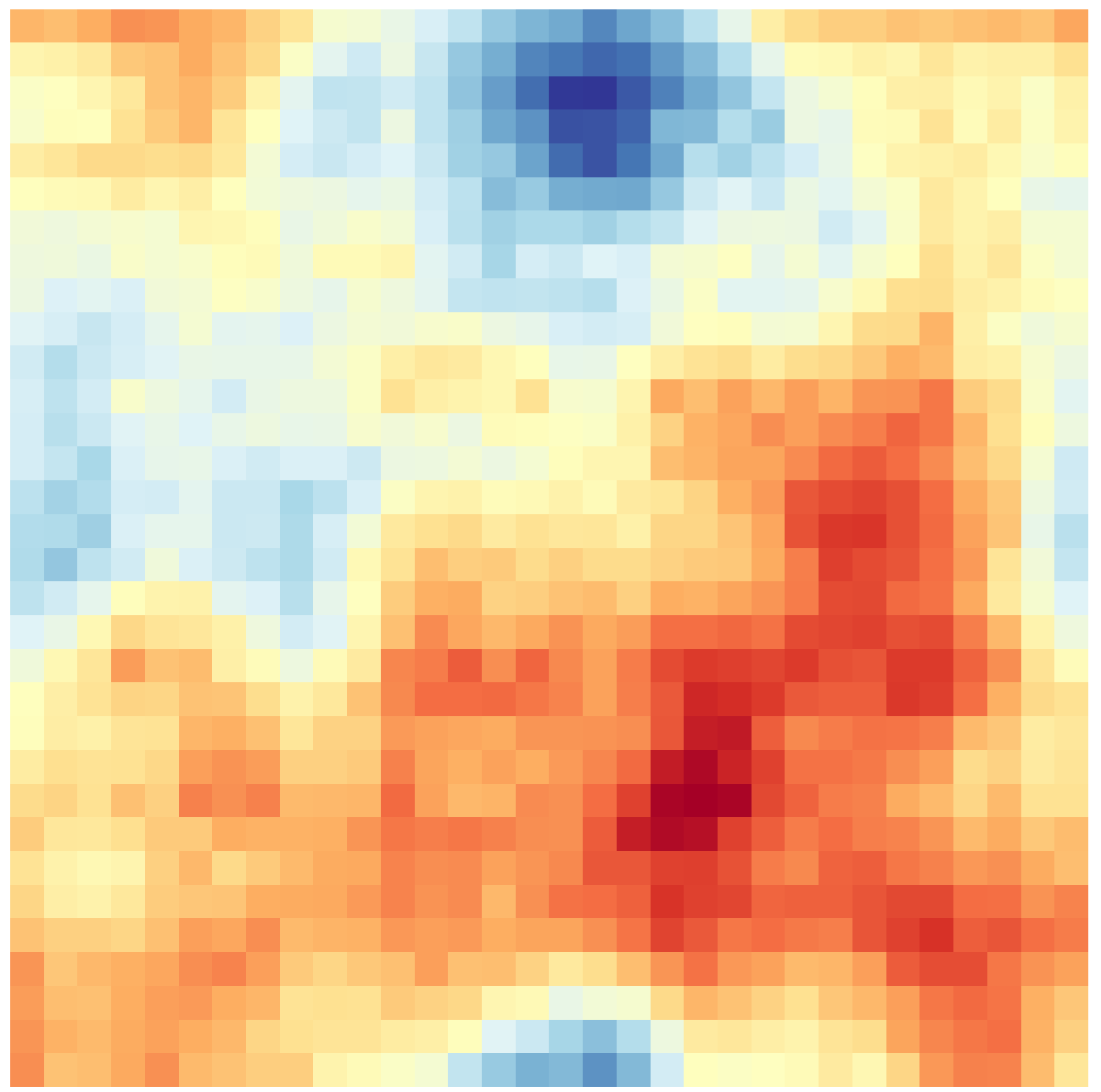} &
        \includegraphics[width=0.138\textwidth]{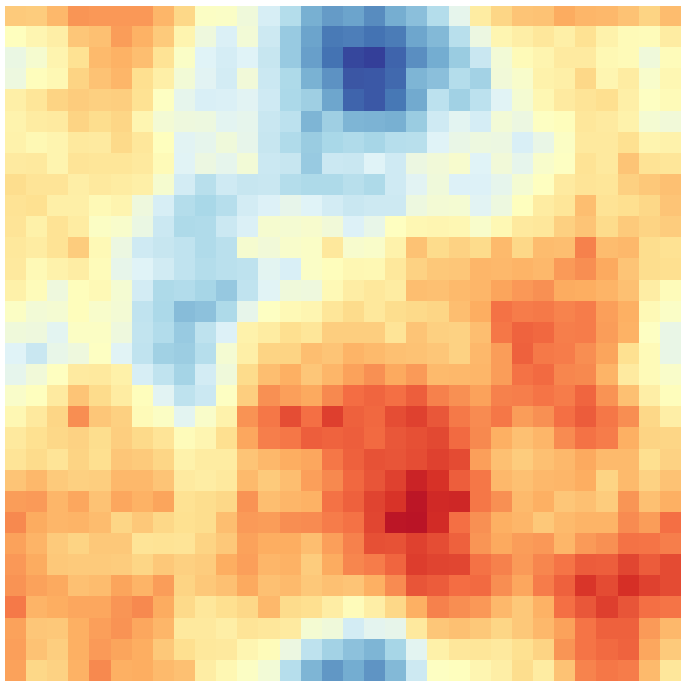} &
        \includegraphics[width=0.138\textwidth]{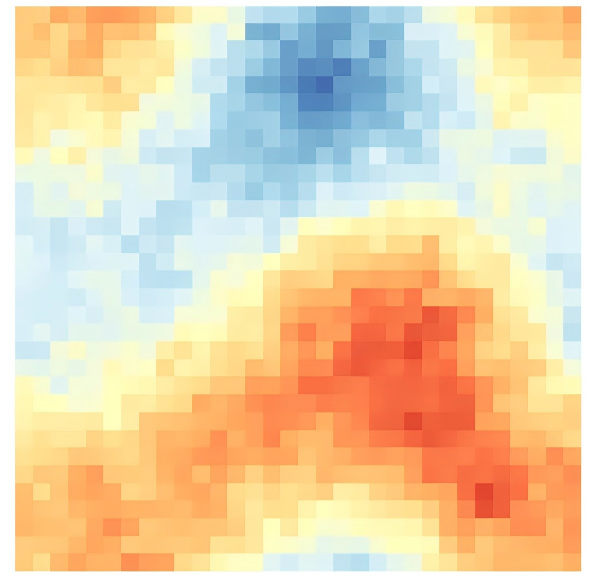} &
        \includegraphics[width=0.138\textwidth]{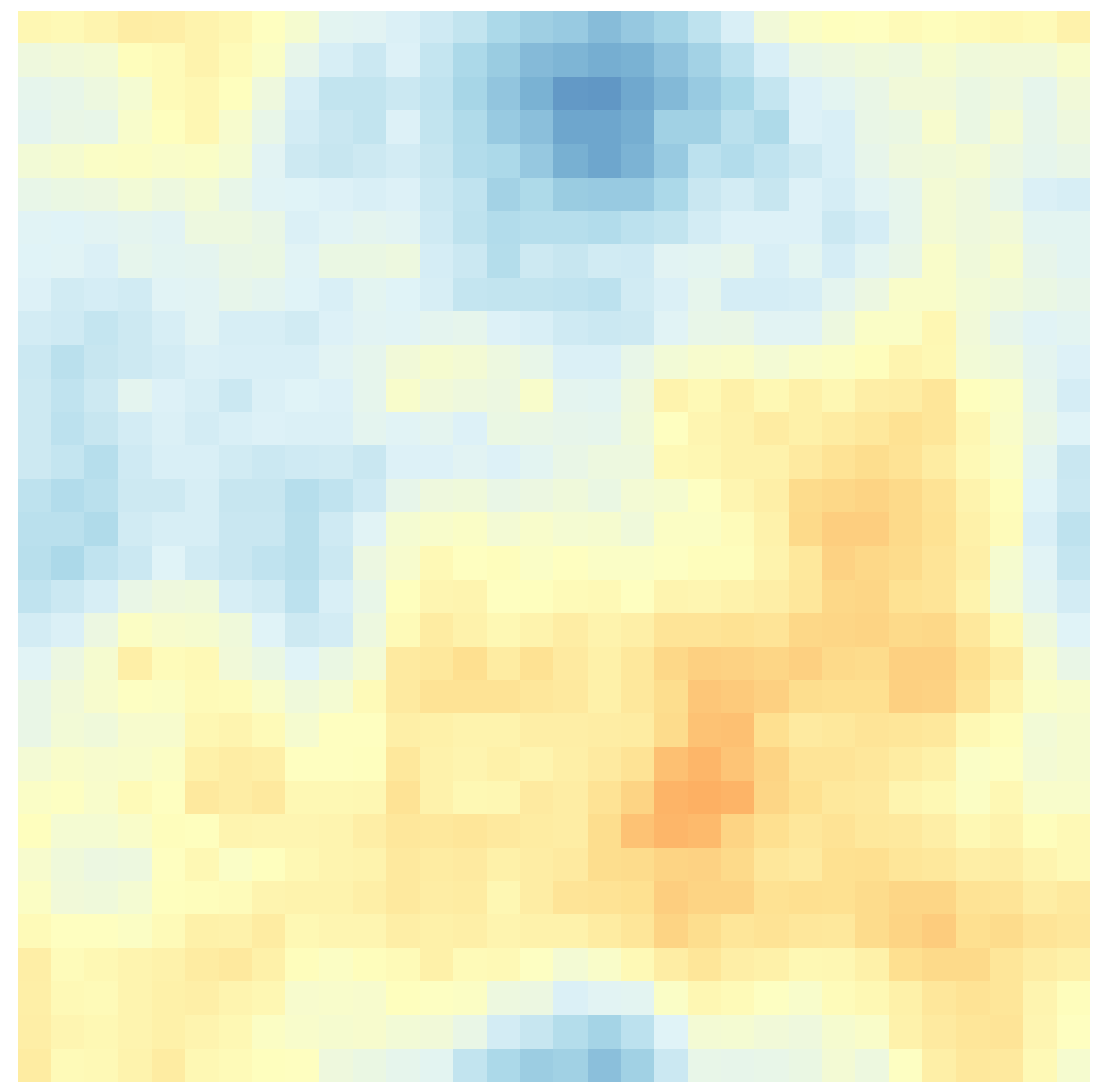} &
        \includegraphics[width=0.138\textwidth]{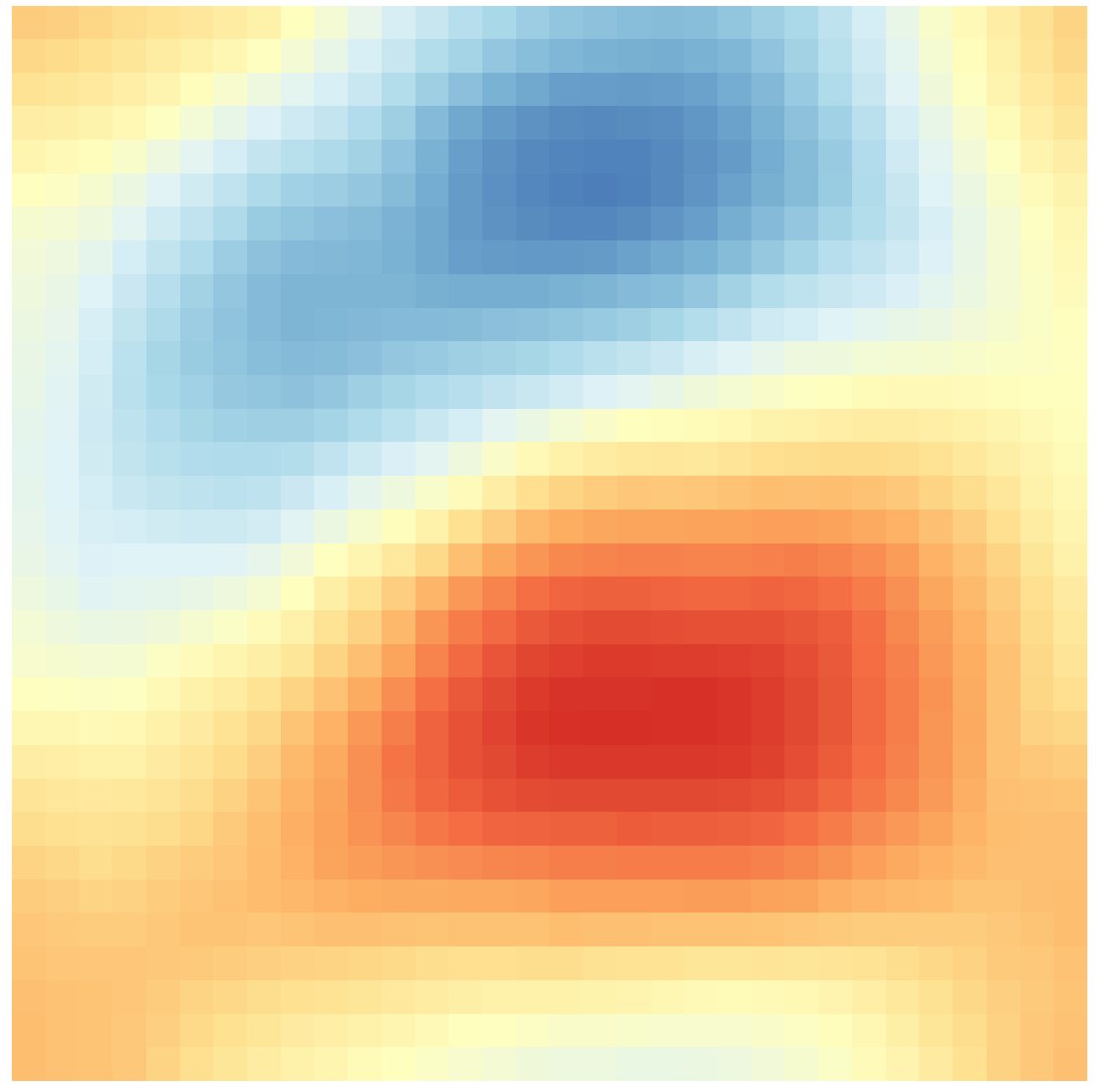} &
        \includegraphics[width=0.138\textwidth]{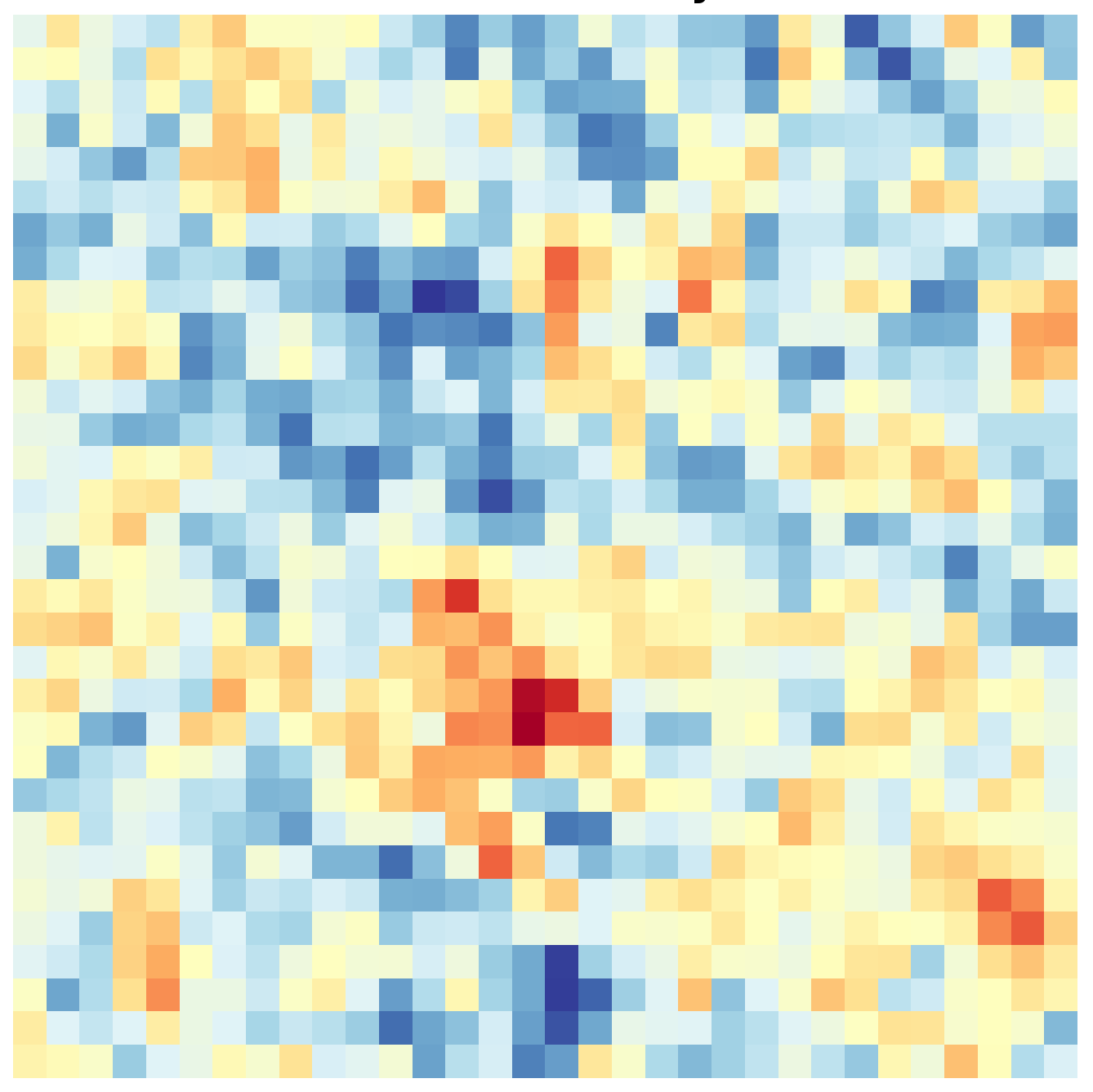} &
        \includegraphics[width=0.138\textwidth]{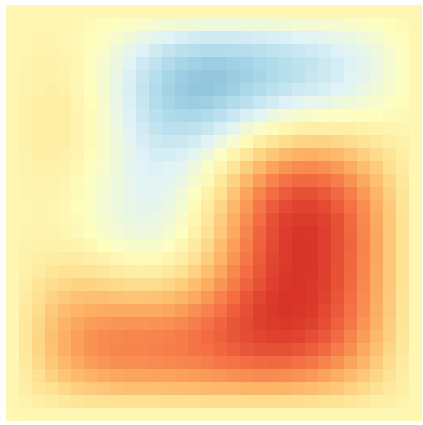} \\
        \multicolumn{7}{c}{\textit{(c) Inverse Navier-Stokes}} \\ 
    \end{tabular}
    
    \caption{Qualitative comparison of reconstruction robustness under $50\%$ Gaussian noise.}
    \label{fig:noise_comparison}
\end{figure}

For the noisy setting, we define
\begin{equation}
y_{\mathrm{noisy}}=y_{\mathrm{clean}}+\gamma\,\mathrm{std}(y_{\mathrm{clean}})\,\epsilon,
\end{equation}
where $\gamma=0.50$, $\epsilon\sim\mathcal N(0,I)$ has the same shape as the observation, and $\mathrm{std}(\cdot)$ is computed over the entries of the clean observation. All methods are evaluated on $N=25$ test instances under both clean and $50\%$ Gaussian-noise conditions. For DiLO, three initialization outcomes are averaged for each target, and Table~\ref{tab:quantitative} reports the resulting MAE as mean $\pm$ standard deviation over the 25 targets. Figure~\ref{fig:noise_comparison} shows the corresponding qualitative results under $50\%$ Gaussian noise.

\begin{table*}[htbp]
\centering
\caption{
Quantitative comparison of reconstruction accuracy (MAE $\downarrow$)
across EIT, inverse scattering, and inverse Navier--Stokes under clean
and $50\%$ Gaussian-noise conditions.
All methods are evaluated on $N=25$ test instances, and results are
reported as mean $\pm$ standard deviation.
}
\label{tab:quantitative}

\small
\setlength{\tabcolsep}{5pt}
\begin{tabular}{lcccccc}
\toprule
& \multicolumn{2}{c}{EIT}
& \multicolumn{2}{c}{Inverse Scattering}
& \multicolumn{2}{c}{Inverse N-S} \\
\cmidrule(lr){2-3}
\cmidrule(lr){4-5}
\cmidrule(lr){6-7}

Method
& Clean
& 50\% Noise
& Clean
& 50\% Noise
& Clean
& 50\% Noise \\
\midrule

DeepONet
& $0.070 \pm 0.027$ & $0.072 \pm 0.032$
& $0.126 \pm 0.057$ & $0.125 \pm 0.057$
& $0.050 \pm 0.007$ & $0.050 \pm 0.007$ \\

FNO
& $0.071 \pm 0.023$ & $0.080 \pm 0.020$
& $0.103 \pm 0.047$ & $0.104 \pm 0.048$
& $0.042 \pm 0.006$ & $0.103 \pm 0.016$ \\

PINNs
& $0.033 \pm 0.006$ & $0.037 \pm 0.007$
& $0.105 \pm 0.040$ & $0.106 \pm 0.033$
& $0.040 \pm 0.006$ & $0.042 \pm 0.006$ \\

PINO
& $0.087 \pm 0.022$ & $0.088 \pm 0.022$
& $0.117 \pm 0.048$ & $0.119 \pm 0.050$
& $0.034 \pm 0.003$ & $0.253 \pm 0.066$ \\

DiffusionPDE
& $0.130 \pm 0.037$ & $0.133 \pm 0.033$
& $0.636 \pm 0.042$ & $0.884 \pm 0.025$
& $0.061 \pm 0.017$ & $0.214 \pm 0.087$ \\

\midrule

Ours
& $\mathbf{0.018 \pm 0.006}$ & $\mathbf{0.041 \pm 0.019}$
& $\mathbf{0.030 \pm 0.042}$ & $\mathbf{0.033 \pm 0.043}$
& $\mathbf{0.015 \pm 0.004}$ & $\mathbf{0.023 \pm 0.005}$ \\

\bottomrule
\end{tabular}

\end{table*}

\subsection{Ablation Studies}
\label{sec:ablation}

To examine the main design choices in DiLO, we perform three ablations on inverse Navier--Stokes with $50\%$ Gaussian observation noise. Unless noted otherwise, results are reported as mean $\pm$ standard deviation over five targets (seeds 999--1003). All variants use $2{,}000$ optimization iterations and a $500$-step denoising trajectory. Reconstruction error is measured by the MAE of the recovered initial vorticity.

\begin{table}[htbp]
    \centering
    \caption{Ablation study of the main design choices of DiLO on inverse N-S under $50\%$ Gaussian observation noise. Lower MAE is better.}
    \label{tab:ablation_r1}
    \small
    \begin{tabular}{llcc}
        \toprule
        Ablation & Setting & Fwd. Dev. vs. Ref. & MAE $\downarrow$ \\
        \midrule
        \multirow{2}{*}{Opt. space}
        & Latent (DiLO) & -- & $\mathbf{0.0217\pm0.0068}$ \\
        & Ambient & -- & $0.3713\pm0.1802$ \\
        \midrule
        \multirow{2}{*}{Trajectory}
        & Deterministic & -- & $\mathbf{0.0267\pm0.0028}$ \\
        & Stochastic & -- & \textit{No convergence within budget} \\
        \midrule
        \multirow{5}{*}{Matched surrogate}
        & Original & $0\%$ & $0.0217\pm0.0068$ \\
        & $k=6$ & $4.4\%$ & $0.0207\pm0.0064$ \\
        & $k=5$ & $8.4\%$ & $0.0206\pm0.0067$ \\
        & $k=4$ & $12.8\%$ & $0.0199\pm0.0072$ \\
        & $k=3$ & $\sim100\%$ & $0.0241\pm0.0072$ \\
        \midrule
        \multirow{5}{*}{Mismatched surrogate}
        & Original & $0\%$ & $\mathbf{0.0217\pm0.0068}$ \\
        & $k=6$ & $4.4\%$ & $0.0209\pm0.0078$ \\
        & $k=5$ & $8.4\%$ & $0.0229\pm0.0106$ \\
        & $k=4$ & $12.8\%$ & $0.0268\pm0.0167$ \\
        & $k=3$ & $\sim100\%$ & $0.1691\pm0.0824$ \\
        \bottomrule
    \end{tabular}
\end{table}

\paragraph{Latent-space optimization}
We compare DiLO with direct optimization of the $32\times32$ vorticity field, using the same decoded field for initialization. Direct ambient-space optimization gives a considerably larger reconstruction error. We also vary its learning rate on target 999, but the gap remains. These results suggest that the learned generative parameterization provides useful regularization during inversion.

\paragraph{Deterministic trajectory}
We compare deterministic and stochastic denoising while keeping the remaining optimization settings fixed. Under stochastic denoising, the measurement loss shows little meaningful descent within the same optimization budget, and we therefore do not report a converged reconstruction MAE for this variant.

\paragraph{Sensitivity to the forward surrogate}
We construct degraded versions of the N-S FNO by retaining a centered $2k\times2k$ spatial Fourier window and the lowest $k$ temporal modes, with $k=6,5,4,3$. The forward deviations in Table~\ref{tab:ablation_r1} are measured relative to the original FNO.

We consider two protocols. In the matched case, the same degraded surrogate is used to generate the observations and to perform inversion; this mainly tests self-consistency of the degraded model. In the mismatched case, observations are generated by the original FNO and inversion is carried out with a degraded surrogate. Reconstruction changes little for the tested moderate mismatch levels, while the $k=3$ case leads to a clear increase in error. Thus, in this spectral-truncation experiment, performance is relatively insensitive to moderate forward-model mismatch but deteriorates when the mismatch becomes severe. This controlled study measures mismatch relative to the original FNO and does not quantify the absolute surrogate error with respect to the numerical PDE solver.

\subsection{More Experiment Results}
Additional reconstruction examples for the three inverse problems are presented below.

\paragraph{EIT} 
To comprehensively evaluate the reconstruction performance, Figures \ref{fig:eit2_combined} to \ref{fig:eit6_combined} share a consistent layout: subfigures (a) and (b) display the predicted reconstruction result and its absolute error map at the initial state (Iteration 0); (c) and (d) show the finalized reconstruction result and the minimized error map at the final iteration; (e) represents the ground truth (GT) distribution; and (f) plots the log-scale loss curve demonstrating model convergence.

\begin{figure}[!htbp]
    \centering
    \begin{minipage}{\textwidth}
        \centering
        \begin{minipage}[c]{0.17\linewidth}
            \centering
            \subfigure[]{\includegraphics[width=0.75\linewidth]{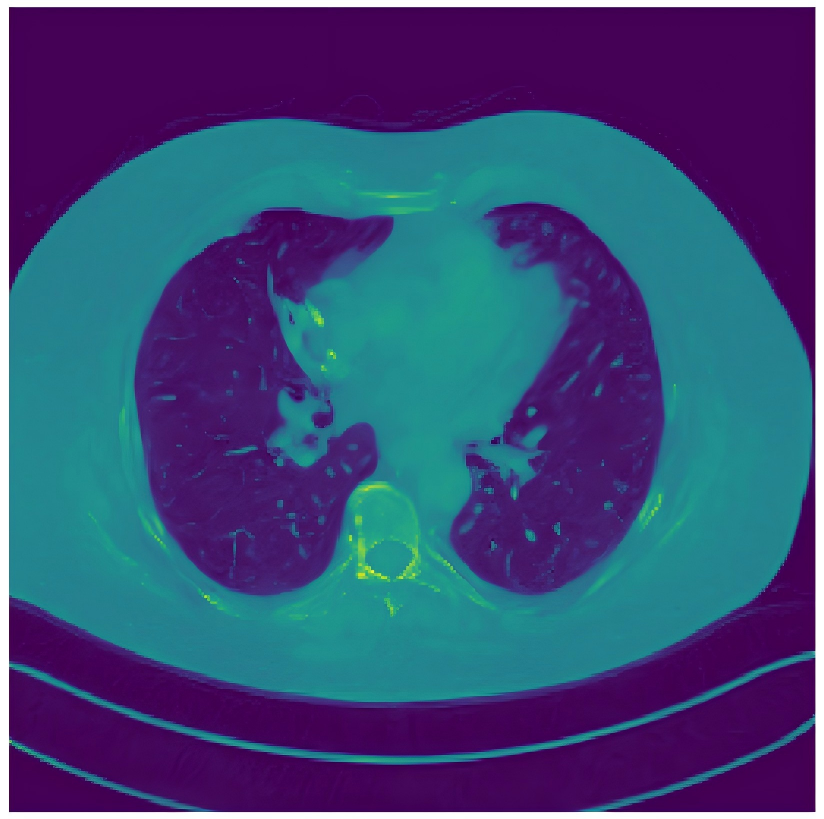}} \\
            \vspace{-0.2cm}
            \subfigure[]{\includegraphics[width=0.75\linewidth]{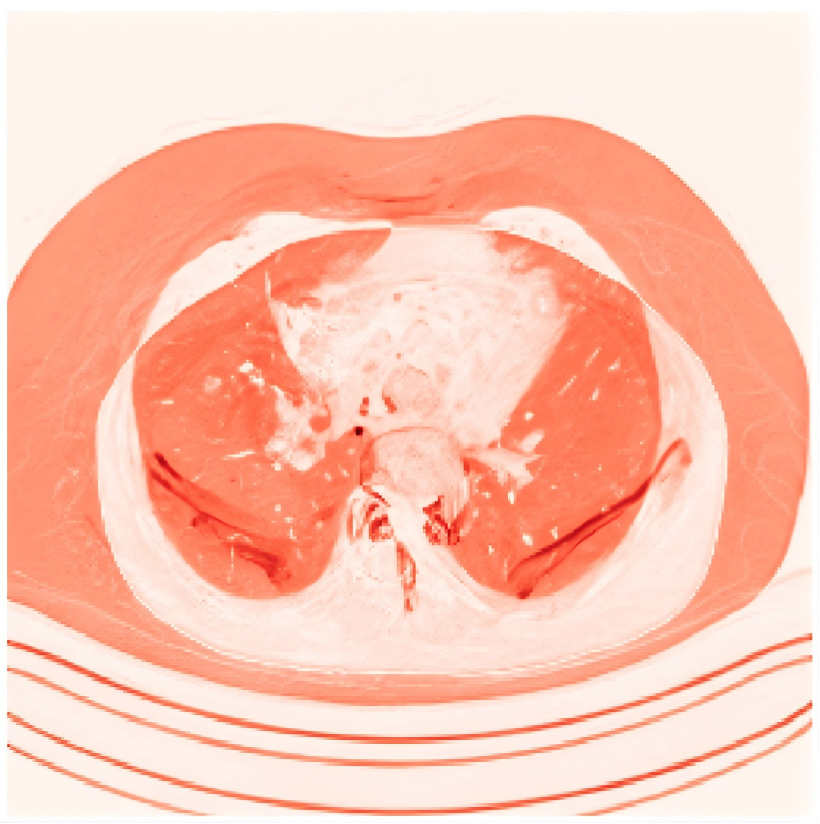}}
        \end{minipage}
        \begin{minipage}[c]{0.17\linewidth}
            \centering
            \subfigure[]{\includegraphics[width=0.75\linewidth]{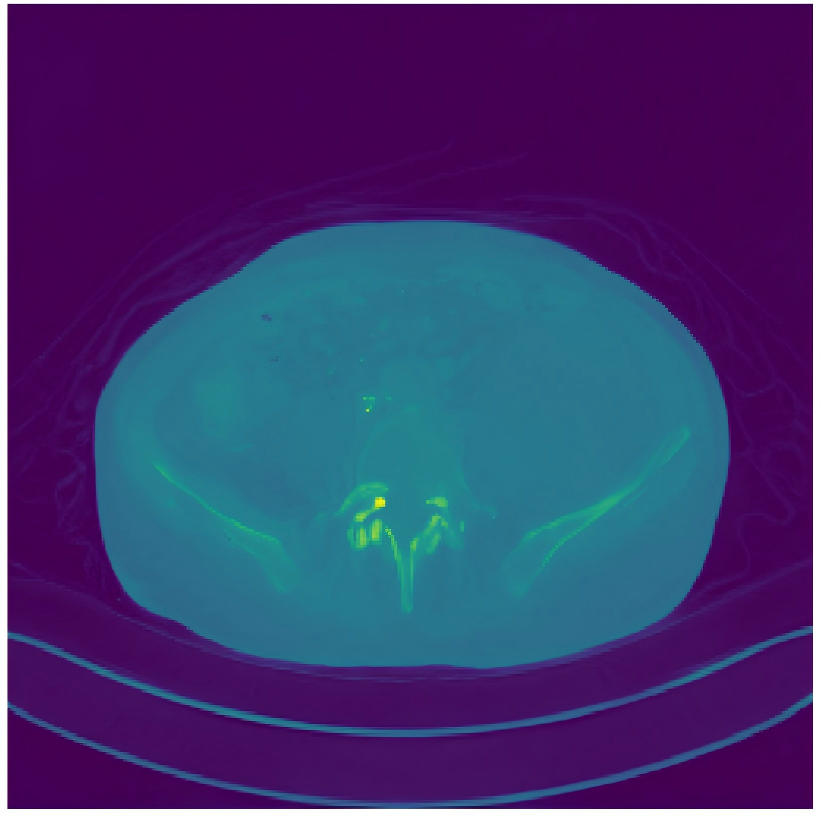}} \\
            \vspace{-0.2cm}
            \subfigure[]{\includegraphics[width=0.75\linewidth]{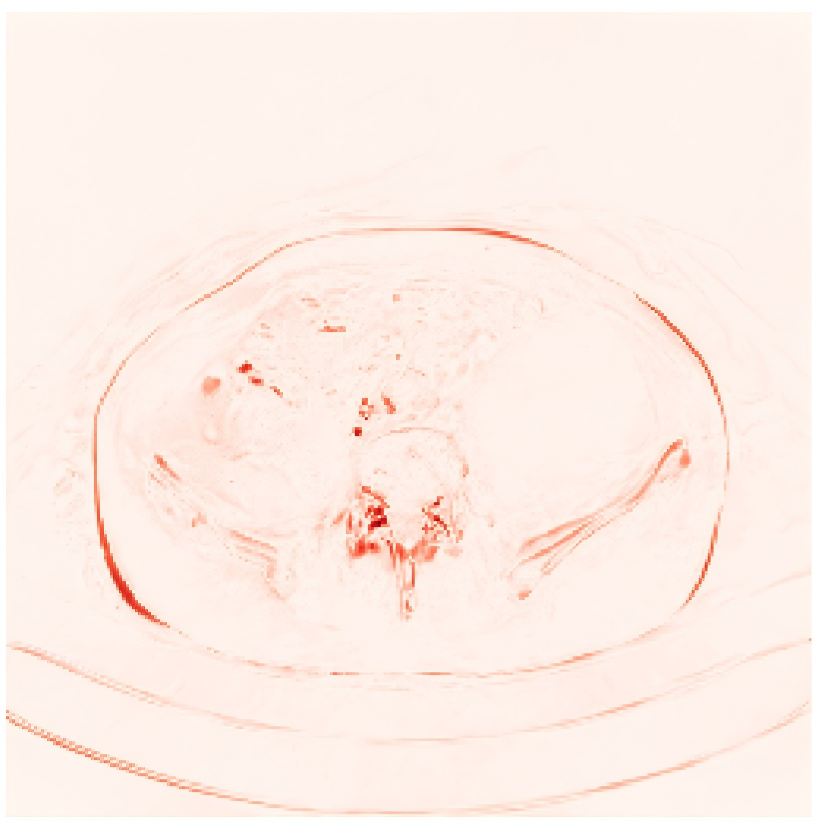}}
        \end{minipage}
        \begin{minipage}[c]{0.198\linewidth}
            \centering
            \subfigure[GT]{\includegraphics[width=0.75\linewidth]{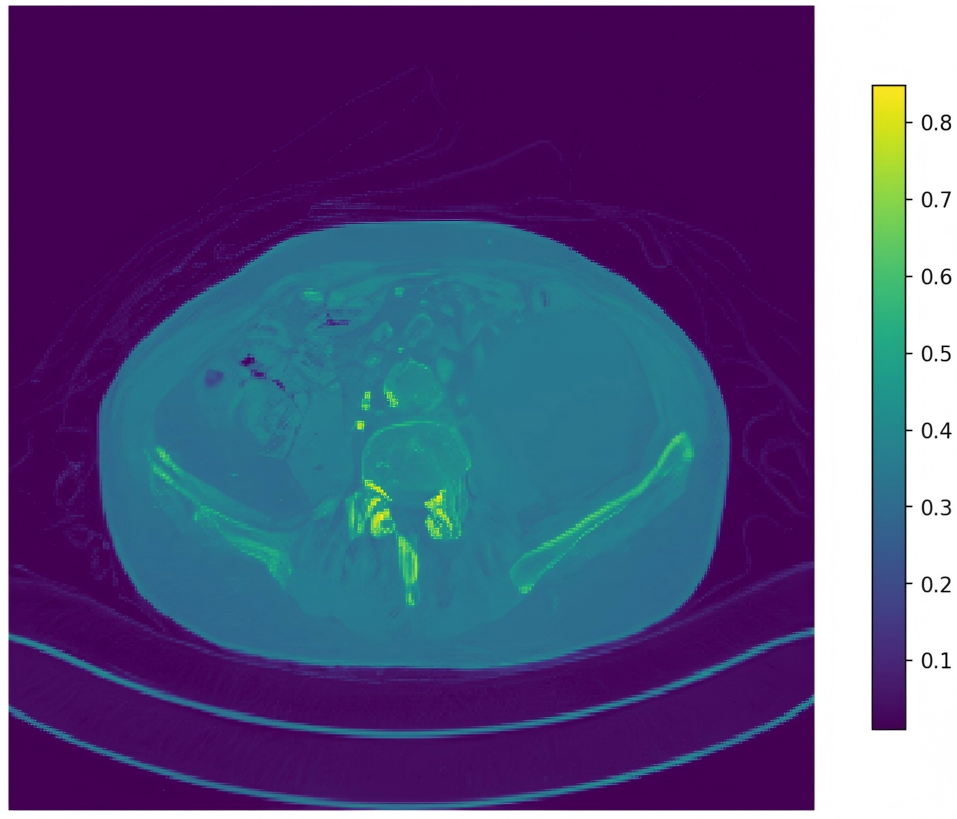}} \\
            \vspace{2.9cm}
        \end{minipage}
        \hfill
        \begin{minipage}[c]{0.38\linewidth}
            \centering
            \subfigure[]{\includegraphics[width=\linewidth]{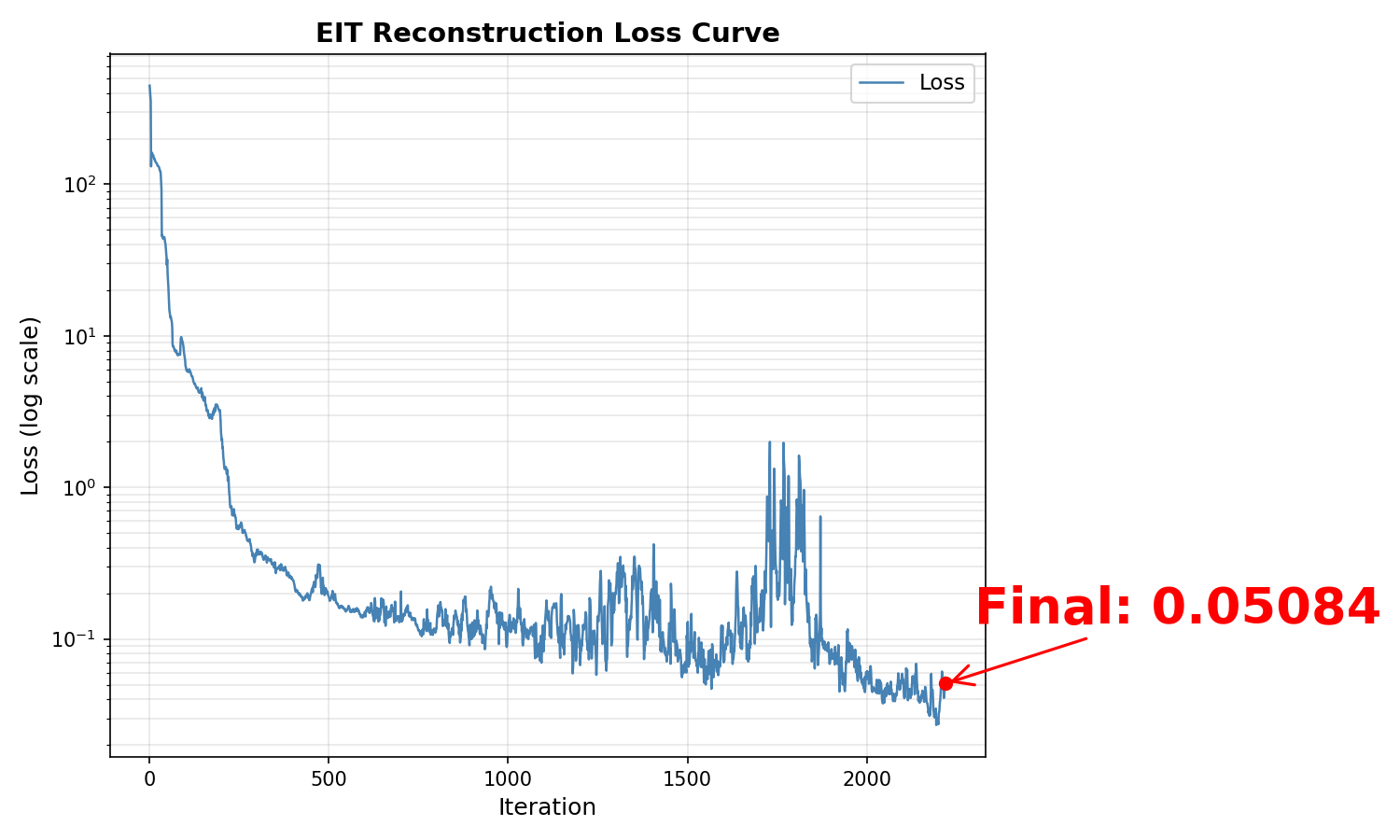}}
        \end{minipage}
    \end{minipage}
    \caption{Reconstruction analysis for a GT characterized by prominent, large-scale internal structures, demonstrating rapid boundary localization and sharp initial loss reduction.}
    \label{fig:eit2_combined}
\end{figure}

\begin{figure}[!htbp]
    \centering
    \begin{minipage}{\textwidth}
        \centering
        \begin{minipage}[c]{0.17\linewidth}
            \centering
            \subfigure[]{\includegraphics[width=0.75\linewidth]{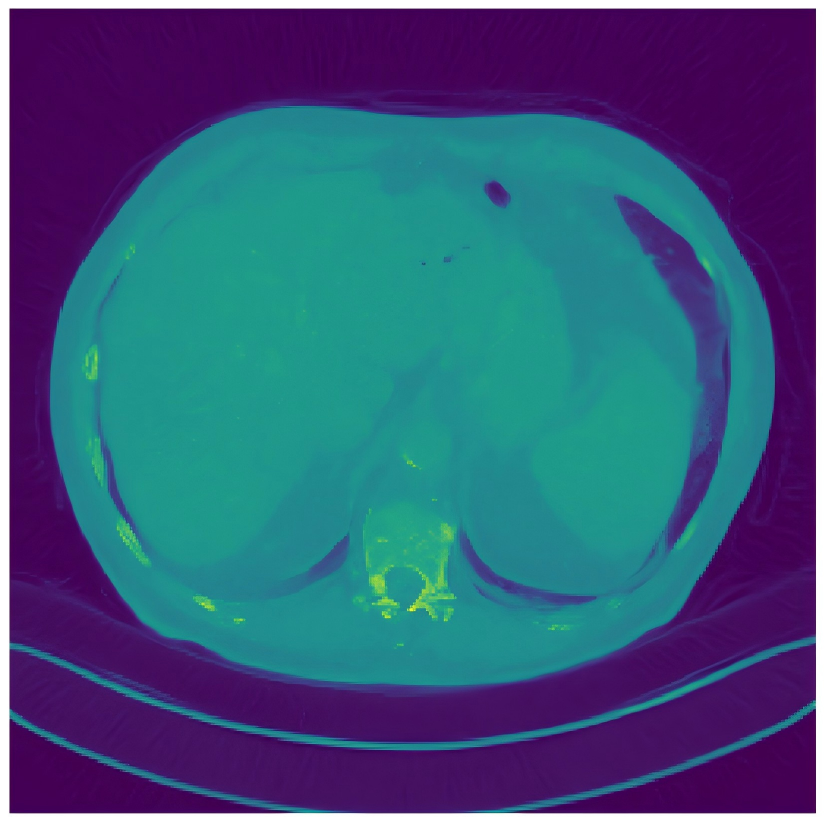}} \\
            \vspace{-0.2cm}
            \subfigure[]{\includegraphics[width=0.75\linewidth]{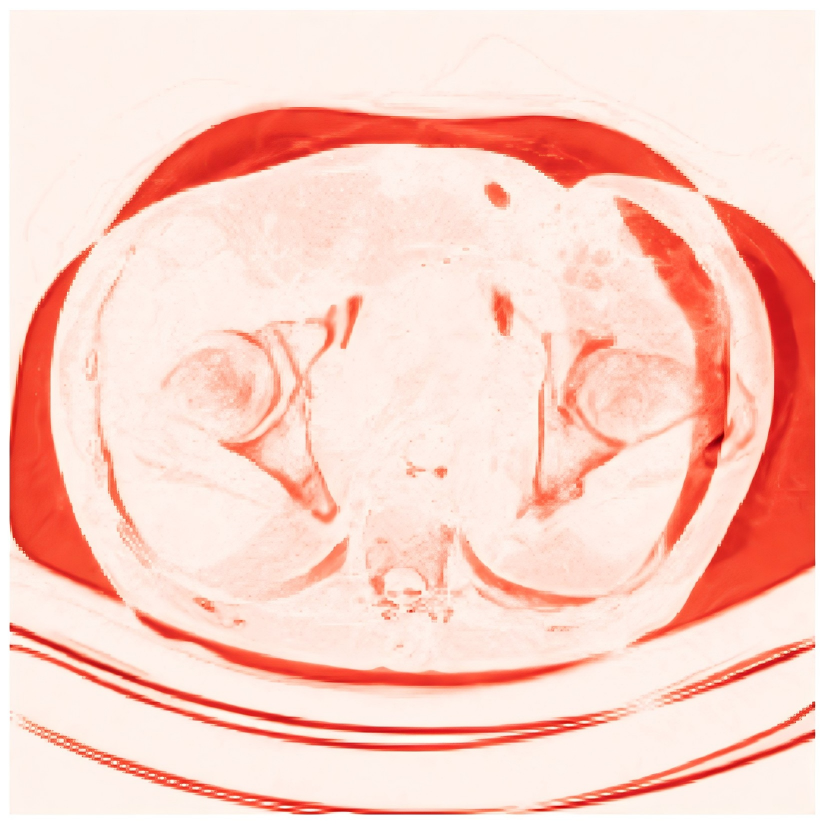}}
        \end{minipage}
        \begin{minipage}[c]{0.17\linewidth}
            \centering
            \subfigure[]{\includegraphics[width=0.75\linewidth]{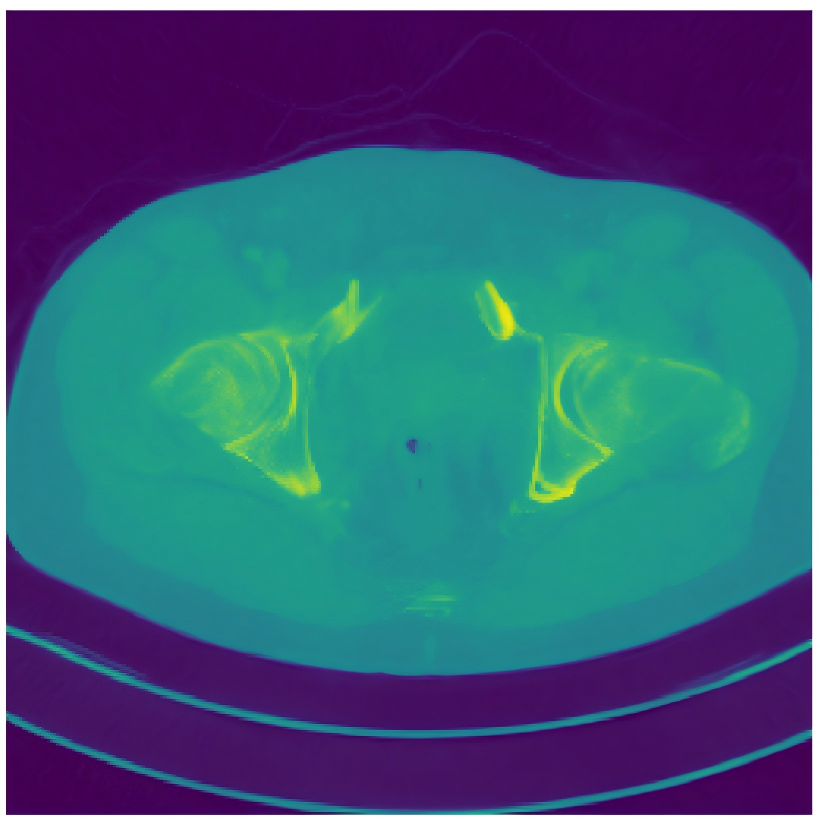}} \\
            \vspace{-0.2cm}
            \subfigure[]{\includegraphics[width=0.75\linewidth]{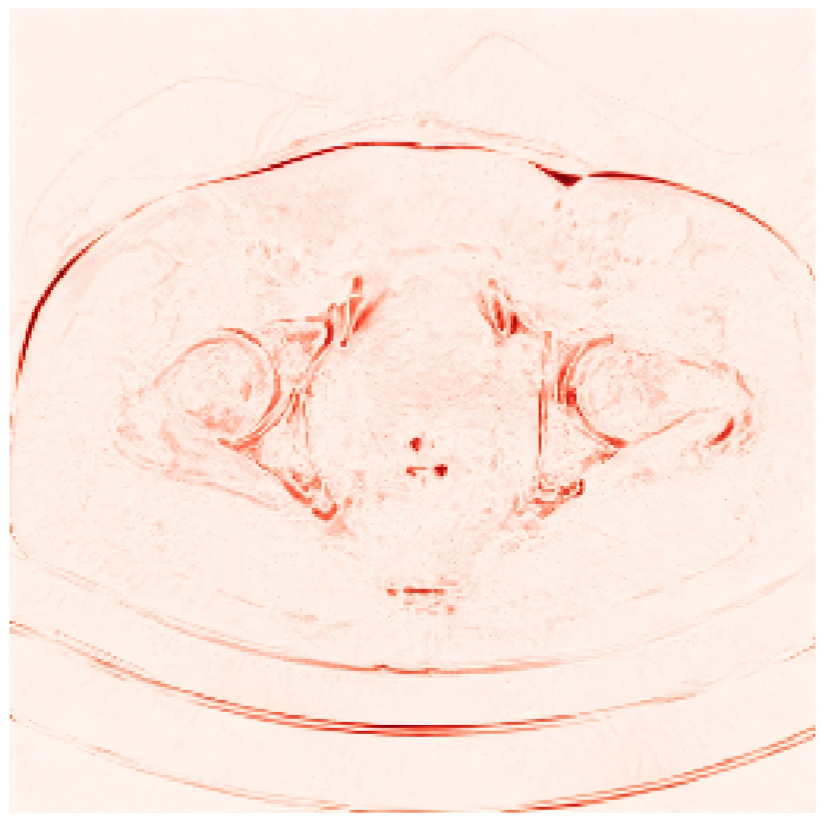}}
        \end{minipage}
        \begin{minipage}[c]{0.198\linewidth}
            \centering
            \subfigure[GT]{\includegraphics[width=0.75\linewidth]{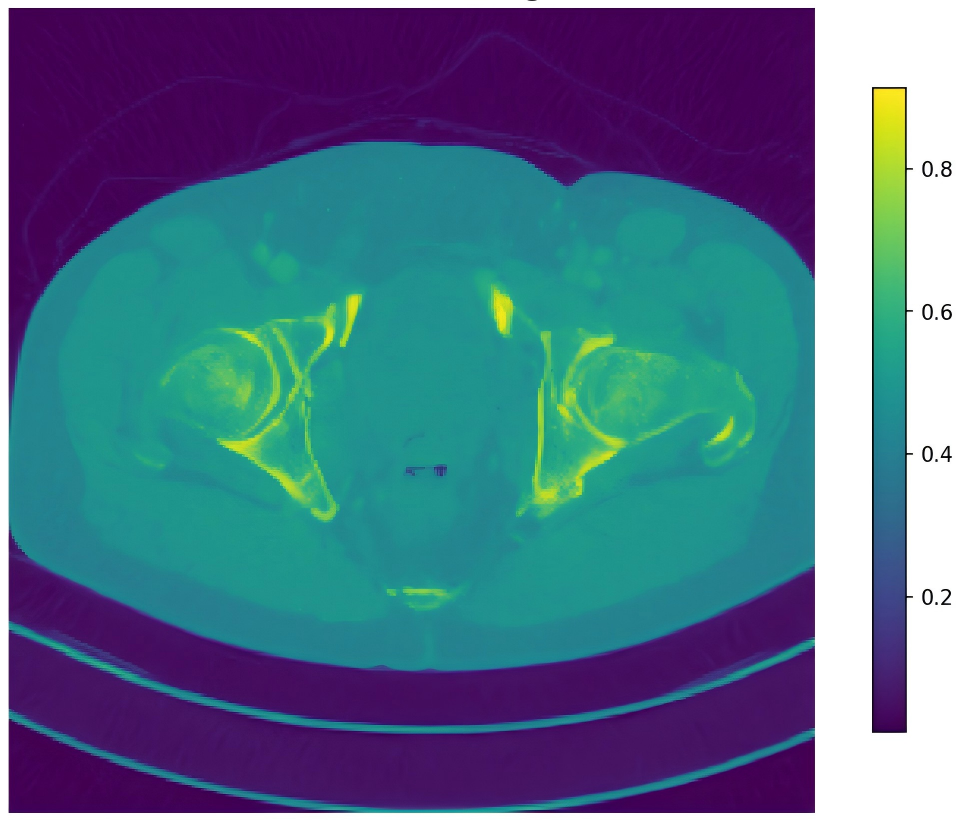}} \\
            \vspace{2.9cm}
        \end{minipage}
        \hfill
        \begin{minipage}[c]{0.38\linewidth}
            \centering
            \subfigure[]{\includegraphics[width=\linewidth]{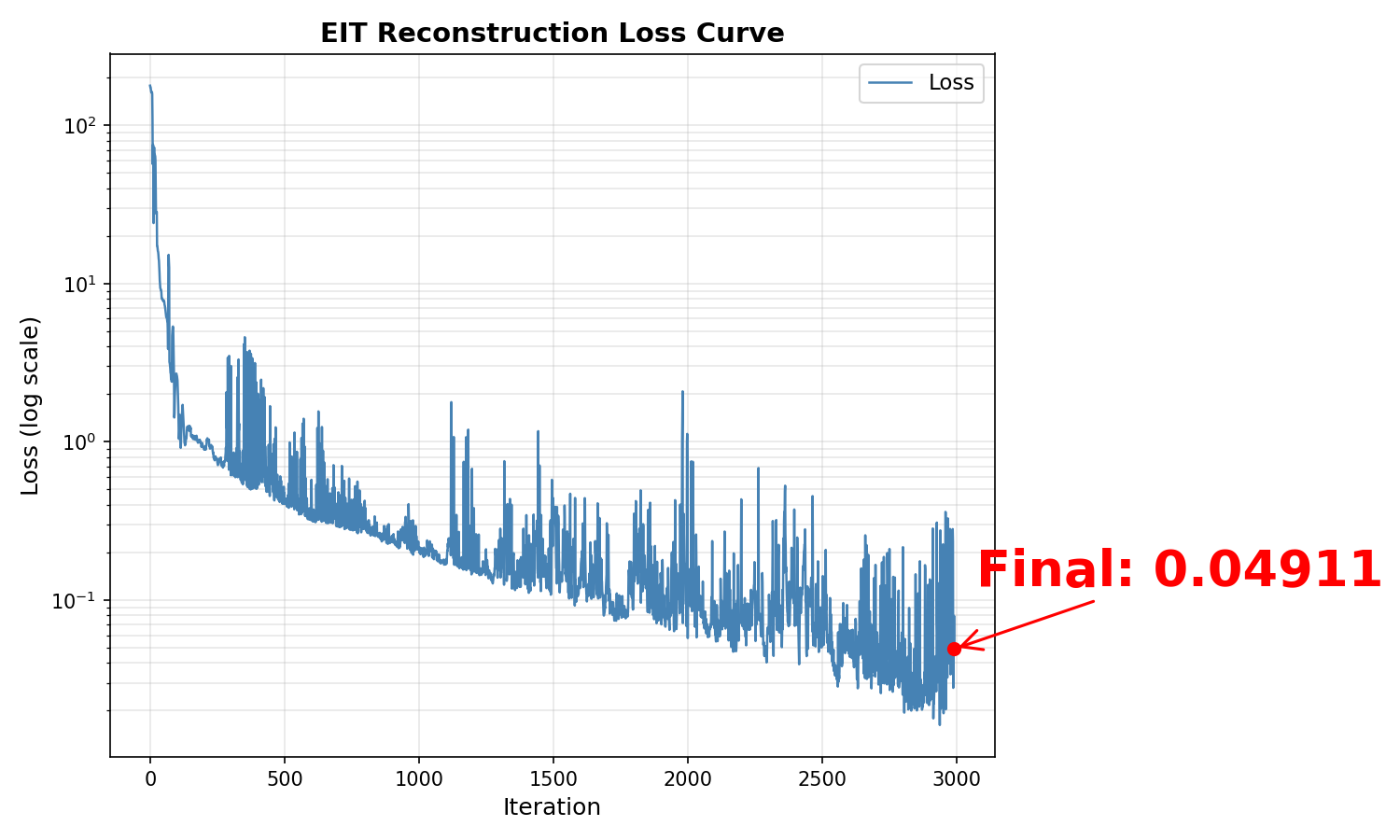}}
        \end{minipage}
    \end{minipage}
    \caption{Reconstruction analysis for a GT featuring multiple closely spaced, distinct anatomical targets, showing successful target separation without merging artifacts.}
    \label{fig:eit3_combined}
\end{figure}

\begin{figure}[!htbp]
    \centering
    \begin{minipage}{\textwidth}
        \centering
        \begin{minipage}[c]{0.17\linewidth}
            \centering
            \subfigure[]{\includegraphics[width=0.75\linewidth]{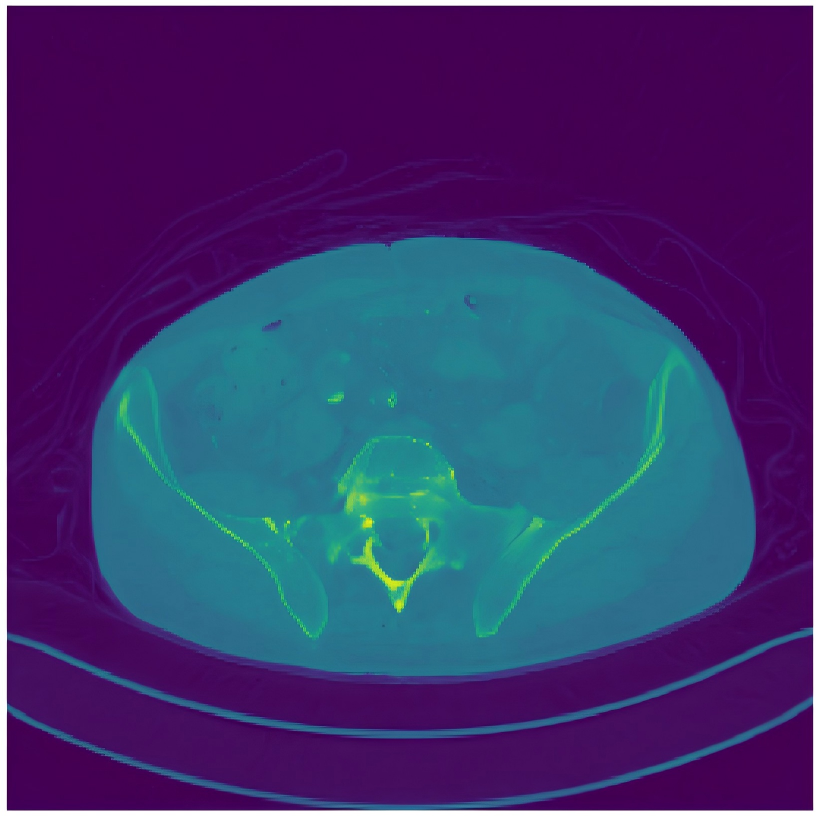}} \\
            \vspace{-0.2cm}
            \subfigure[]{\includegraphics[width=0.75\linewidth]{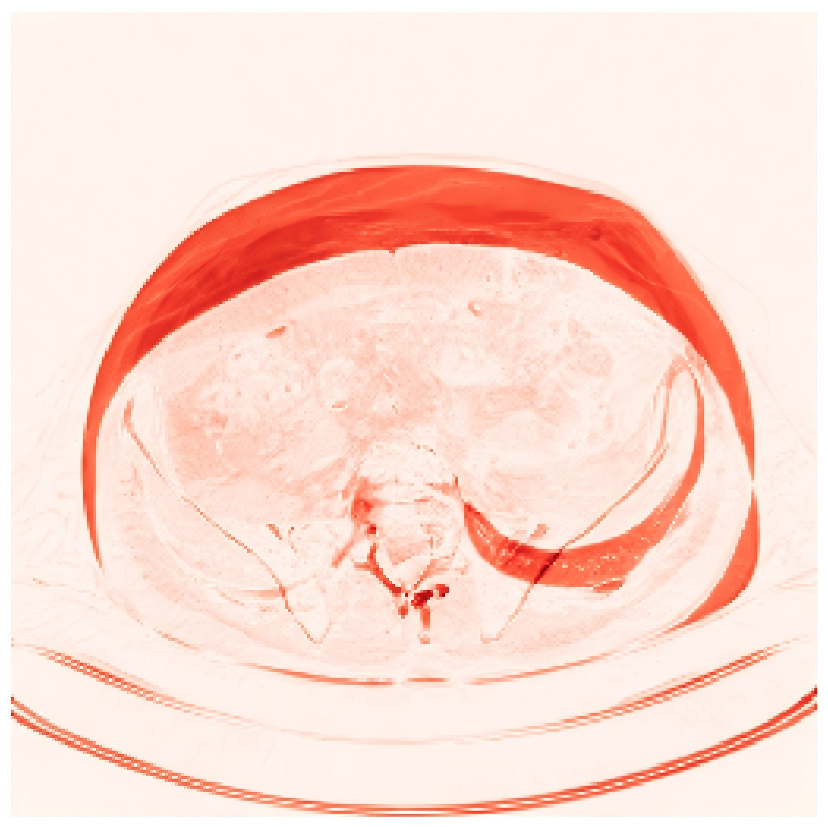}}
        \end{minipage}
        \begin{minipage}[c]{0.17\linewidth}
            \centering
            \subfigure[]{\includegraphics[width=0.75\linewidth]{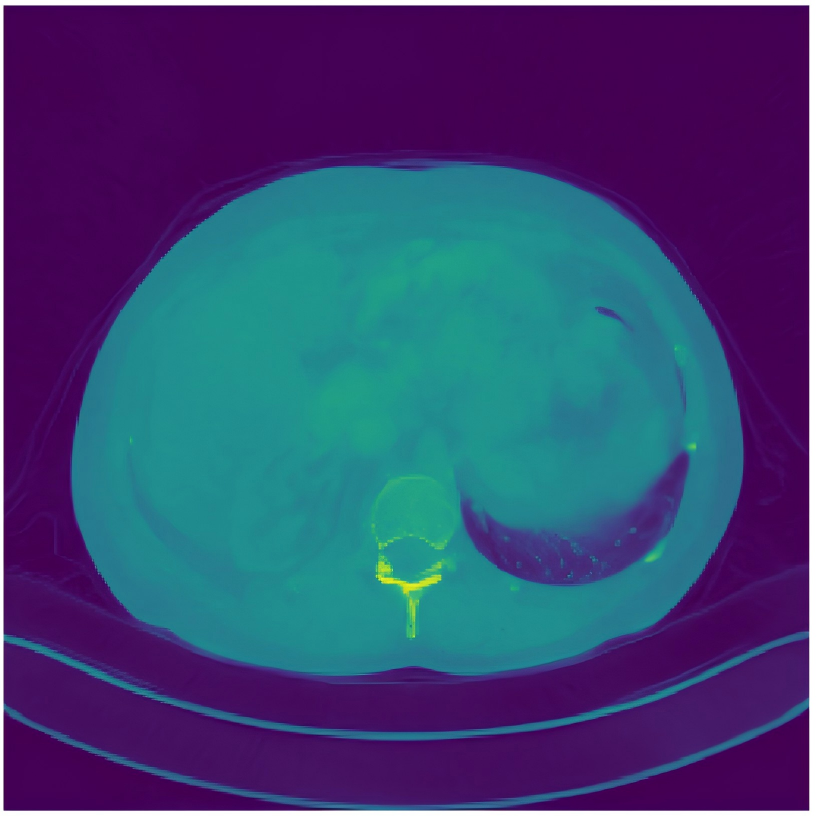}} \\
            \vspace{-0.2cm}
            \subfigure[]{\includegraphics[width=0.75\linewidth]{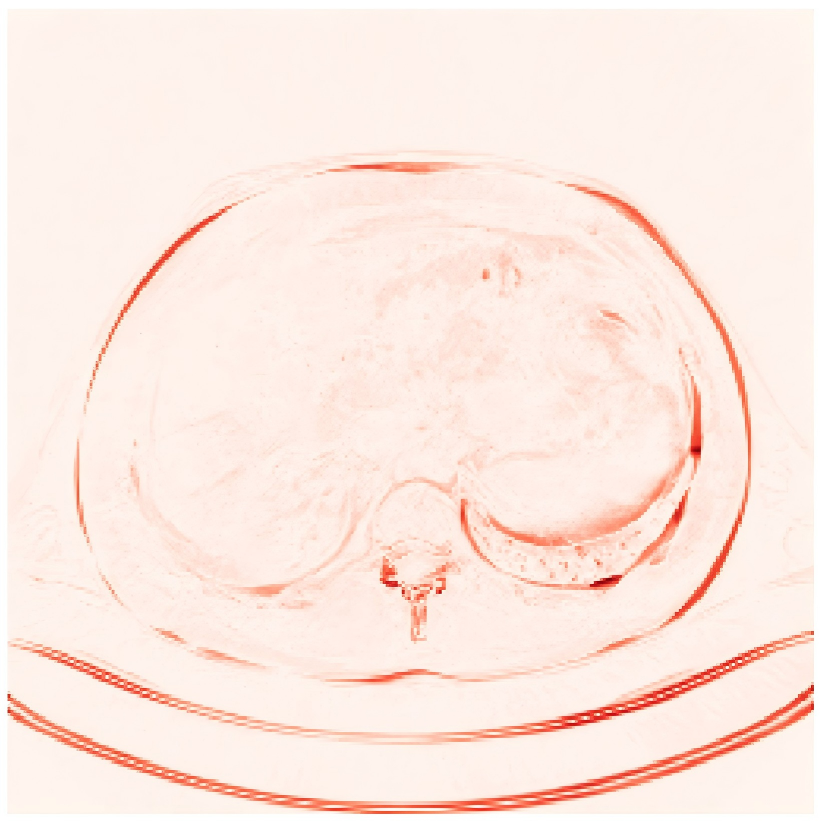}}
        \end{minipage}
        \begin{minipage}[c]{0.17\linewidth}
            \centering
            \subfigure[GT]{\includegraphics[width=0.75\linewidth]{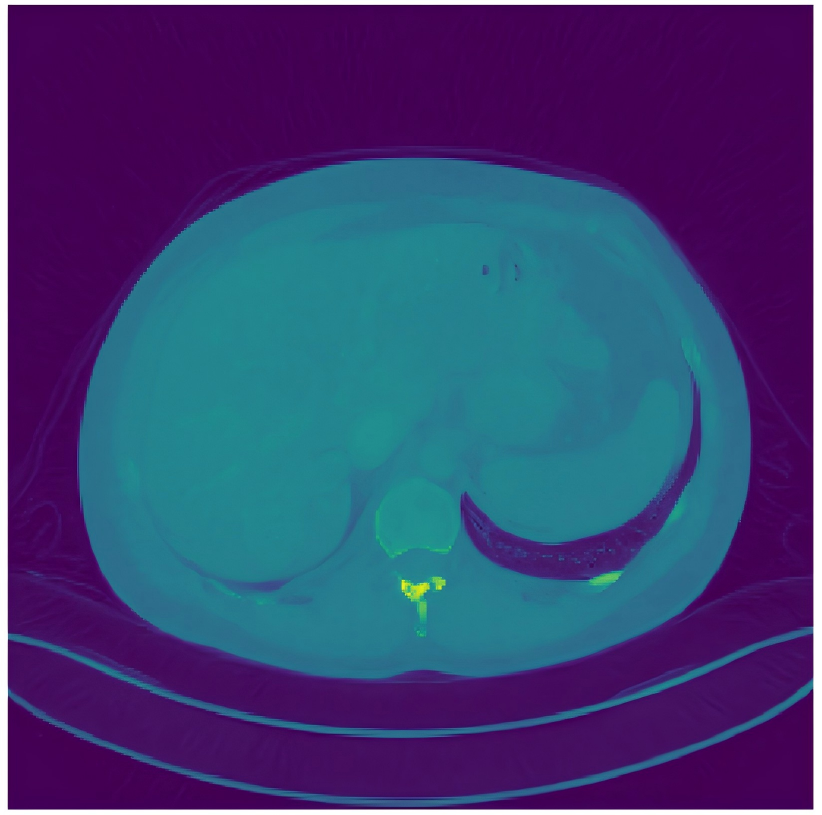}} \\
            \vspace{2.9cm}
        \end{minipage}
        \hfill
        \begin{minipage}[c]{0.38\linewidth}
            \centering
            \subfigure[]{\includegraphics[width=\linewidth]{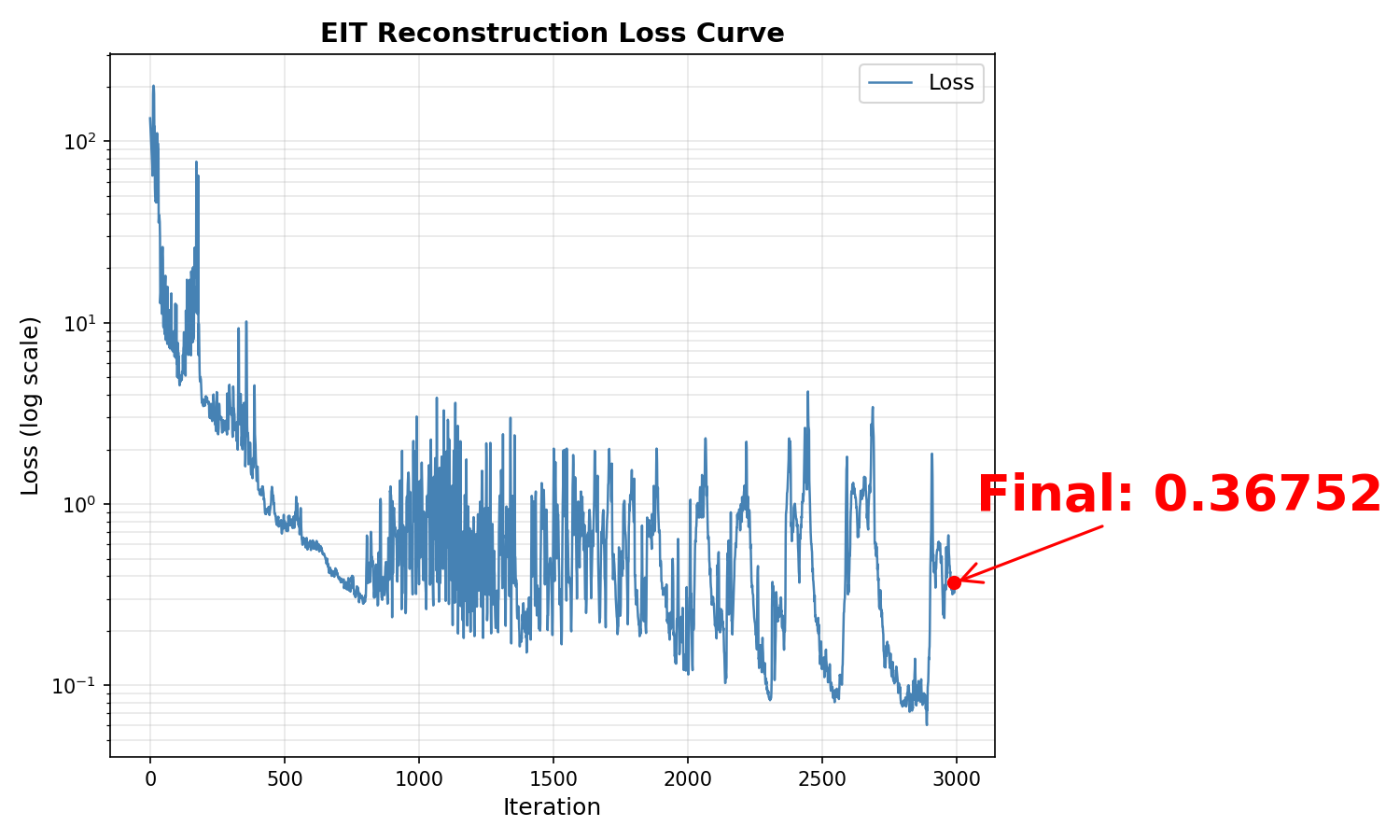}}
        \end{minipage}
    \end{minipage}
    \caption{Reconstruction analysis for an abdominal GT containing small, localized high-contrast targets within a large uniform background, highlighting the structural regularization provided by the diffusion prior.}
    \label{fig:eit4_combined}
\end{figure}

\begin{figure}[!htbp]
    \centering
    \begin{minipage}{\textwidth}
        \centering
        \begin{minipage}[c]{0.17\linewidth}
            \centering
            \subfigure[]{\includegraphics[width=0.75\linewidth]{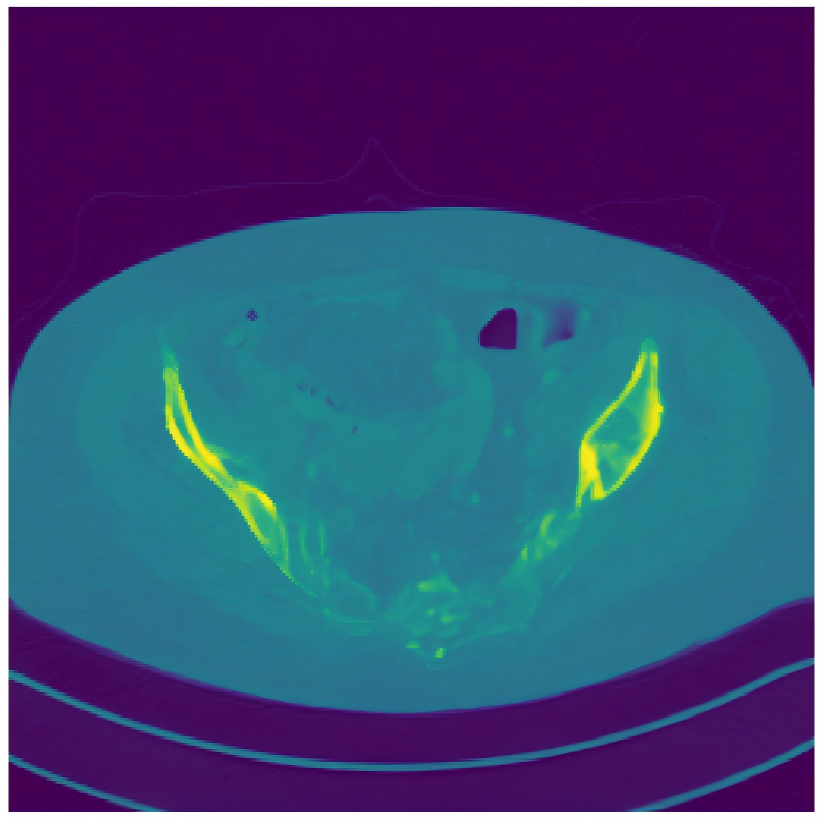}} \\
            \vspace{-0.2cm}
            \subfigure[]{\includegraphics[width=0.75\linewidth]{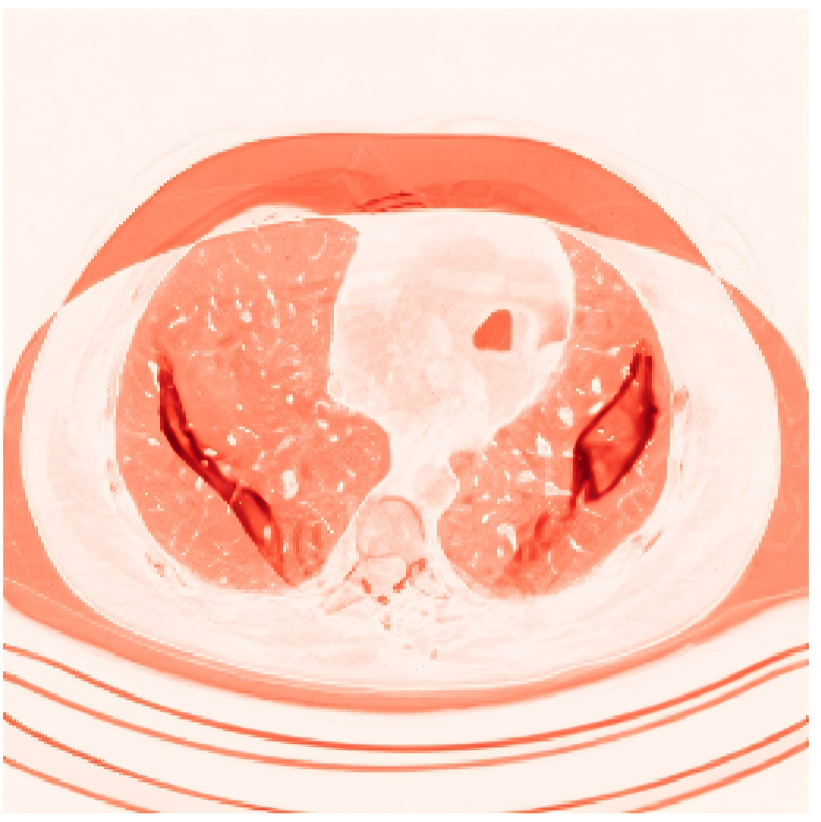}}
        \end{minipage}
        \begin{minipage}[c]{0.17\linewidth}
            \centering
            \subfigure[]{\includegraphics[width=0.75\linewidth]{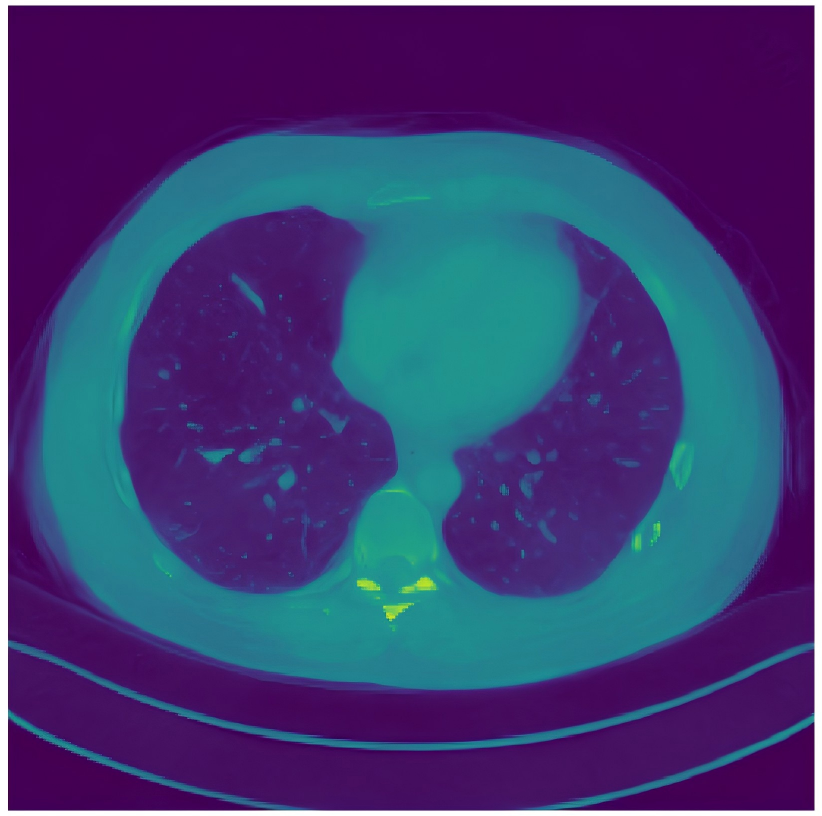}} \\
            \vspace{-0.2cm}
            \subfigure[]{\includegraphics[width=0.75\linewidth]{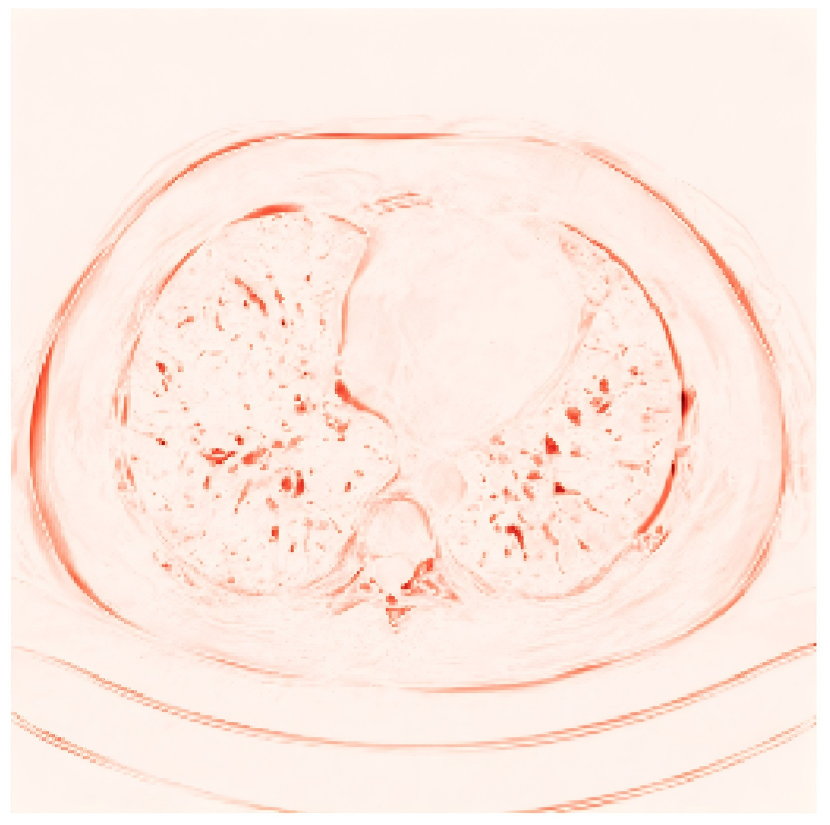}}
        \end{minipage}
        \begin{minipage}[c]{0.198\linewidth}
            \centering
            \subfigure[GT]{\includegraphics[width=0.75\linewidth]{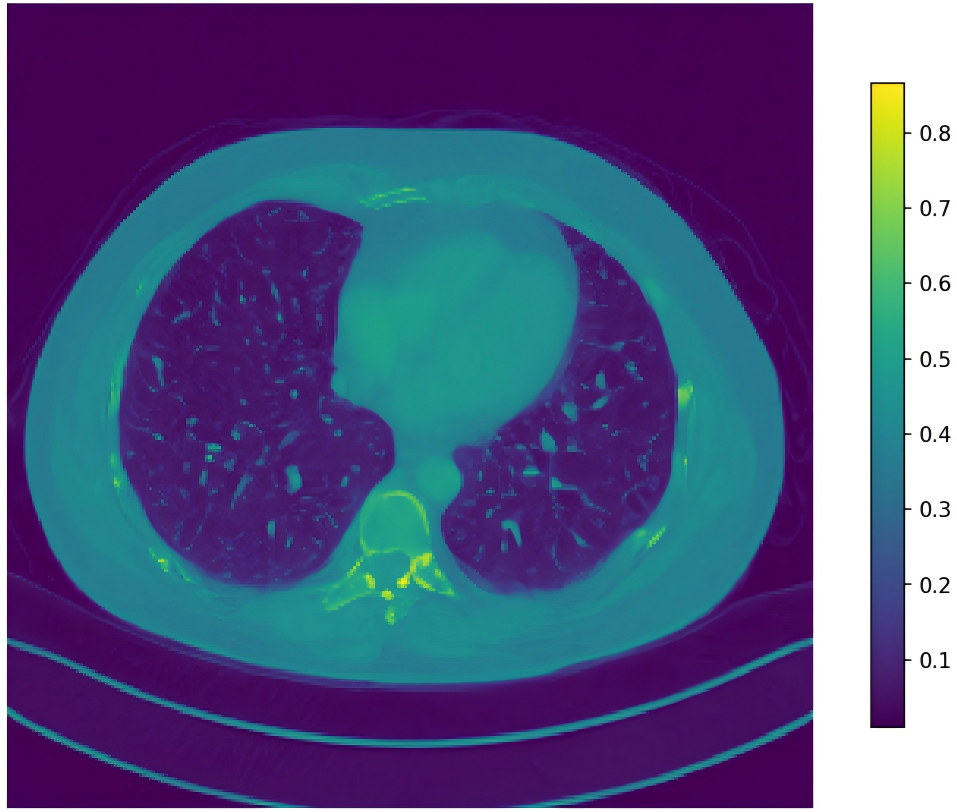}} \\
            \vspace{2.9cm}
        \end{minipage}
        \hfill
        \begin{minipage}[c]{0.38\linewidth}
            \centering
            \subfigure[]{\includegraphics[width=\linewidth]{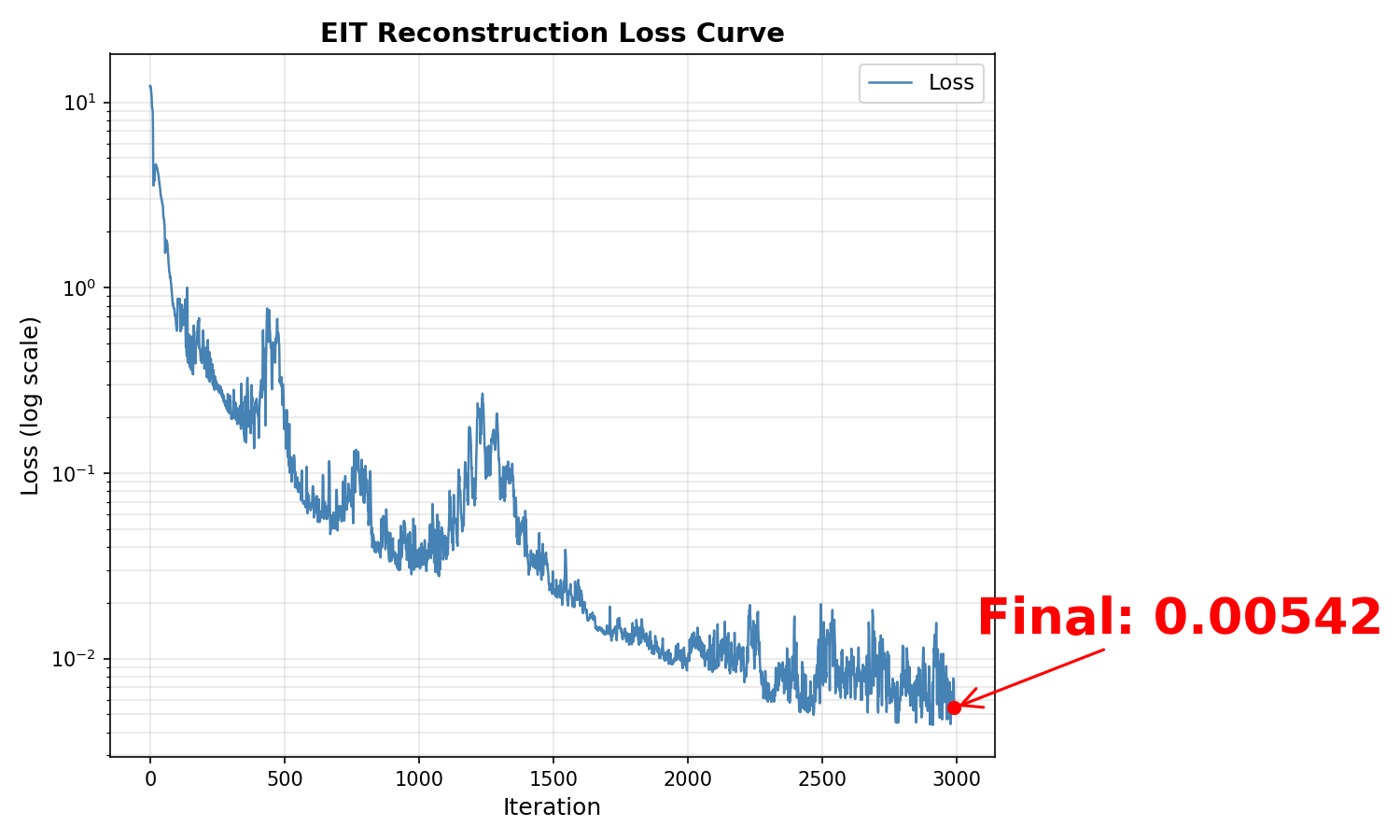}}
        \end{minipage}
    \end{minipage}
    \caption{Reconstruction analysis for a chest GT exhibiting large, low-conductivity lung fields contrasted with very small, sharp anomalies, demonstrating the algorithm's robustness to scale variations and its ability to preserve sharp boundaries.}
    \label{fig:eit6_combined}
\end{figure}

\paragraph{Inverse Scattering}
To systematically evaluate the efficacy of DiLO in inverse scattering, Figures \ref{fig:is1_combined}--\ref{fig:is2_combined} adopt a unified layout: the left panel presents a $3 \times 4$ matrix displaying the predicted parameter $\eta$ alongside the real part, imaginary part, and magnitude of the scattered field $u$ at the initial state (top row), the final iteration (middle row), and the Ground Truth (bottom row). The right panel shows the corresponding real-time loss trajectory.

\begin{figure}[!htbp]
    \centering
    \begin{minipage}[c]{0.48\textwidth}
        \centering
        \subfigure{\includegraphics[width=\textwidth]{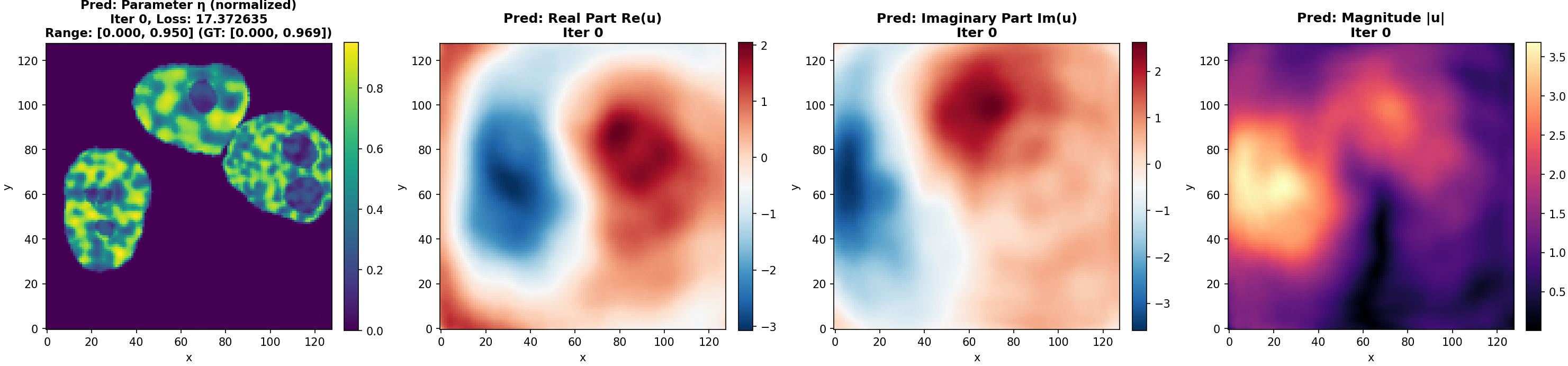}}
        \\[1ex]
        \subfigure{\includegraphics[width=\textwidth]{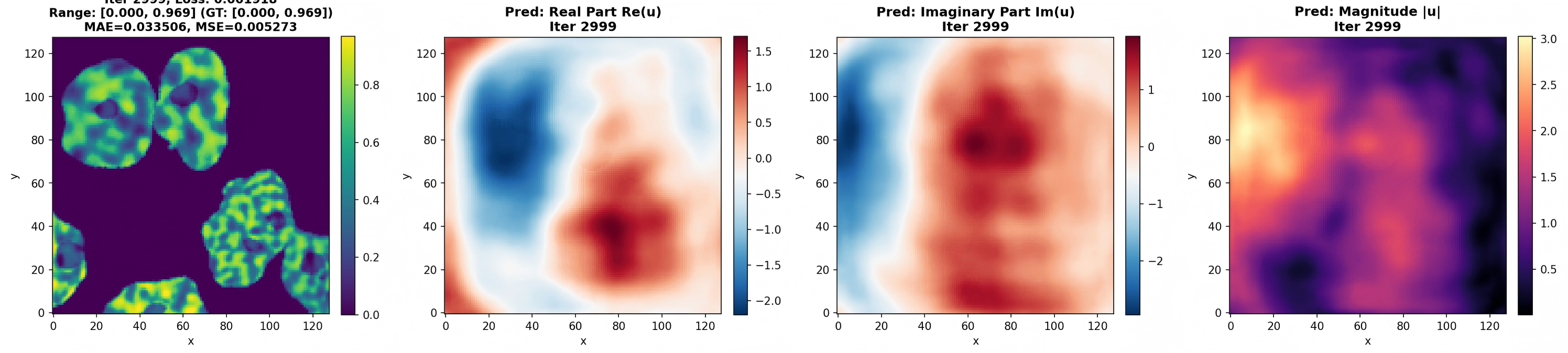}}
        \\[1ex]
        \subfigure{\includegraphics[width=\textwidth]{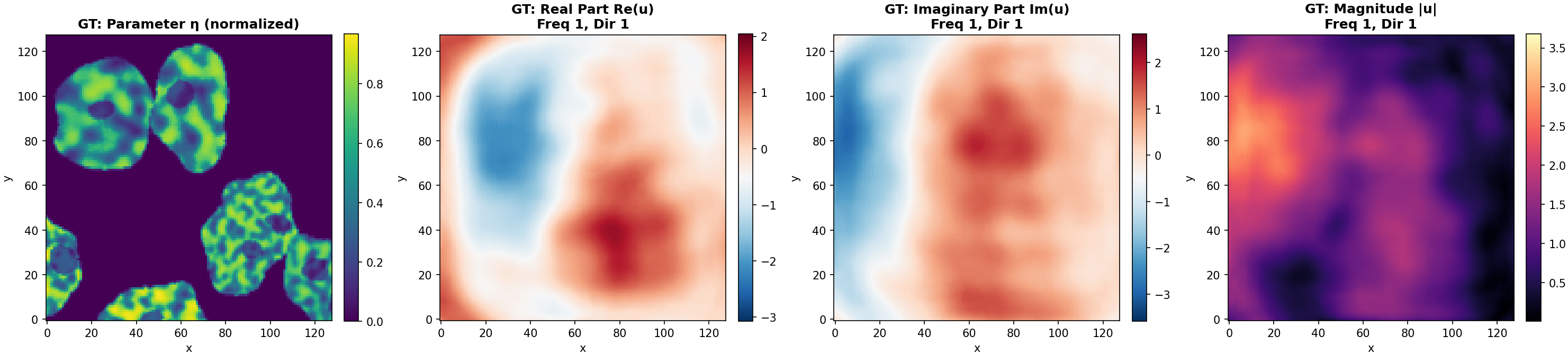}}
    \end{minipage}
    \hfill
    \begin{minipage}[c]{0.50\textwidth}
        \centering
        \subfigure[Real-time Loss Curve]{\includegraphics[width=\textwidth]{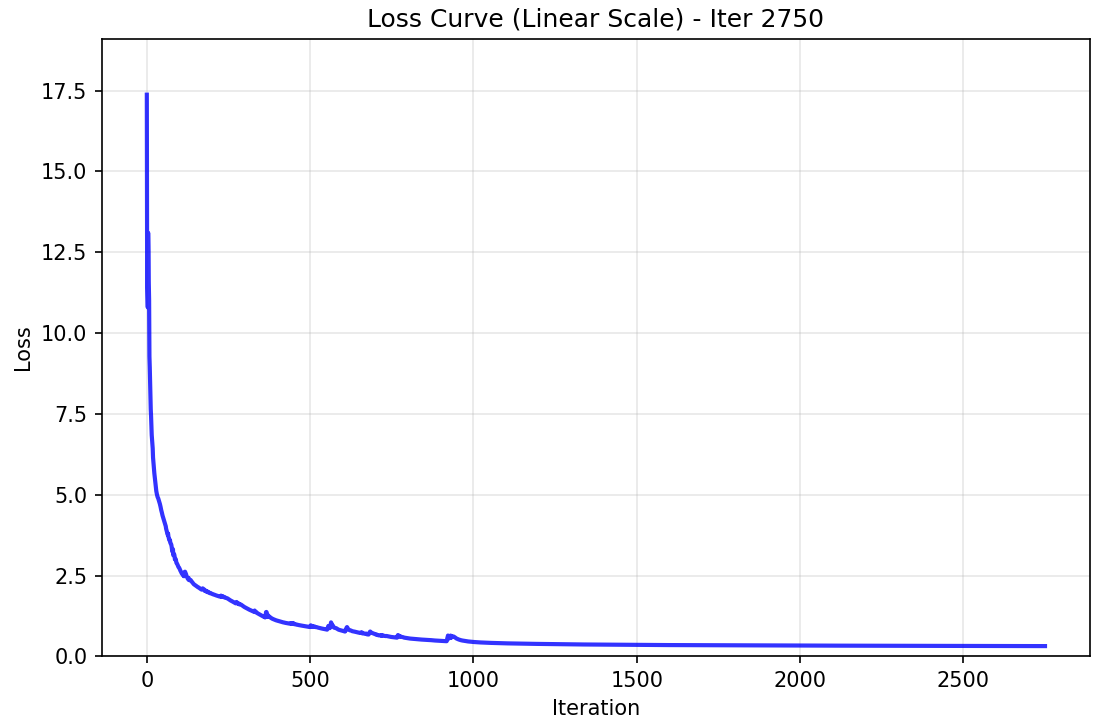}}
    \end{minipage}
    \caption{Reconstruction analysis for sparse cell structures with heterogeneous internal contrasts, demonstrating rapid localization of internal details and stable convergence.}
    \label{fig:is1_combined}
\end{figure}

\begin{figure}[!htbp]
    \centering
    \begin{minipage}[c]{0.48\textwidth}
        \centering
        \subfigure{\includegraphics[width=\textwidth]{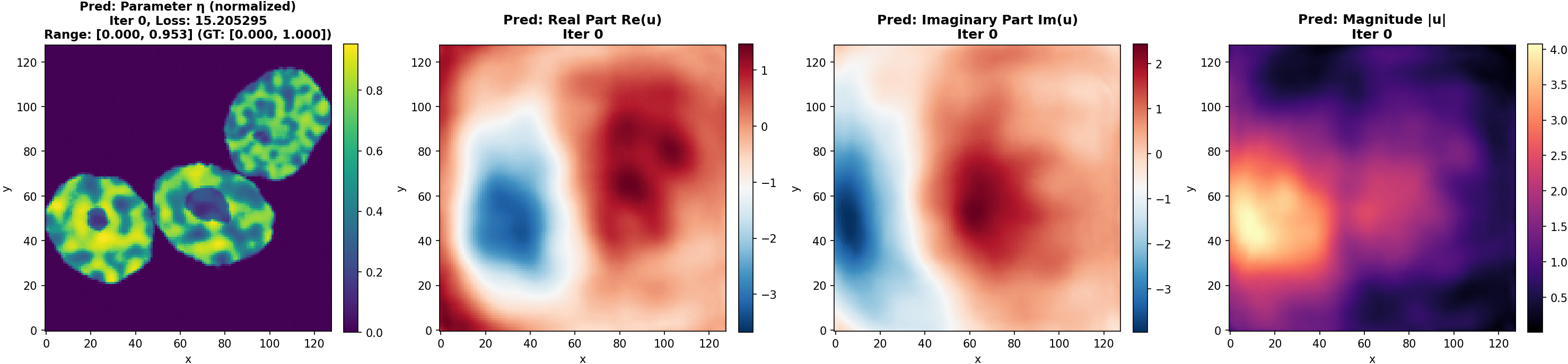}}
        \\[1ex]
        \subfigure{\includegraphics[width=\textwidth]{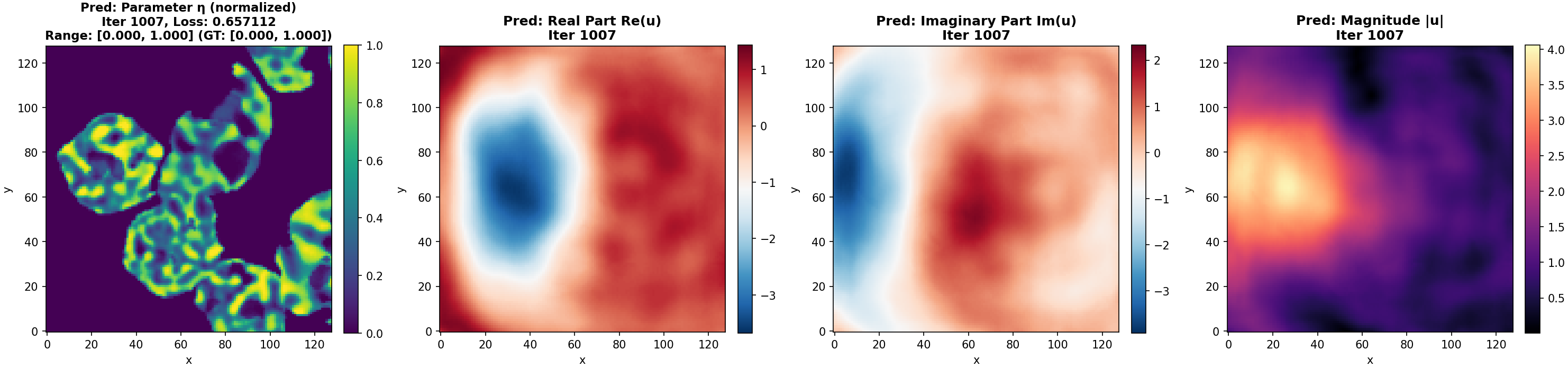}}
        \\[1ex]
        \subfigure{\includegraphics[width=\textwidth]{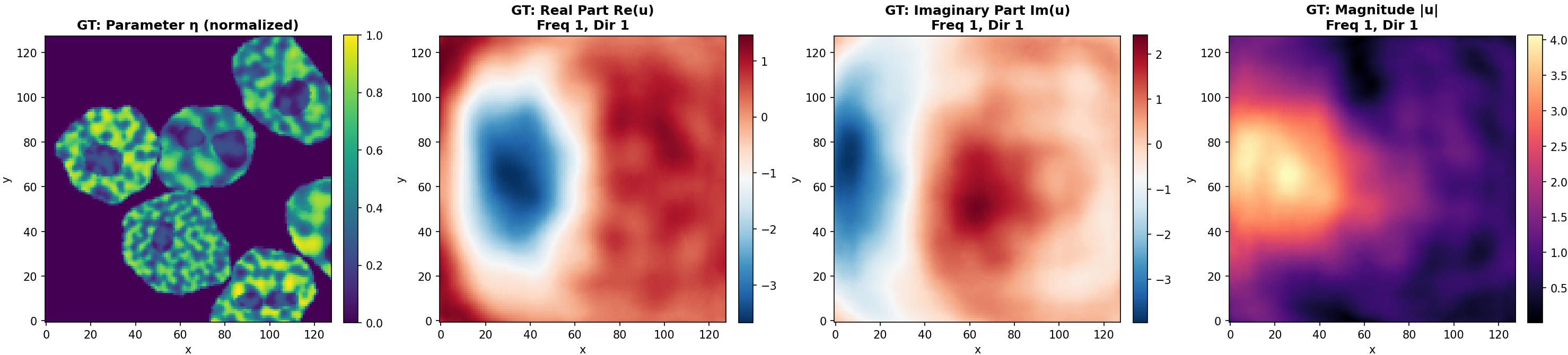}}
    \end{minipage}
    \hfill
    \begin{minipage}[c]{0.50\textwidth}
        \centering
        \subfigure[Real-time Loss Curve]{\includegraphics[width=\textwidth]{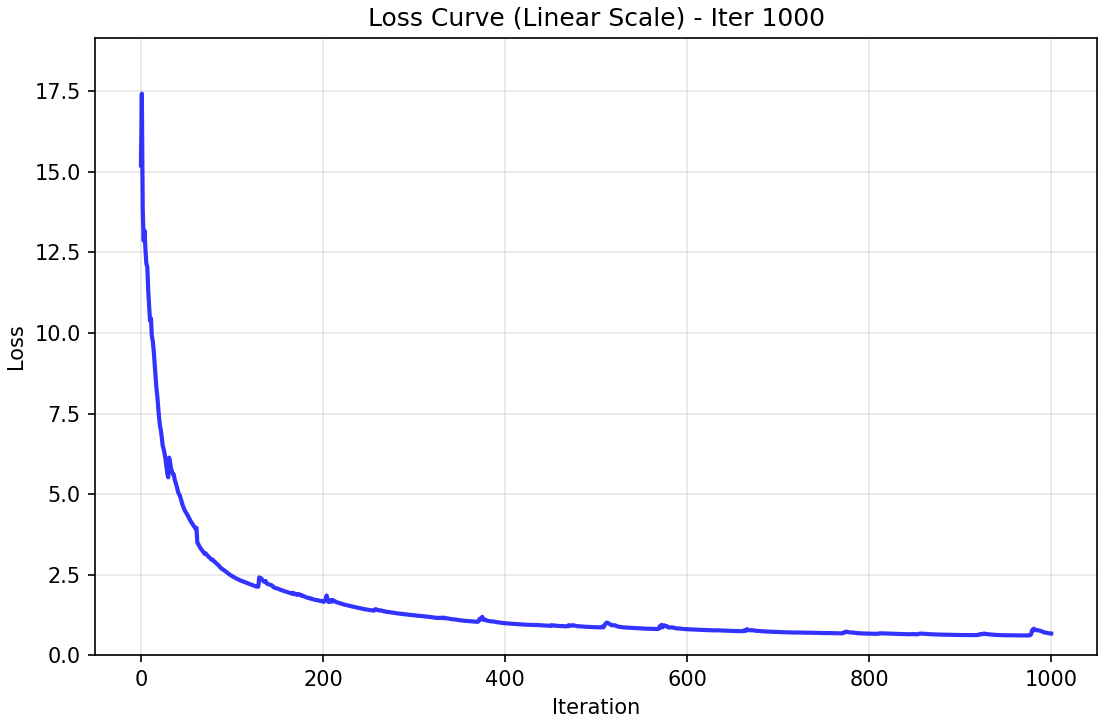}}
    \end{minipage}
    \caption{Reconstruction analysis for a tightly packed cluster of biological-like cells, confirming accurate recovery of fine details despite complex multiple scattering and inter-cellular interactions.}
    \label{fig:is2_combined}
\end{figure}

\paragraph{Inverse Navier-Stokes}
The framework's capability in recovering complex fluid dynamics is validated in Figures~\ref{fig:ns1_combined}--\ref{fig:ns5_noise}. To systematically present the reconstruction performance, each figure shares a consistent layout: panel (a) displays the initial state (Iteration 0) comprising the ground truth, the predicted field, and the absolute difference; panel (b) shows the corresponding finalized reconstruction state; and panel (c) plots the log-scale loss curve.

\begin{figure}[htbp]
    \centering
    \subfigure[ns1: Iteration 0000]{\includegraphics[width=0.32\textwidth]{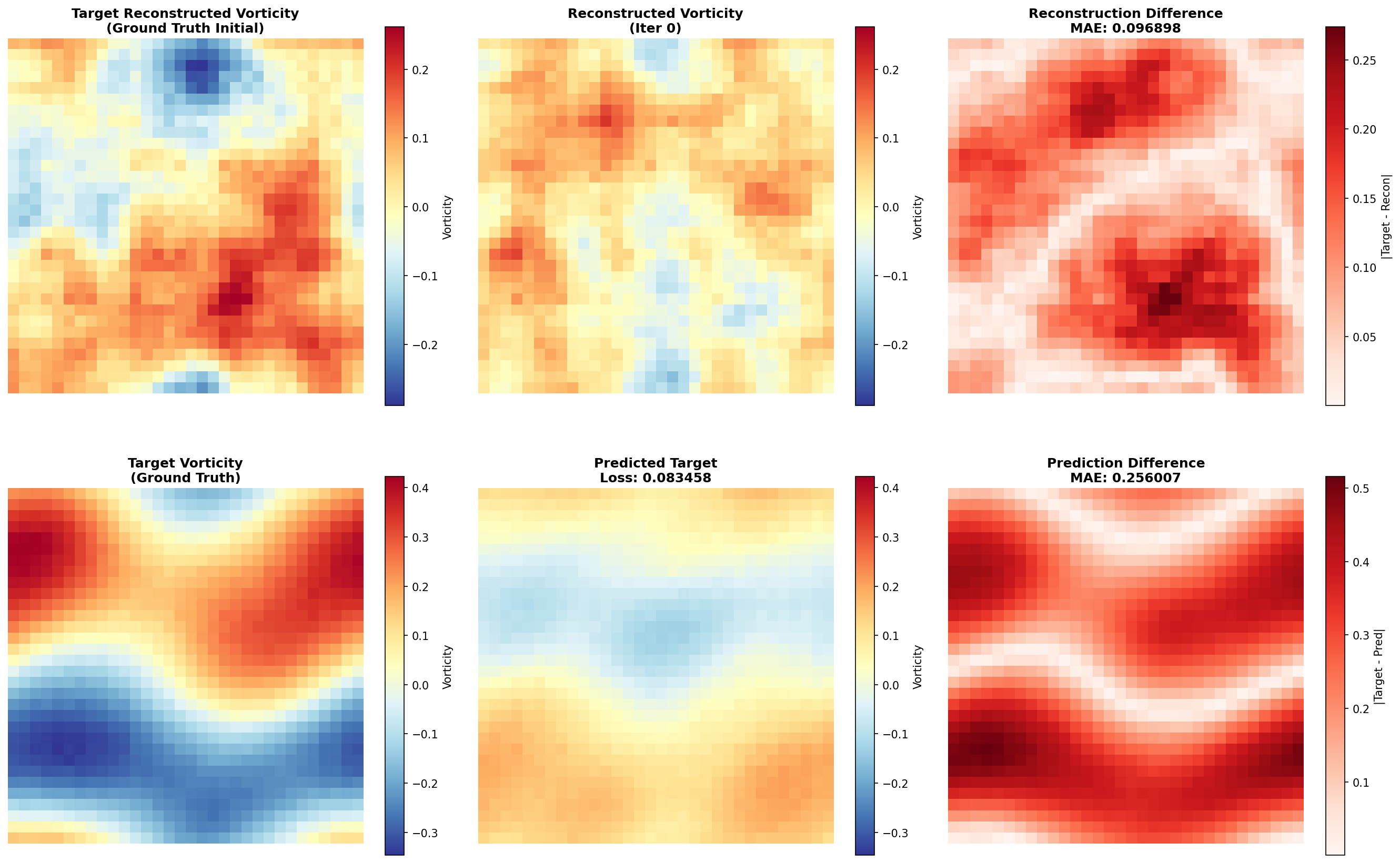}}
    \subfigure[ns1: Final (0670)]{\includegraphics[width=0.32\textwidth]{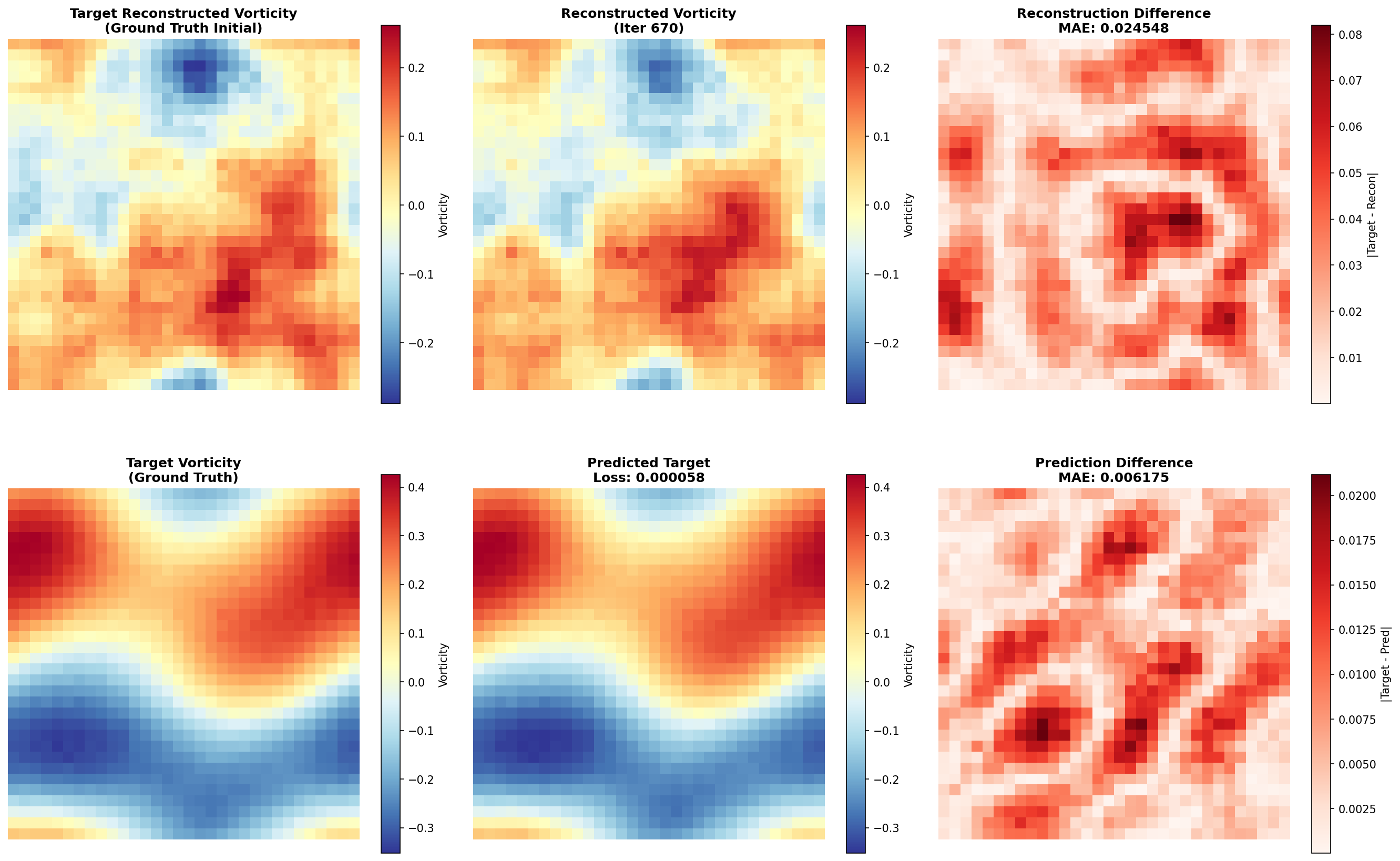}}
    \subfigure[ns1: Loss Curve]{\includegraphics[width=0.32\textwidth]{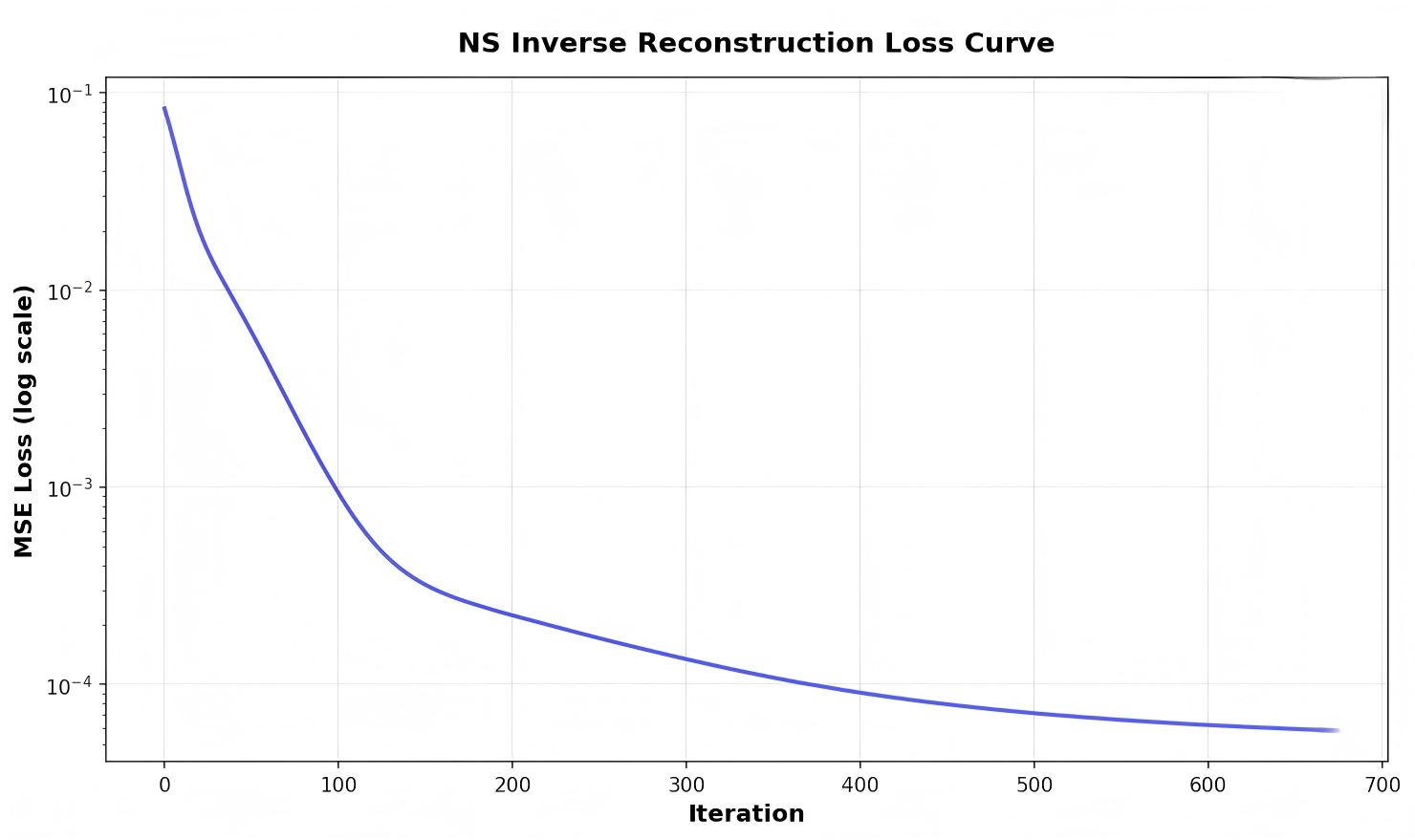}}
    \caption{Reconstruction analysis for an initial vorticity field generated via a Gaussian random field (GRF), featuring an irregular, dipole-like vortex pair. The results demonstrate the rapid, high-fidelity recovery of these fragmented fluid structures.}
    \label{fig:ns1_combined}
\end{figure}

\begin{figure}[htbp]
    \centering
    \subfigure[ns2: Iteration 0000]{\includegraphics[width=0.32\textwidth]{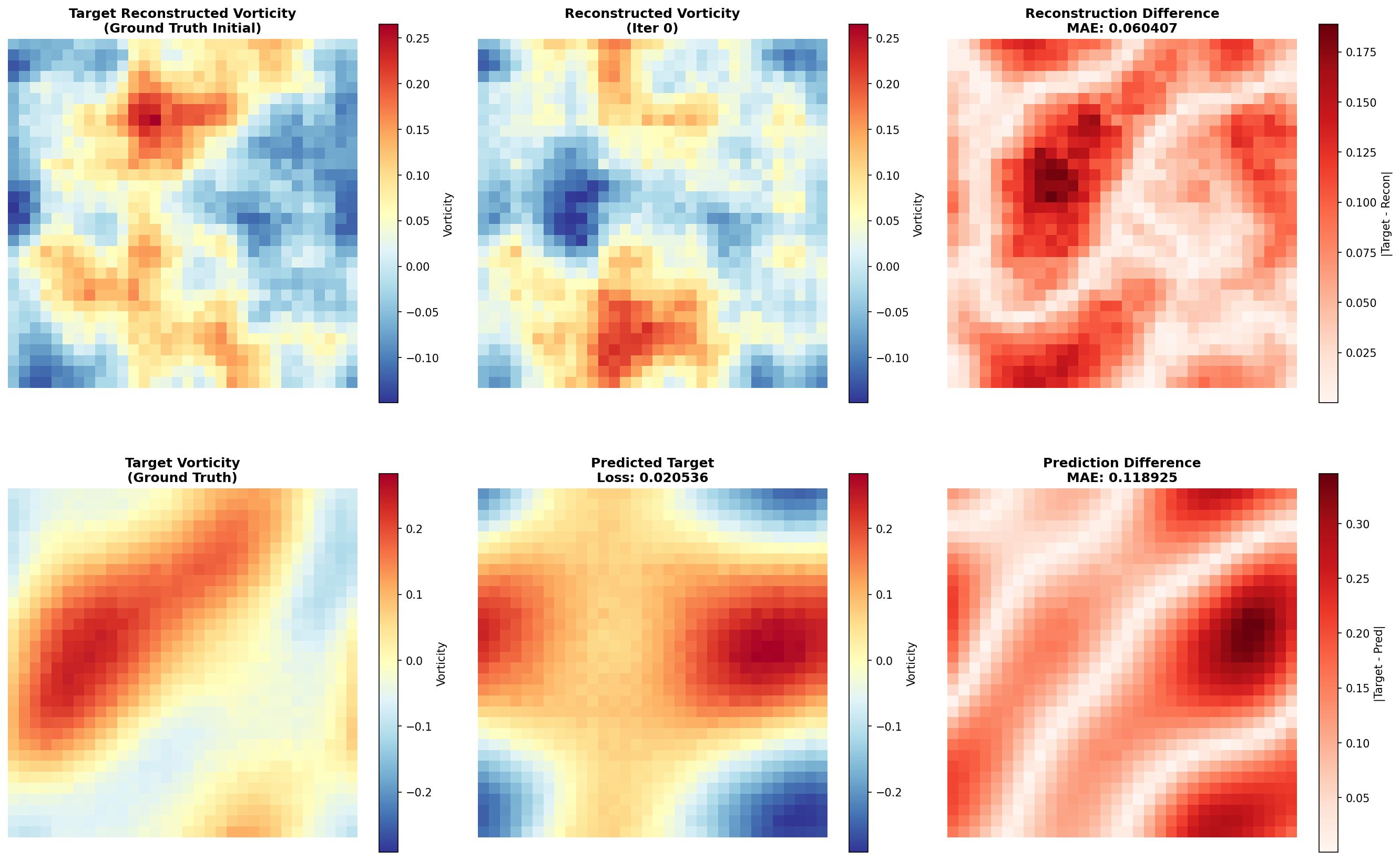}}
    \subfigure[ns2: Final (1990)]{\includegraphics[width=0.32\textwidth]{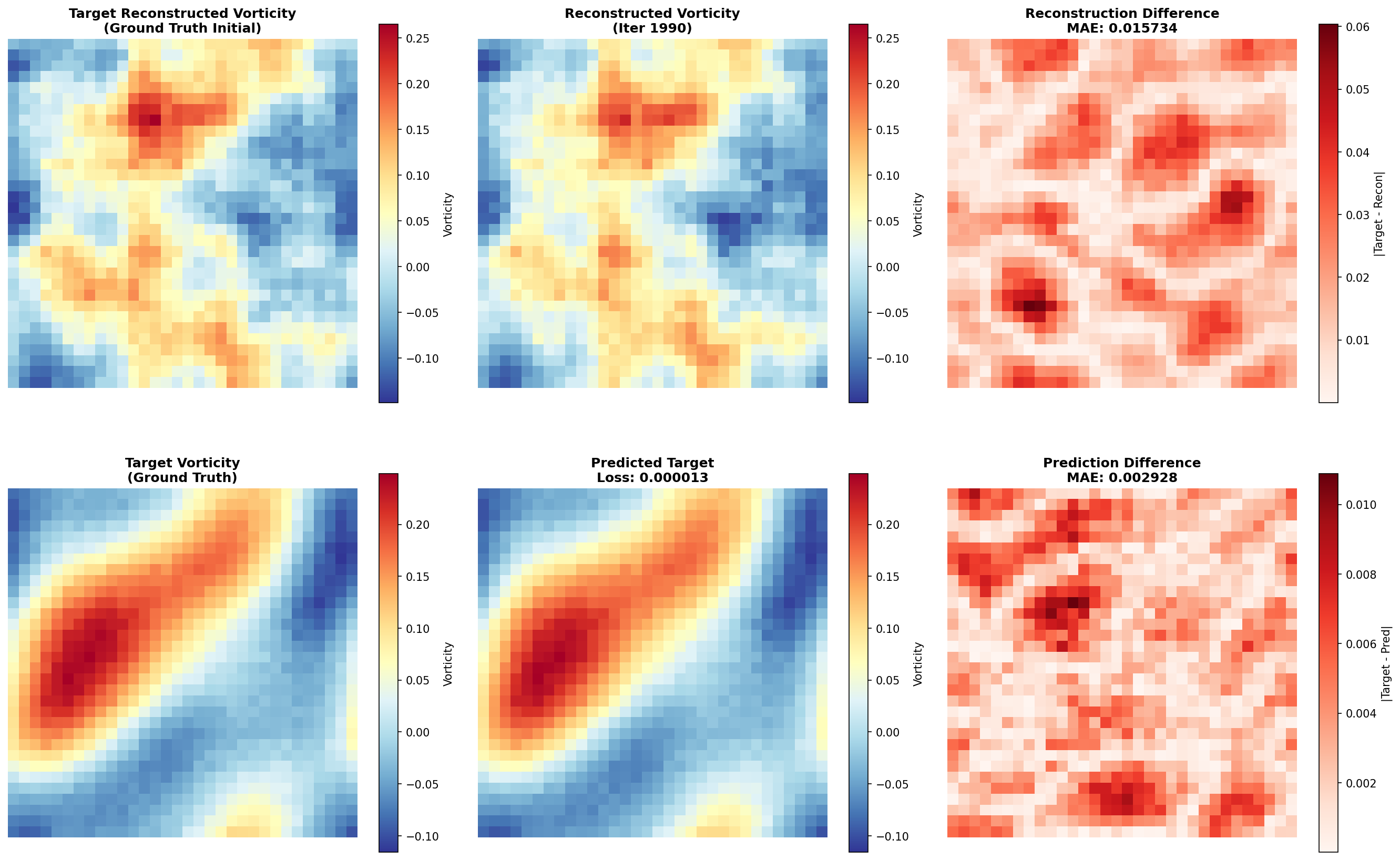}}
    \subfigure[ns2: Loss Curve]{\includegraphics[width=0.32\textwidth]{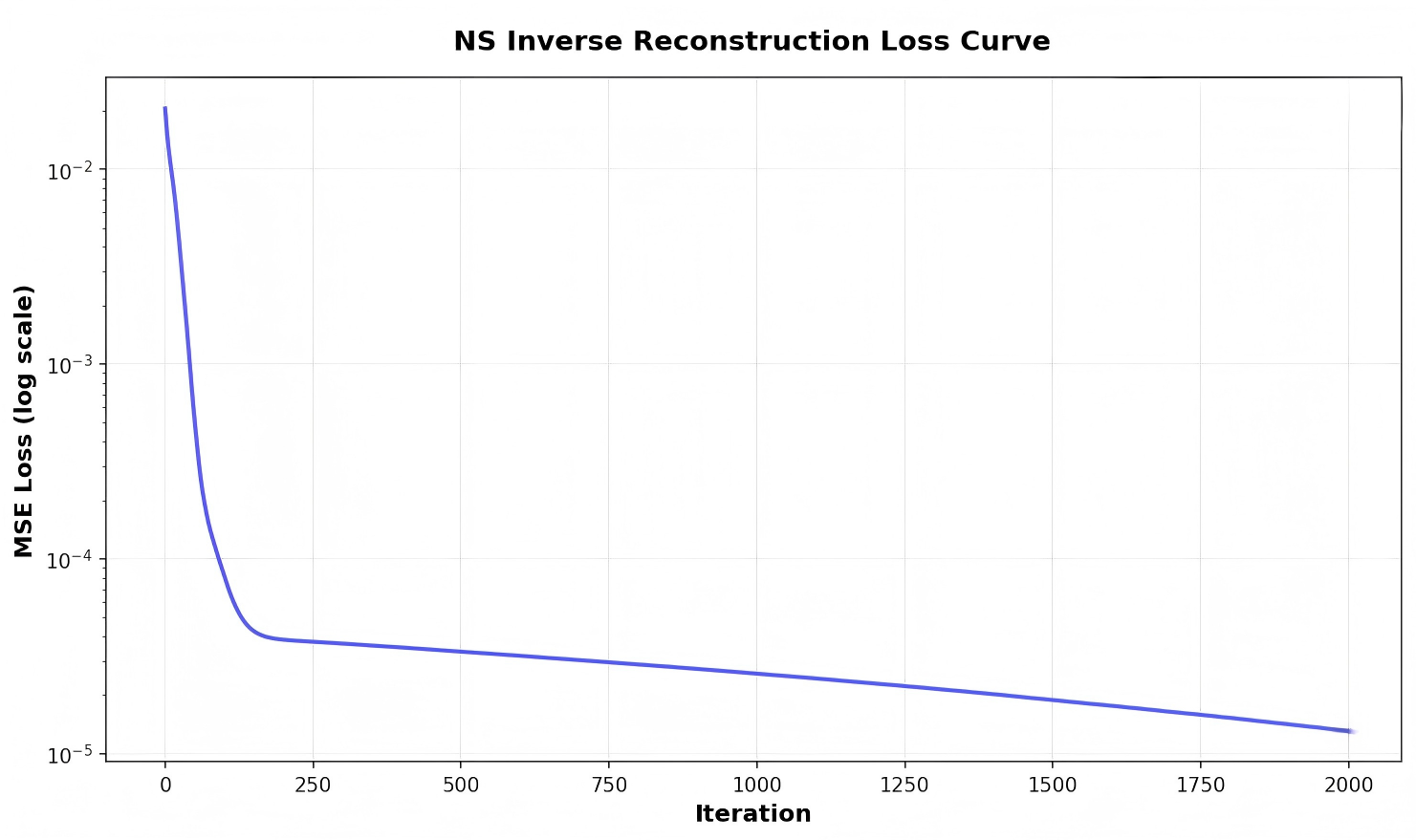}}
    \caption{Reconstruction analysis for a GRF-generated vorticity field characterized by a highly localized, dominant positive vortex amidst fragmented background flows, confirming the rapid matching of ground truth targets.}
    \label{fig:ns2_combined}
\end{figure}

\begin{figure}[htbp]
    \centering
    \subfigure[ns3: Iteration 0000]{\includegraphics[width=0.32\textwidth]{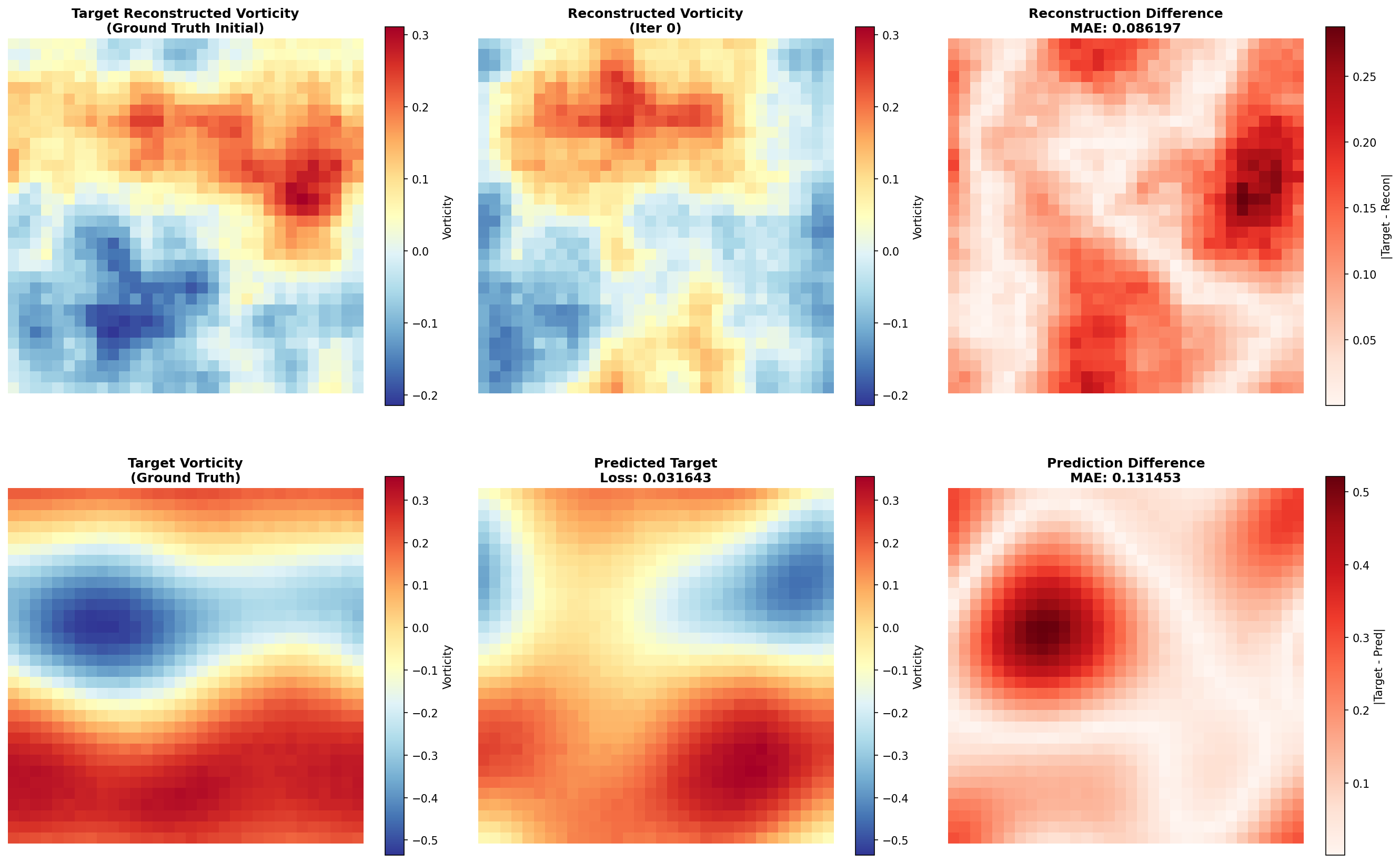}}
    \subfigure[ns3: Final (1990)]{\includegraphics[width=0.32\textwidth]{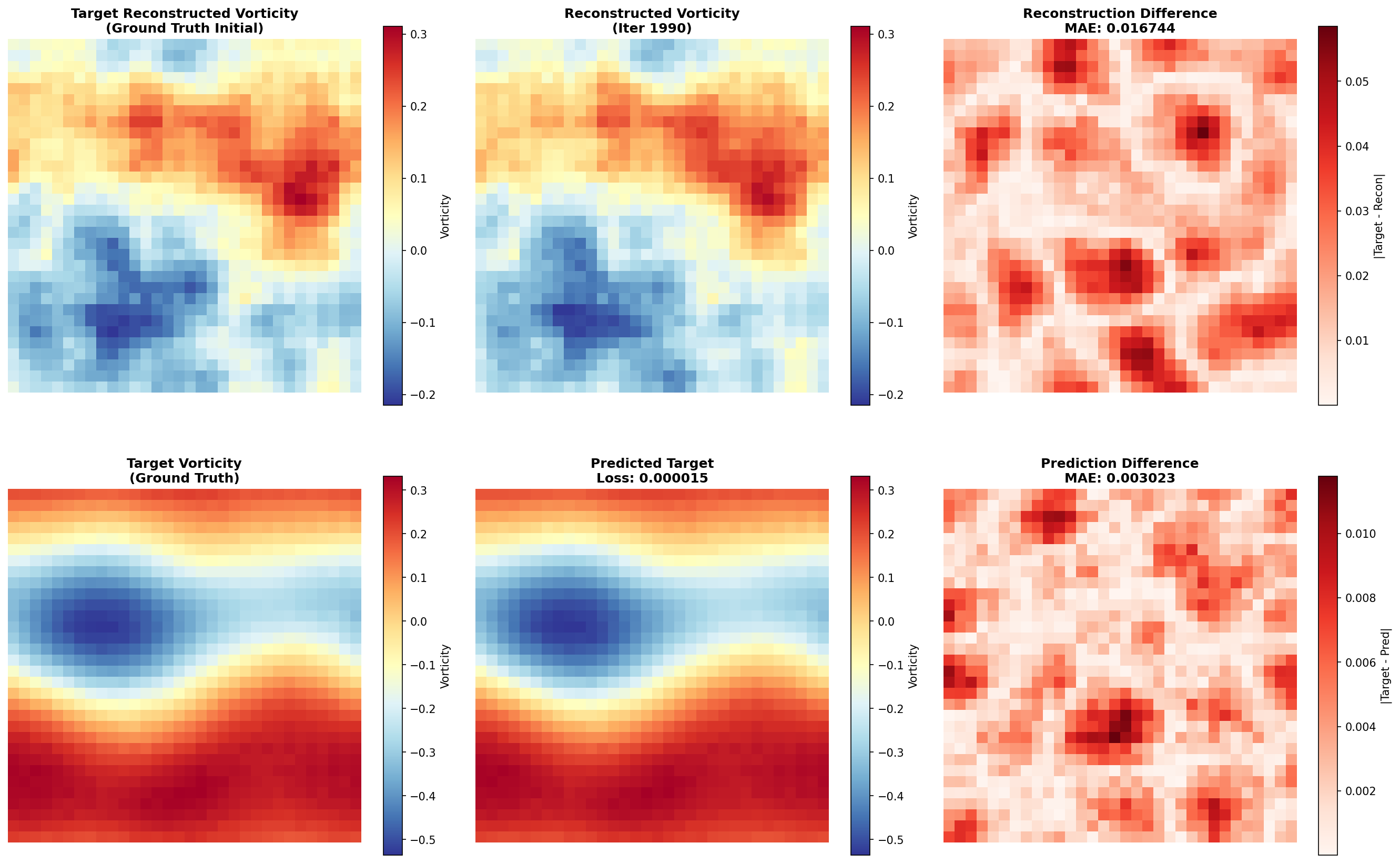}}
    \subfigure[ns3: Loss Curve]{\includegraphics[width=0.32\textwidth]{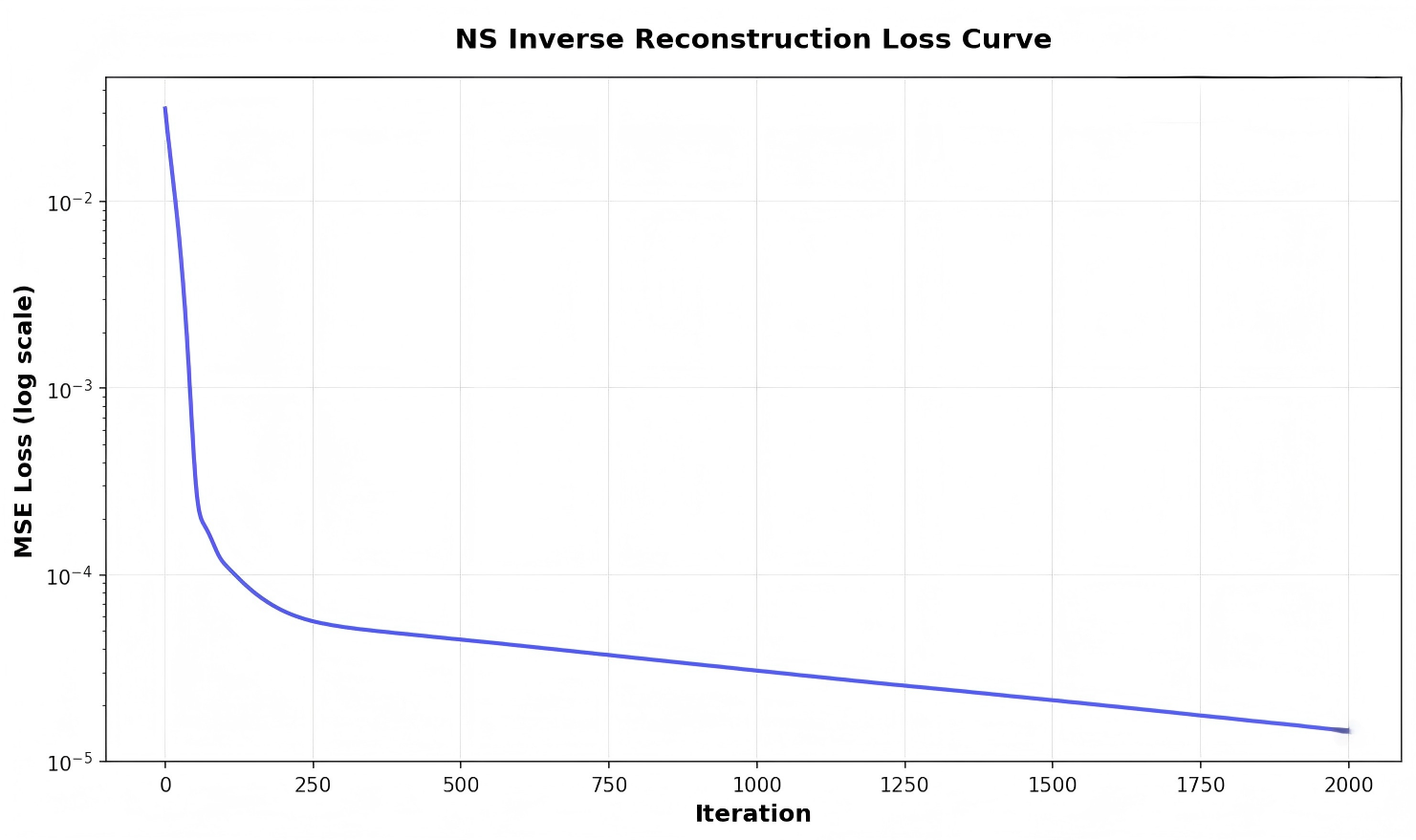}}
    \caption{Reconstruction analysis for a chaotic, multi-scale vorticity distribution sampled from a GRF. The results highlight the latent prior's effectiveness in resolving elongated fluid filaments and complex interactions.}
    \label{fig:ns3_combined}
\end{figure}

\begin{figure}[htbp]
    \centering
    \subfigure[ns4: Iteration 0000]{\includegraphics[width=0.32\textwidth]{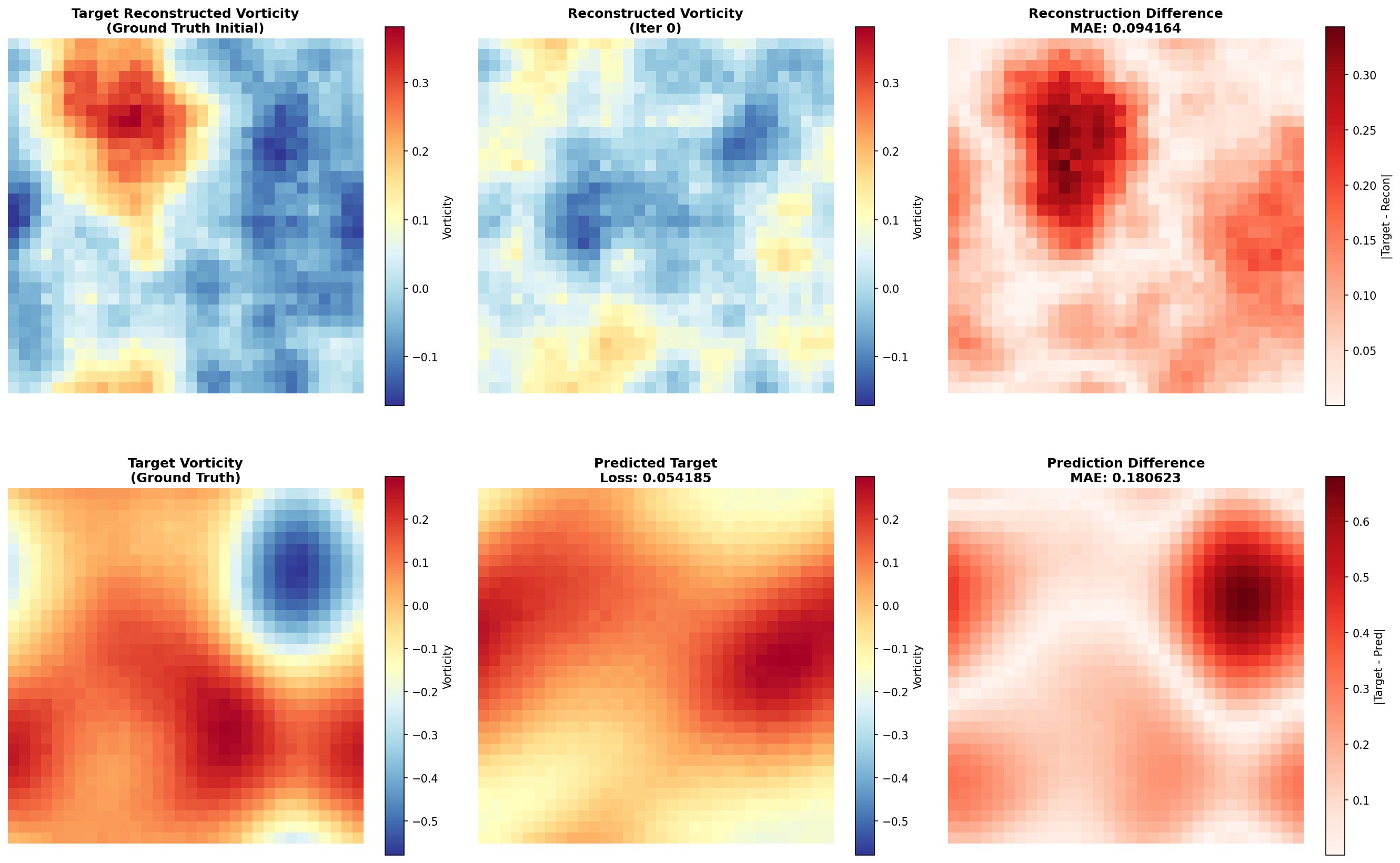}}
    \subfigure[ns4: Final (1990)]{\includegraphics[width=0.32\textwidth]{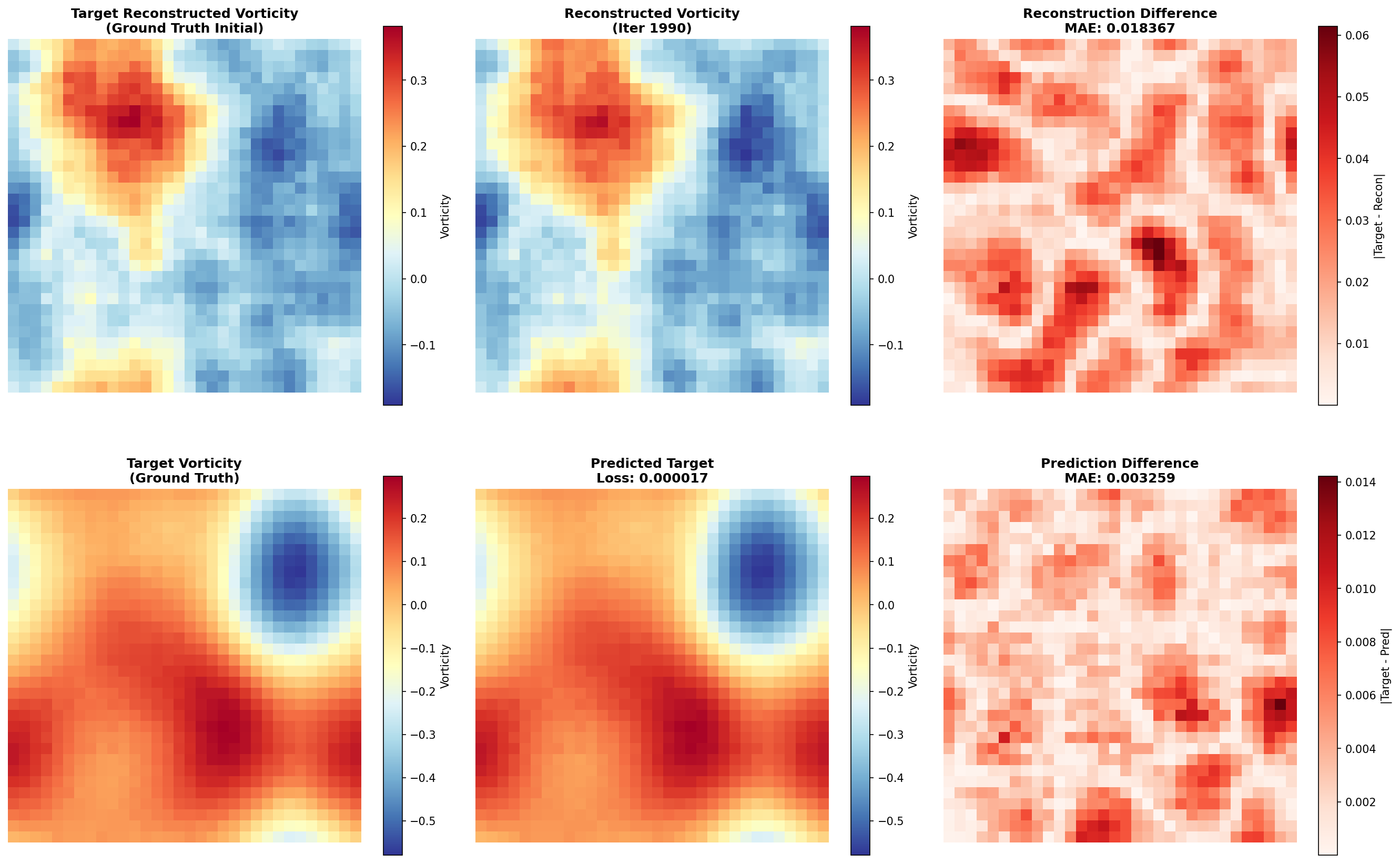}}
    \subfigure[ns4: Loss Curve]{\includegraphics[width=0.32\textwidth]{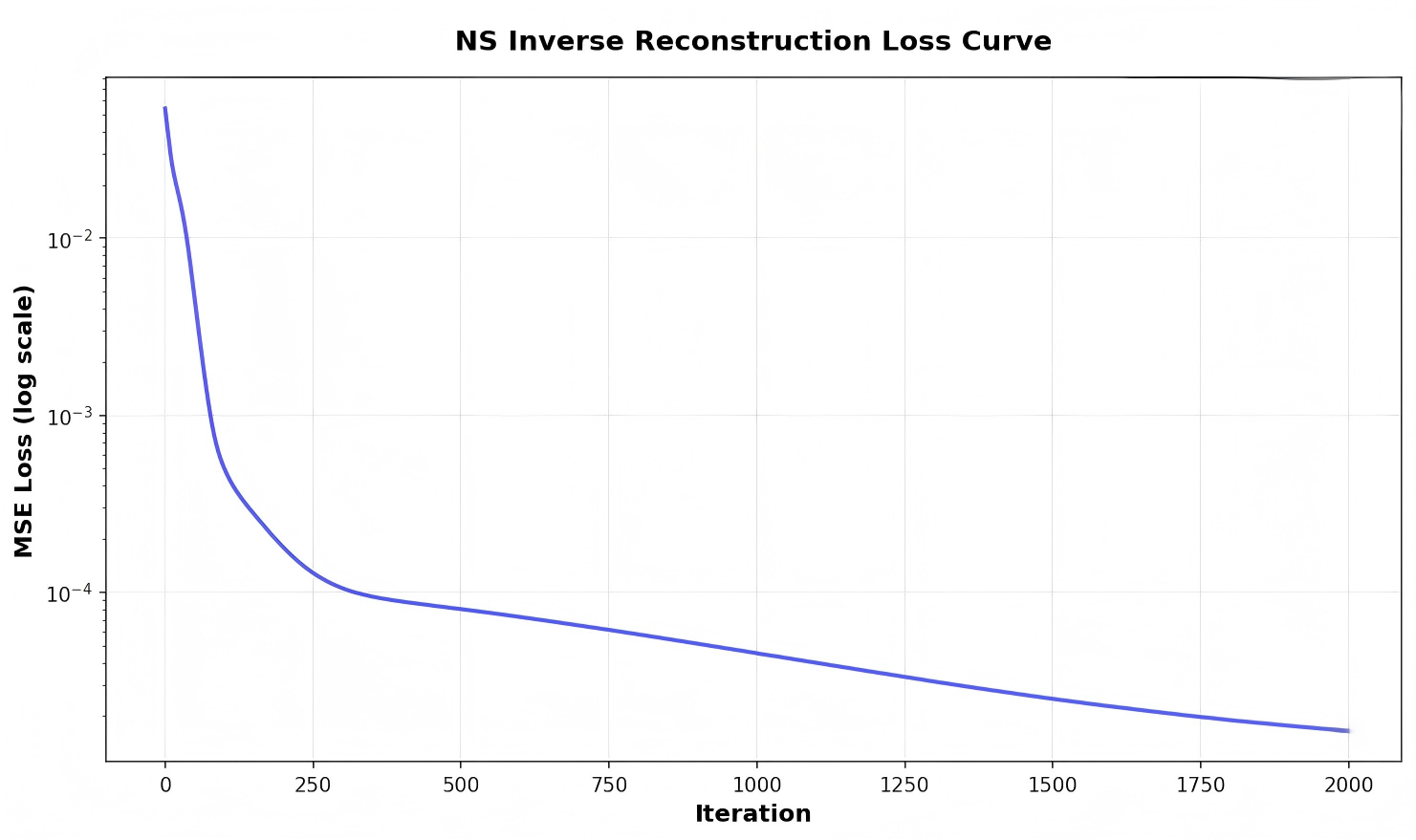}}
    \caption{Reconstruction analysis for a large-scale, coherent vortex structure originating from a GRF.}
    \label{fig:ns4_combined}
\end{figure}

\begin{figure}[htbp]
    \centering
    \subfigure[ns5: Noisy Iteration 0000]{\includegraphics[width=0.32\textwidth]{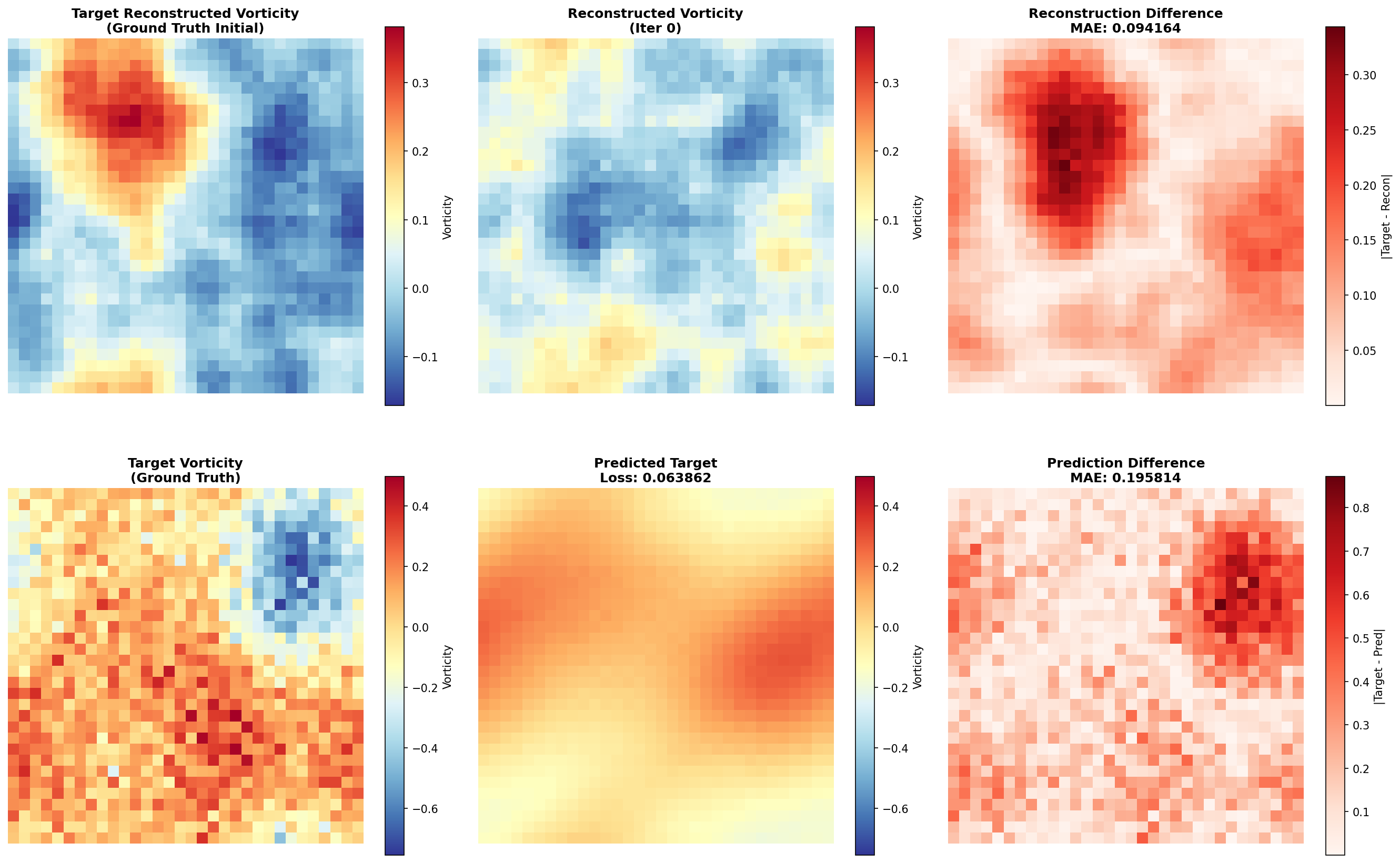}}
    \subfigure[ns5: Final Recovery (1990)]{\includegraphics[width=0.32\textwidth]{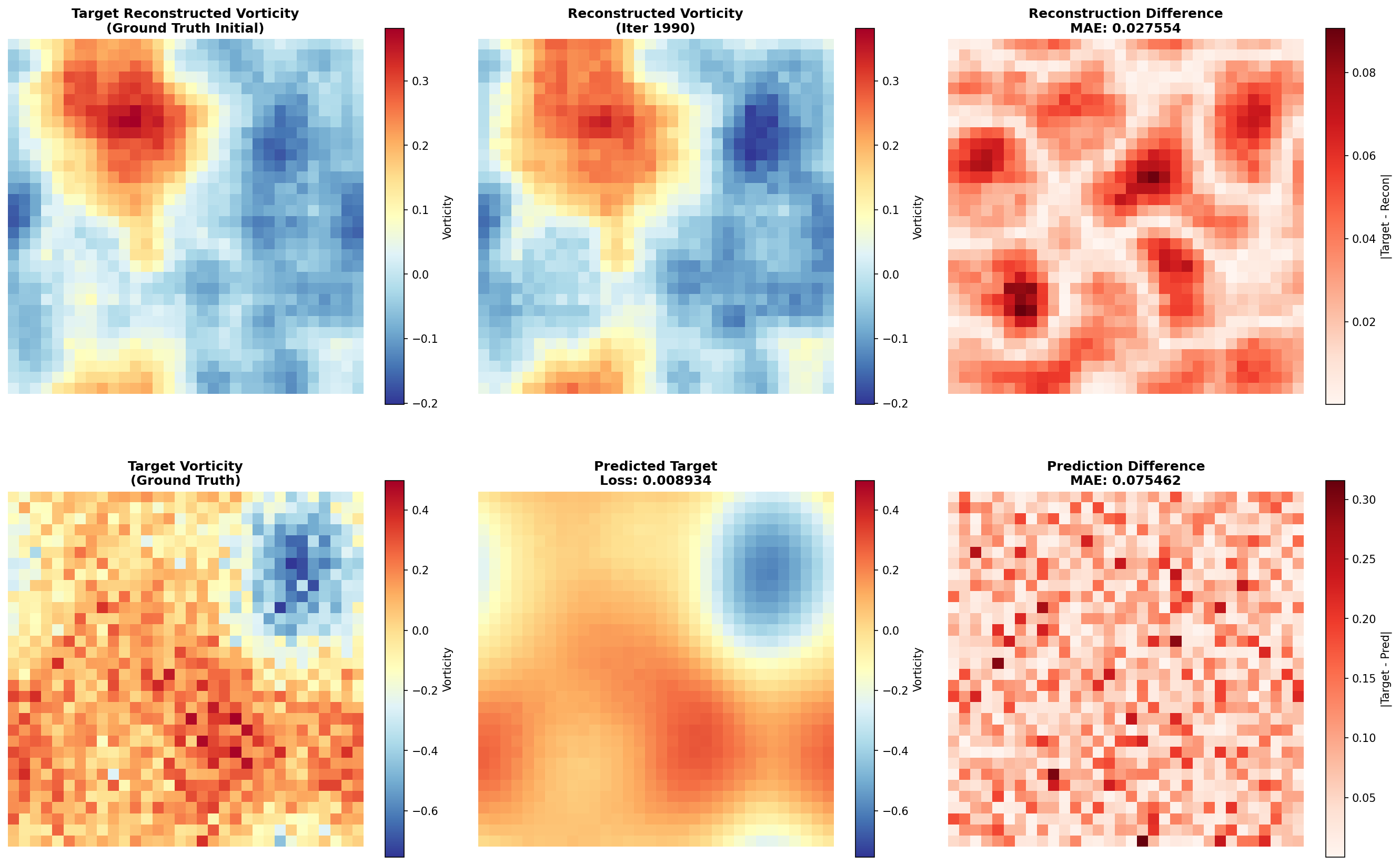}}
    \subfigure[ns5: Noisy Loss Curve]{\includegraphics[width=0.32\textwidth]{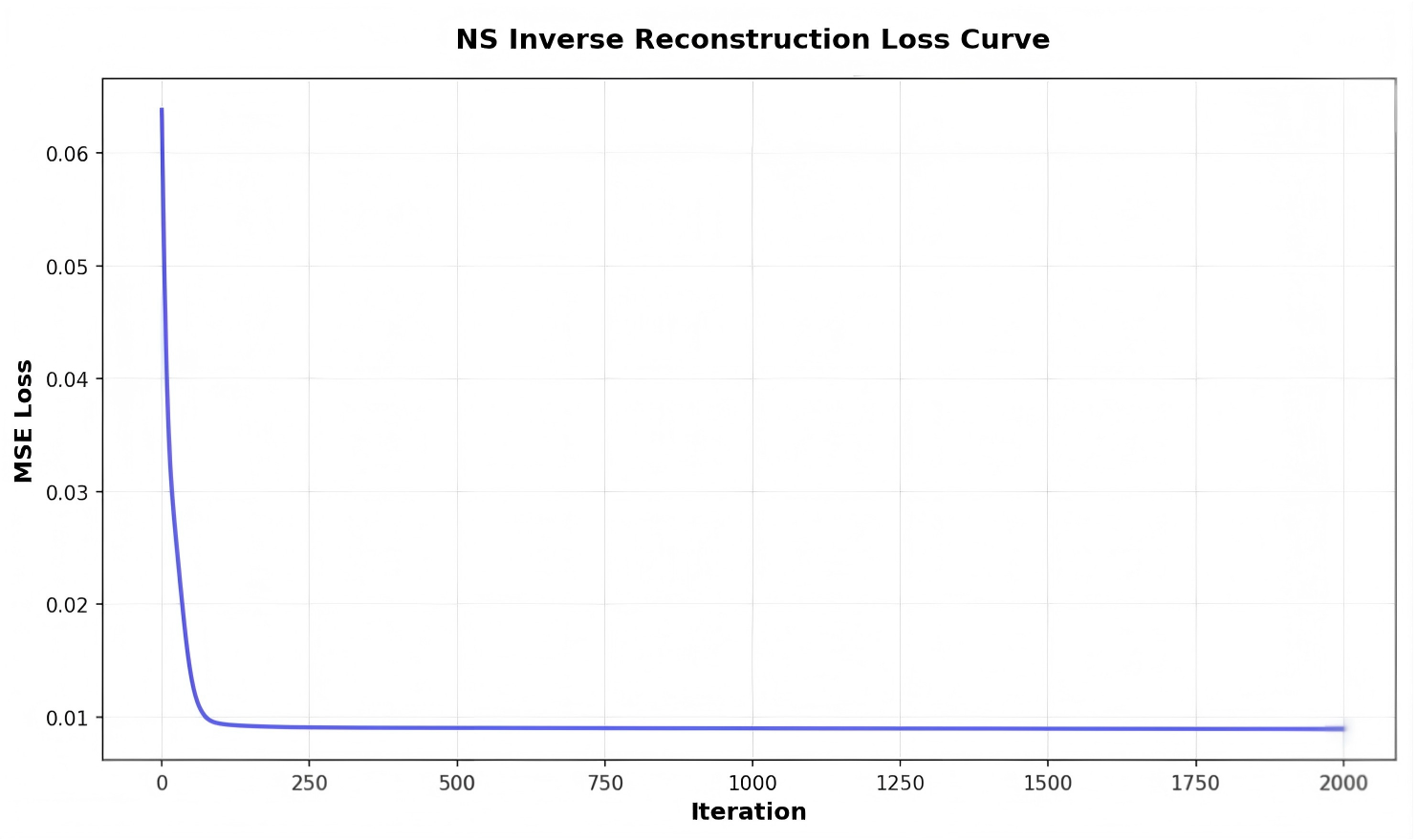}}
    \caption{Robustness analysis under 50\% observation noise.}
    \label{fig:ns5_noise}
\end{figure}


\section{Discussion and Conclusion}\label{Discussion}

In this study, we presented DiLO, a stable framework for solving PDE-constrained inverse problems by decoupling the generative prior from the forward physical constraints and optimizing through a deterministic diffusion trajectory. This design ensures that the differentiable surrogate is evaluated only on fully denoised parameter fields, avoiding the OOD issue of partially denoised states. Extensive experiments across EIT, inverse scattering, and inverse N-S demonstrate strong reconstruction accuracy and robustness to severe observation noise. The inverse N-S ablations further clarify the role of the main design choices. Optimization through the learned generative representation performs better than direct optimization in the ambient field, and the deterministic trajectory provides a reproducible measurement objective, whereas the stochastic variant shows little meaningful loss descent within the same optimization budget. In the spectral-truncation study, the tested moderate forward-model mismatch has a limited effect, whereas severe mismatch leads to a marked loss in reconstruction accuracy. This dependence on the forward surrogate remains an important limitation. Looking forward, we aim to extend DiLO to high-dimensional 3D inverse problems, incorporate multimodal structural priors, and explore more general physics foundation models as forward surrogates to improve generalization across complex physical systems.

\section{Acknowledgement}

Guang Lin would like to thank the support of National Science Foundation (DMS-2533878, DMS-2053746, DMS-2134209, ECCS-2328241, CBET-2347401 and OAC-2311848), and U.S.~Department of Energy (DOE) Office of Science Advanced Scientific Computing Research program under the "Uncertainty Quantification for Multifidelity Operator Learning (MOLUcQ)" project (Project No. 81739), DE-SC0023161, the SciDAC LEADS Institute, and DOE–Fusion Energy Science, under grant number: DE-SC0024583.

    \bibliographystyle{model1-num-names}
    \bibliography{references}
    
    \newpage
    \appendix
    
\section{Proof of Lemma \ref{lemma:hessian_decomp}}
\label{app:hessian_derivation}

\begin{proof}
    Let $\mathcal{L}_{\text{exact}}(z_T) = \mathcal{L}_{\text{phys}}(\mathcal{M}(z_T))$. The proof proceeds by evaluating the exact first- and second-order Fréchet derivatives of the composite functional through continuous variational analysis.

    \vspace{0.5em}
    \noindent \textit{Step 1: First-order variations and the adjoint state.} \\
    Throughout this appendix, the Neumann compatibility and zero-mean normalization conventions introduced in Section~\ref{sec:convergence} are understood for the forward, adjoint, tangent, and second-order adjoint states. Recall the physical objective $\mathcal{L}_{\text{phys}}(\sigma) = \frac{1}{2} \|u - V_{\text{obs}}\|_{L^2(\partial \Omega)}^2$, subject to the forward boundary value problem $\nabla \cdot (\sigma \nabla u) = 0$ in $\Omega$ with $\sigma \frac{\partial u}{\partial \mathbf{n}} = g$ on $\partial \Omega$.
    Introducing the first-order adjoint state $w \in H^1(\Omega)$, the standard adjoint state method yields the Lagrangian variation with respect to the state $u$, leading to the adjoint equation:
    \begin{equation}
        \nabla \cdot (\sigma \nabla w) = 0 \text{ in } \Omega, \quad \sigma \frac{\partial w}{\partial \mathbf{n}} = u - V_{\text{obs}} \text{ on } \partial \Omega.
    \end{equation}
    
    The Fréchet derivative of $\mathcal{L}_{\text{phys}}$ with respect to $\sigma$ is represented by the density $-\nabla u\cdot\nabla w\in L^1(\Omega)$. Using the $L^1$--$L^\infty$ duality pairing, the continuous first-order gradient with respect to $z_T\in\mathcal{Z}$ is
    \begin{equation} \label{eq:app_first_grad}
    \begin{aligned}
        \nabla_{z_T} \mathcal{L}_{\text{exact}}(z_T)
        &= D\mathcal{M}(z_T)^*\big(-\nabla u\cdot\nabla w\big),\\
        \big[\nabla_{z_T}\mathcal{L}_{\text{exact}}(z_T)\big]_j
        &= \int_\Omega \big(-\nabla u\cdot\nabla w\big)
        \frac{\partial\mathcal{M}(z_T)(\mathbf{x})}{\partial(z_T)_j}\,d\mathbf{x}.
    \end{aligned}
    \end{equation}

    \vspace{0.5em}
    \noindent \textit{Step 2: Second-order physical variations.} \\
    Differentiating \eqref{eq:app_first_grad} with respect to $z_T$ yields the exact composite Hessian $\mathbf{H}_{z_T} \in \mathbb{R}^{\dim(\mathcal{Z}) \times \dim(\mathcal{Z})}$:
    \begin{equation} \label{eq:app_hessian_expansion}
        \mathbf{H}_{z_T} = D\mathcal{M}(z_T)^* \mathbf{H}_{\text{phys}}(\sigma) D\mathcal{M}(z_T) + \int_\Omega \big(-\nabla u \cdot \nabla w\big)\nabla_{z_T}^2\mathcal{M}(z_T)(\mathbf{x})\,d\mathbf{x}.
    \end{equation}
    The physical Hessian operator $\mathbf{H}_{\text{phys}}(\sigma)$ is characterized by its action on a perturbation $\hat{\sigma}$ via the directional Fréchet derivative of the continuous gradient:
    \begin{equation} \label{eq:app_hessian_action}
        \mathbf{H}_{\text{phys}} \hat{\sigma} = \delta_{\hat{\sigma}} \big( -\nabla u \cdot \nabla w \big) = -\nabla \hat{u} \cdot \nabla w - \nabla u \cdot \nabla \hat{w},
    \end{equation}
    where the tangent state $\hat{u} \in H^1(\Omega)$ is the unique weak solution to the linearized forward equation:
    \begin{equation} \label{eq:app_tangent_pde}
    \begin{cases}
        \nabla \cdot (\sigma \nabla \hat{u}) = -\nabla \cdot (\hat{\sigma} \nabla u), \quad &\text{in } \Omega, \\
        \sigma \frac{\partial \hat{u}}{\partial \mathbf{n}} = -\hat{\sigma} \frac{\partial u}{\partial \mathbf{n}}, \quad &\text{on } \partial \Omega.
    \end{cases}
    \end{equation}
    By linearizing the first-order adjoint equation and accounting for the state-dependent boundary variation $\delta_{\hat{\sigma}} (u - V_{\text{obs}}) = \hat{u}$, the second-order adjoint state $\hat{w} \in H^1(\Omega)$ is determined as the solution to:
    \begin{equation} \label{eq:app_second_adjoint_pde}
    \begin{cases}
        \nabla \cdot (\sigma \nabla \hat{w}) = -\nabla \cdot (\hat{\sigma} \nabla w), \quad &\text{in } \Omega, \\
        \sigma \frac{\partial \hat{w}}{\partial \mathbf{n}} = \hat{u} - \hat{\sigma} \frac{\partial w}{\partial \mathbf{n}}, \quad &\text{on } \partial \Omega.
    \end{cases}
    \end{equation}
    These characterizations provide the explicit bilinear forms required for the spectral norm analysis.
    
    \vspace{0.5em}
    \noindent \textit{Step 3: Unrolling the generative curvature.} \\
    The generative mapping $\mathcal{M}(z_T) = \mathcal{D}(z_0)$ is formed by the composition of $T$ deterministic diffusion steps $z_{t-1} = c_t z_t + d_t \epsilon_\theta(z_t, t)$. The corresponding step-wise Jacobian and Hessian are given by:
    \begin{equation}
        J_t = c_t \mathbf{I} + d_t \nabla_{z_t} \epsilon_\theta(z_t, t), \qquad \mathcal{H}_t = d_t \nabla_{z_t}^2 \epsilon_\theta(z_t, t).
    \end{equation}
    By applying the multivariate chain rule for second derivatives to the composed mapping $\mathcal{M}(z_T)$, the tensor derivative distributes across the discrete sequence. After spatial discretization, let $i$ index the decoder output coordinates. Specifically, differentiating the component-wise gradient $\nabla_{z_T} \mathcal{M}_i(z_T) = \left( \prod_{s=1}^T J_s \right)^T \nabla_{z_0} \mathcal{D}_i(z_0)$ via the product rule exactly expands the generative curvature by propagating the local step-wise curvatures $\mathcal{H}_t$ and the spatial Hessian $\nabla_{z_0}^2 \mathcal{D}_i$ backward through the accumulated Jacobians:
    \begin{equation}
        \begin{aligned}
            \nabla_{z_T}^2 \mathcal{M}_i(z_T) &= \nabla_{z_T} \left( \left( \prod_{s=1}^T J_s \right)^T \nabla_{z_0} \mathcal{D}_i(z_0) \right) \\
            &= \left( \prod_{s=1}^T J_s \right)^T \Big[ \nabla_{z_0}^2 \mathcal{D}_i(z_0) \Big] \left( \prod_{s=1}^T J_s \right) \\
            &\quad + \sum_{t=1}^T \left( \prod_{s=t+1}^T J_s \right)^T \Bigg[ \sum_k \left( J_{\mathcal{D}}(z_0) \prod_{m=1}^{t-1} J_m \right)_{i,k} \mathcal{H}_{t,k} \Bigg] \left( \prod_{s=t+1}^T J_s \right).
        \end{aligned}
    \end{equation}
    Substituting this exact unrolled tensor alongside the physical Hessian action into \eqref{eq:app_hessian_expansion} establishes the analytical decomposition.
\end{proof}
    
\end{document}